\documentclass[a4paper,english,openright,11pt]{smfartVSautomodl}

\usepackage{smfenum,smfthm,amsmath,amssymb,amscd,mathrsfs,euscript,color}
\usepackage{stmaryrd,mathabx,upgreek}
\usepackage[latin1]{inputenc} 
\input{xypic}

\numberwithin{equation}{section} 
\theoremstyle{plain}

\newcounter{nonum}

\def\AA{\mathbb{A}}
\def\CC{\mathbb{C}}

\def\QQ{\mathbb{Q}}
\def\RR{\mathbb{R}}
\def\ZZ{\mathbb{Z}} 
\def\FF{\mathbb{F}} 
\def\GG{{\bf G}}
\def\TT{{\bf T}}
\def\MM{{\bf S}}

\def\B{{\rm B}}

\def\D{{\rm D}}
\def\E{{\rm E}}
\def\F{{F}}
\def\G{{G}}
\def\H{{H}}
\def\I{{I}}
\def\J{{\rm J}}
\def\K{{\rm K}}
\def\L{{\rm L}}

\def\N{{\rm N}}
\def\P{{P}}

\def\R{{\rm R}}
\def\SS{{\rm S}}
\def\T{{\rm T}}
\def\U{{\rm U}}
\def\V{{\rm V}}
\def\W{{\rm W}}
\def\X{{\rm X}}
\def\Y{\boldsymbol{\sf Y}}
\def\Z{{Z}}

\def\Aa{\mathscr{A}}

\def\Cc{\EuScript{C}}

\def\Ee{\EuScript{E}}

\def\Hh{\EuScript{H}}

\def\Mm{\boldsymbol{\sf V}}

\def\Oo{\EuScript{O}}

\def\Ga{\Gamma}
\def\La{\Lambda}
\def\Li{{\it\La}}

\def\a{\alpha} 
\def\b{\beta}
\def\d{\delta}
\def\e{\varepsilon}

\def\g{\gamma}
\def\h{\varphi}
\def\k{k}
\def\l{\lambda}
\def\m{\mathfrak{m}}
\def\n{\eta}

\def\p{\mathfrak{p}}
\def\s{\sigma}
\def\t{\theta}
\def\v{\upsilon}
\def\w{\varpi}

\def\>{\geqslant}
\def\<{\leqslant}

\def\Hom{{\rm Hom}}
\def\End{{\rm End}}
\def\Aut{{\rm Aut}}
\def\Mat{\boldsymbol{{\sf M}}}
\def\GL{{\rm GL}}
\def\SL{{\rm SL}}
\def\SO{{\rm SO}}
\def\Sp{{\rm Sp}}
\def\Gal{{\rm Gal}}
\def\Ker{{\rm Ker}}

\def\Ind{{\rm Ind}}

\def\qpb{{\overline{\mathbb{Q}}_p}}
\def\qlb{{\overline{\mathbb{Q}}_\ell}}
\def\zlb{{\overline{\mathbb{Z}}_\ell}}
\def\flb{{\overline{\mathbb{F}}_{\ell}}}

\def\ip{\boldsymbol{i}}

\def\r{{\textbf{\textsf{r}}}}
\def\j{{\textbf{\textsf{j}}}}

\def\kk{\boldsymbol{k}}

\def\Sp{{\rm Sp}}

\def\ii{\iota}

\def\Alg{{\rm Alg}}

\def\Std{{\rm Std}}

\def\BC{\boldsymbol{\sf t}}
\def\CB{\boldsymbol{\sf b}}
\def\rec{\boldsymbol{\sf a}}
\def\res{\boldsymbol{\sf res}}
\def\WD{{\rm WD}}
\def\GGH{\widehat{{\bf G}}}
\def\GGL{{}^L\GG}
\def\LLC{{\rm rec}}
\def\Frob{\Phi}
\def\fss{*}
\def\zz{{\bf Z}}

\long\def\MSC#1\EndMSC{\def\arg{#1}\ifx\arg\empty\relax\else
     {\par\narrower\noindent%
     2010 Mathematics Subject Classification: #1\par}\fi}

\long\def\KEY#1\EndKEY{\def\arg{#1}\ifx\arg\empty\relax\else
	{\par\narrower\noindent Keywords and Phrases: #1\par}\fi}

\title{Local transfer for quasi-split classical groups
and congruences mod $\ell$}

\author{Alberto M\'{\i}nguez}
\address{
University of Vienna, 
Fakult{\"a}t f{\"u}r Mathematik,
Oskar-Morgenstern-Platz 1,
1090 Wien}
\email{alberto.minguez@univie.ac.at}

\author{Vincent S\'echerre}
\address{
Laboratoire de Math\'emati\-ques de Versailles, 
UVSQ, 
CNRS, 
Universit\'e Paris-Saclay,
78035 Versailles, France,
Institut Universitaire de France}
\email{vincent.secherre@uvsq.fr}

\begin{abstract}
Let $G$ be the group of rational points of a quasi-split $p$-adic special 
orthogonal,~sym\-plectic or unitary group for some odd  prime number $p$.
Following Arthur and Mok, there are~an~in\-te\-ger $N\>1$,
a $p$-adic field $E$ and a local functorial transfer from isomorphism
classes of irreducible smooth complex representations of $G$ to
those of $\GL_N(E)$.
By fixing a prime number $\ell$ different from $p$ and an isomorphism between 
the field of complex numbers and an algebraic closure of the~field~of 
$\ell$-adic numbers, we obtain a transfer map between representations with 
$\ell$-adic coefficients.
Now~con\-si\-der a cuspidal irreducible $\ell$-adic representation $\pi$ of 
$G$: 
we can define its reduction mod $\ell$,
which is a~semi-simple smooth representation of $G$ of finite length,
with coefficients in
a field of~characteris\-tic~$\ell$.
Let $\pi'$ be a cuspidal irreducible $\ell$-adic representation of $G$
whose reduction mod~$\ell$ is isomor\-phic~to that of $\pi$.
We prove that the transfers~of $\pi$ and $\pi'$ have reductions
mod $\ell$~which may~not be isomorphic,
but which have isomorphic supercuspidal supports. 
When $G$ is not the~split special~or\-tho\-gonal group $\SO_2$,
we~further prove that the  reductions
mod $\ell$ of the transfers~of $\pi$ and $\pi'$ 
share~a unique common generic component. 
\end{abstract}

\begin{document} 

\maketitle

\MSC 22E50, 11F70
\EndMSC
\KEY 
Automorphic representation,
Classical group,
Congruences mod $\ell$,
Cuspidal representation,
Functorial transfer
\EndKEY

\thispagestyle{empty}

\section{Introduction}

\subsection{}
\label{fernandduluc}

Let $F$ be a $p$-adic field for some odd prime number $p$
and $G$ be the group of rational~points of~a~quasi-split
special orthogonal, unitary or symplectic group defined over $F$.
In the case~where $G$ is unitary,
let $E$ be the quadratic extension of $F$ with respect to which 
$G$ is defined~;
other\-wi\-se, 
let~$E$ be equal to $F$.
According to Arthur \cite{Arthur} for special orthogonal and 
symplectic groups, 
and to Mok \cite{Mok} for unitary groups,
there is a positive integer $N=N(G)$ and a 
map~from~isomor\-phism classes of irreducible (smooth)
complex representations
of $G$
to those of the general~li\-near group $\GL_N(E)$,
called the local \textit{transfer} or \textit{base change},
which we will denote by $\BC$. 

\subsection{}
\label{sommerive}
  
Let us fix a prime number $\ell$ different from $p$
and an isomorphism of fields $\ii$ between~$\CC$~and~an
algebraic closure $\qlb$ of the field of $\ell$-adic numbers.
Replacing formally $\CC$ by $\qlb$ thanks~to~$\ii$,~we 
get a local transfer 
between isomorphism classes of
irreducible smooth $\qlb$-re\-pre\-sen\-ta\-tions,
deno\-ted $\BC_{\ell}$. 
(We describe the dependency of $\BC_{\ell}$ in the choice of $\ii$
-- or equivalently the behavior of~$\BC$ with respect to automorphisms
of $\CC$:
see Paragraph \ref{appaAforNRrep} for unramified representations, 
and~Pa\-ra\-graph~\ref{appA} for
{discrete series}
representations, 
of $G$.)
  
We can now consider irreducible 
$\qlb$-representations which are \textit{integral}
-- that is, which carry a stable $\zlb$-lattice,
where $\zlb$ denotes the ring of integers of $\qlb$.
Given~such a representation $\pi$, 
one can define its reduction mod $\ell$:
this is the semi-simplification of the~re\-duc\-tion of any of its stable 
$\zlb$-lattices modulo the maximal ideal of $\zlb$.
This is a smooth~repre\-sentation of finite~length
with coefficients in $\flb$,
the residue field of $\zlb$,~de\-no\-ted $\r_\ell(\pi)$.
One then can ask whether the map $\BC_{\ell}$ preserves the fact of
being integral,
and how it behaves with respect to congruences mod $\ell$.

Similar questions have already been answered for other
local correspondences:
see \cite{Vigl,Datl,BHl}
for the local Langlands correspondence for $\GL_n$,
as well as \cite{Datjl,MSjl}
for~the local Jacquet-Lang\-lands correspondence between inner~forms of $\GL_n$,
for $n\>1$.
(See also Paragraph \ref{introCLBC} below~and Ap\-pen\-dix~\ref{appBC},
where we discuss the case of
the cyclic local base change for $\GL_n$.)
In this paper, we pro\-ve~the following theorem. 

\begin{theo}
\label{MAINTHEOINTRO}
Let $\pi_1$, $\pi_2$ be integral cuspidal irreducible $\qlb$-representations 
of $G$, and assu\-me that
\begin{equation}
\label{condpipi}
\r_\ell(\pi_1) \< \r_\ell(\pi_2)
\end{equation}
that is, $\r_\ell(\pi_1)$ is contained in $\r_\ell(\pi_2)$ as semi-simple 
$\flb$-representations of $G$. 
Then
\begin{enumerate}
\item 
The local transfer $\BC_{\ell}(\pi_i)$ is an integral $\qlb$-representation of 
$\GL_N(E)$ for each $i=1,2$.
\item
The irredu\-ci\-ble components of the semi-simple
$\flb$-representation 
$\r_\ell(\BC_{\ell}(\pi_1)) \oplus \r_\ell(\BC_{\ell}(\pi_2))$ 
all have the same supercuspidal support
(see below for a definition).
\item
Assume that $G$ is not isomorphic to the split special orthogonal 
group $\SO_2(F)\simeq F^\times$.~The semi-simple
$\flb$-representations $\r_\ell(\BC_{\ell}(\pi_1))$ and
$\r_\ell(\BC_{\ell}(\pi_2))$ have a unique generic irreducible~compo\-nent 
in common. 
\end{enumerate} 
\end{theo}

As in the case of complex coefficients, 
an irreducible representation of $\GL_N(E)$ on an $\flb$-vector space $V$
is said to be \textit{generic} if $V$ carries a non-zero $\flb$-linear form 
$\La$ such that $\La(\pi(u)v)=\t(u)v$
for all $v\in V$ and all upper triangular unipotent
matrices $u$ of $\GL_N(E)$,
where $\t$ is the $\flb$-character 
\begin{equation*}
u \mapsto \psi(u_{1,2}+\dots+u_{n-1,n})
\end{equation*}
and $\psi$ is a non-trivial $\flb$-character of $F$.

An irreducible $\flb$-representation of $\GL_N(E)$ is 
\textit{supercuspidal} if it does not occur as a subquo\-tient of any 
representation parabolically induced from a proper Levi subgroup.
A \textit{supercuspi\-dal support} of an irreducible
$\flb$-representation~$\pi$~of
$\GL_N(E)$ is a pair $(M,\rho)$ made of a Levi sub\-group of 
$\GL_N(E)$ and a supercuspidal representa\-tion $\rho$ of $M$ such that $\pi$
occurs as a subquotient of the normalized parabolic induction of $\rho$.
It is uniquely determined up to conjugacy
(\cite{Vigs} V.4, \cite{MSc} Th\'eo\-r\`e\-me 8.16).

Note that,
unless $G$ is the~split~spe\-cial~or\-tho\-gonal group 
$\SO_2(F)\simeq F^\times$,
the centre of $G$ is~com\-pact. 
When this is the case, 
any cuspidal irreducible $\qlb$-re\-pre\-sen\-tation of $G$ 
is~in\-te\-gral.~We will discuss the case of the split $\SO_2(F)$ 
in detail in Paragraph \ref{splitSO2}.

Also note that,
if $G$ is not isomorphic to $\SO_2(F)\simeq F^\times$,
then (3) implies (2),
since all~irredu\-cible components of the reduction mod $\ell$ of an integral
irreducible $\qlb$-representation of $\GL_N(E)$
have the sa\-me supercuspidal support \cite{Vigl} \S 1.3. 

Before discussing the other assumptions of Theorem \ref{MAINTHEOINTRO}
(in Paragraph \ref{discussthm11}),
let us explain~how we prove it.~The
general~strategy goes back to Khare \cite{Khare} and Vign\'eras 
\cite{Vigl} who study the~con\-gruence properties of the local
Langlands correspondence for $\GL_n(F)$ with $n\>1$. 

\subsection{}
\label{par13intro}

The first step is to pass from our given local situation to the following global 
situation (which is the purpose of Sections \ref{global1} to \ref{global3}).

First, 
$k$ is a totally real number field,
$l$ is either $k$ or a totally imaginary quadratic extension~of $k$
and $w$ is a finite place of $k$ above $p$, 
inert in $l$, 
such that $k_w=F$ and $l_w=E$
(see \textsection 2.4). 

Next, 
$\GG$ is a connected reductive group defined over $k$ such that 
\begin{enumerate}
\item
the group $\GG(F)$ naturally identifies with $G$,
\item 
the group $\GG(k_v)$ is compact for any real place $v$
and quasi-split for any finite place~$v$, 
\item 
the $k$-group $\GG$ is an inner form of a quasi-split special orthogonal,
unitary or symplectic group $\GG^*$.
\end{enumerate}
The existence of such a group $\GG$ is proved in Section \ref{global1}
(see Theorem \ref{GLOBALG}).

Finally, $\Pi_1$ and $\Pi_2$ are 
irreducible~auto\-morphic~re\-pre\-sentations of $\GG(\AA_k)$,
where $\AA_\k$ denotes the ring of ad\`eles of $k$, such that
\begin{enumerate}
\item 
$\Pi_{1,w}\otimes_{\CC}\qlb$ is isomorphic to $\pi_1$
and $\Pi_{2,w}\otimes_{\CC}\qlb$ is isomorphic to $\pi_2$,
\item
the representations
$\Pi_{1,v}$ and $\Pi_{2,v}$ are trivial for any real place $v$,
\item 
for a given 
finite place $u\neq w$ of $k$,
the representations
$\Pi_{1,u}$ and $\Pi_{2,u}$ are both isomorphic to some~cus\-pidal 
irreducible unitary representation $\rho$ of $\GG(k_u)$ which is 
compactly induced from a compact mod centre,
open subgroup of $\GG(k_u)$,
\item
there is a finite set $\SS$ of places of $k$, containing all real places, 
such that for all $v\notin\SS$~:
\begin{enumerate}
\item 
the group $\GG$ is unramified over $k_v$,
\item
the representations $\Pi_{1,v}\otimes_{\CC}\qlb$ and
$\Pi_{2,v}\otimes_{\CC}\qlb$ are un\-ra\-mified 
with respect to some hyperspecial
ma\-xi\-mal compact subgroup of $\GG(k_v)$,
\item
their Satake parameters (in the sense of Paragraph \ref{sub:conditions}) 
are integral and~con\-gruent mod the maximal ideal of $\zlb$, 
\end{enumerate}
\end{enumerate}
where all tensor products are taken with respect to $\ii$.
We construct such $\Pi_1$ and $\Pi_2$ in Sections \ref{global2} and \ref{global3}.

\subsection{}
\label{bidib4}

The next step
-- which is the purpose of Section \ref{transfer5} --
is to associate to $\Pi_1$ and $\Pi_2$ two
cuspidal irreducible automorphic representa\-tions 
$\widetilde{\Pi}_1$ and $\widetilde{\Pi}_2$ of $\GL_N(\AA_l)$ such that,
for any finite place $v$,
the local transfer~of~$\Pi_{i,v}$ is~isomor\-phic to $\widetilde{\Pi}_{i,v}$,
for $i=1,2$.
For this,
we use the results of Taïbi \cite{Taibi}
if $\GG^*$ is symplectic or~spe\-cial or\-thogonal, 
and Labesse \cite{Labesse} if $\GG^*$ is unitary.
Namely, let $\widetilde{\Pi}_i$ be
\begin{itemize}
\item 
the Arthur parameter associated 
with $\Pi_i$ if $\GG^*$ is symplectic or special othogonal (\cite{Taibi}), 
\item
the stable base change of $\Pi_i$ to $\GL_N(\AA_l)$
if $\GG^*$ is unitary (\cite{Labesse}).
\end{itemize}
In both cases,
$\widetilde{\Pi}_i$ is algebraic regular and \cite{Taibi, Labesse}
provide $\widetilde{\Pi}_i$ with certain 
local-global compati\-bi\-lities at all finite places. 
In order to ensure that these local-global compatibilities
are what~we want,~na\-mely,
that the~local transfer~of $\Pi_{i,v}$ is isomorphic to $\widetilde{\Pi}_{i,v}$
at all finite $v$,
and in prevision of the next step,
we need $\widetilde{\Pi}_i$ to be cuspidal.

In order to choose $\Pi_1$, $\Pi_2$ so that 
$\widetilde{\Pi}_1$, $\widetilde{\Pi}_2$ are cuspidal, 
we use the~au\-xi\-liary cuspidal~re\-pre\-sen\-tation $\rho$ of
Paragraph \ref{par13intro}.
More precisely,
we prove the following result
(see Lemma \ref{placeu}).

\begin{prop}
Given $k$, $w$ and $\GG$ as in Paragraph \ref{par13intro},
the finite place $u$ of $k$
and the~re\-pre\-sen\-tation $\rho$ of $\GG(k_u)$
can be chosen so that the local transfer of $\rho$ is cuspidal. 
\end{prop}

If $\GG^*$ is unitary,
it suffices to choose $u$ so that $\GG$ is split over $k_u$.
In the symplectic and special orthogonal cases,
this is the purpose of Appendices \ref{app1} and \ref{app2}.
(In particular,
the place $u$ has to divide $2$ in the symplectic case.) 

\subsection{}

We now have two algebraic regular,
cuspidal irreducible automorphic representations 
$\widetilde{\Pi}_1$ and $\widetilde{\Pi}_2$ of $\GL_N(\AA_l)$ such that,
for $i=1,2$ and all finite places $v$,
the transfer of $\Pi_{i,v}$~is~iso\-morphic~to $\widetilde{\Pi}_{i,v}$.
Besides, 
it follows from the properties of the transfer from $\GG(\AA_k)$
to $\GL_N(\AA_l)$ that the~con\-ju\-ga\-te of the contragredient of
$\widetilde{\Pi}_i$
by the generator $c$ of $\Gal(l/k)$ is isomorphic to $\widetilde{\Pi}_i$. 

From the properties of $\Pi_1$ and $\Pi_2$ at all places $v\notin\SS$,
and from the congruence properties of the unramified local transfer that 
we establish in Section \ref{ULT}, it follows that, for all $v\notin\SS$:
\begin{enumerate}
\item
the local components $\widetilde{\Pi}_{1,v}$ and $\widetilde{\Pi}_{2,v}$ are unramified,
\item
the Satake parameters of $\widetilde{\Pi}_{1,v}\otimes_{\CC}\qlb$ and
$\widetilde{\Pi}_{2,v}\otimes_{\CC}\qlb$
are integral and congruent mod the maximal ideal of $\zlb$.
\end{enumerate} 

We now apply the results of \cite{BLGHT},
which give us two continuous $\ell$-adic Galois representations
\begin{equation*}
\Sigma_i : \Gal(\overline{\QQ}/l) \to \GL_N(\qlb),
\quad
i=1,2,
\end{equation*}
such that,
for any finite place $v$ of $l$ not dividing $\ell$,
the ($\ell$-adic)
Weil-Deligne representation~asso\-ciated with
$\widetilde{\Pi}_{i,v}|\det|_v^{(1-N)/2}$ 
by the local Langlands correspondence is isomorphic
to~the~Frobe\-nius-semisimplification of the Weil-Deligne representation
associated with $\Sigma_{i,v}$,
the restriction~of $\Sigma_i$ to a decomposition
subgroup of $\Gal(\overline{\QQ}/l)$ at $v$. 
(Here $|\cdot|_v$ denotes the absolute value~of $l_v$~nor\-malized so that
the absolute value of any uniformizer of $l_v$
is the inverse of the cardinality of the residue field of $l_v$.)

Thanks to our local conditions at all $v\notin\SS$,
the representations $\Sigma_{1,v}$, $\Sigma_{2,v}$
are congruent mod~$\ell$.
A density argument then implies that $\Sigma_1$ and~$\Sigma_2$ 
are congruent mod $\ell$.
In particular, $\Sigma_{1,w}$, $\Sigma_{2,w}$
are congruent mod~$\ell$.

Associated with $\Sigma_{i,w}$,
there is a Frobenius-semisimple Weil-Deligne representation $(\rho_i,N_i)$.~We
show in Section \ref{GaloisWeil} that the fact that $\Sigma_{1,w}$ and $\Sigma_{2,w}$
are congruent mod~$\ell$ implies that the
smooth semi-simple representations
$\rho_1$ and $\rho_2$ are integral and congruent mod~$\ell$.
Since $\rho_i$ corresponds~to the cuspidal support of
$\widetilde{\Pi}_{i,w}|\det|_w^{(1-N)/2}\otimes_{\CC}\qlb$
(thanks to the local-global compatibility at~$w$~gi\-ven by \cite{BLGHT}),
it follows from the mod $\ell$ local Langlands correspondence of
Vign\'eras \cite{Vigl} that 
\begin{equation*}
\widetilde{\Pi}_{1,w}|\det|_w^{(1-N)/2}\otimes_{\CC}\qlb,
\quad
\widetilde{\Pi}_{2,w}|\det|_w^{(1-N)/2}\otimes_{\CC}\qlb
\end{equation*}
are integral and have the same \textit{mod $\ell$ supercuspidal support}
(which is the supercuspidal support of any irreducible component of the
reduction mod $\ell$): it follows that 
the supercuspidal~sup\-port of~the ge\-neric irreducible 
component of the reduction mod $\ell$ of
$\widetilde{\Pi}_{i,w}|\det|_w^{(1-N)/2}\otimes_{\CC}\qlb$,~de\-noted $\d_i$,
is independent of $i\in\{1,2\}$.
Since a generic irreducible $\flb$-representation
is uniquely~de\-termi\-ned by its supercuspidal~sup\-port,~we
deduce that $\d_1$, $\d_2$ are
isomorphic.
The Main~Theo\-rem \ref{MAINTHEOINTRO} now follows from the fact that
$\widetilde{\Pi}_{i,w}|\det|_w^{(1-N)/2}\otimes_{\CC}\qlb$
is isomorphic to $\BC_\ell(\pi_i
)$.~We~refer~to~Sec\-tions~\ref{sec8}~and \ref{argumentfinal} for more 
details.  

\subsection{}
\label{discussthm11}

Now let us discuss the assumptions of the main theorem. 

First,
the construction of $\Pi_1$ does not require $\pi_1$ to be cuspidal: 
it would be enough to assume that $\pi_1\otimes_{\qlb}\CC$ is a
discrete series representation of $G$
(for one, or equivalently any, choice of the
field isomorphism $\ii$: see Remark \ref{contrevie}).

However,
in order to construct the representation $\Pi_2$ satisfying our local 
conditions at all~pla\-ces
$v\notin\SS$ by the method of Khare--Vign\'eras, 
we need $\pi_2$ to be cuspidal --
even more precisely,~we
need $\pi_2$ to be~compactly~indu\-ced from some open, compact mod centre 
subgroup of $G$,
which is true of any cuspidal re\-pre\-sen\-ta\-tion of $G$,
thanks to the work of Stevens \cite{ShaunINVENT08}
since $p$~is~odd.~Con\-se\-quently, 
both the cuspidality of $\pi_2$ and \eqref{condpipi} imply that $\pi_1$
should be cuspidal,
as~the~pa\-ra\-bo\-lic restriction functors commute with reduction mod $\ell$.

For the same reason,
we want the auxiliary representation $\rho$ of $\GG(k_u)$
to be~compactly~in\-du\-ced from an open,
compact mod centre subgroup.

Moreover,
as has been explained in Paragraph \ref{bidib4},
we also need $\rho$ to have a cuspidal transfer to $\GL_N(\k_u)$.
This is why the symplectic group requires a special treatment
(see Appendix \ref{app2}),
since no cuspidal representation of a $p$-adic symplectic group
has a cuspidal transfer when $p$ is odd,
and the work of Stevens \cite{ShaunINVENT08} is not available when $p=2$.

On the other hand,
we show that part (3) of
Theorem \ref{MAINTHEOINTRO} does not hold in general
for~non-cuspidal~re\-pre\-sentations:~Re\-mark \ref{VarNR} gives an 
example of integral unramified irreducible 
$\qlb$-re\-presentations~$\pi_1$ and $\pi_2$ of $\SO_5(F)$ such that
$\r_\ell(\pi_1) = \r_\ell(\pi_2)$
but $\r_\ell(\BC_{\ell}(\pi_1))$ and $\r_\ell(\BC_{\ell}(\pi_2))$ have~no
irreducible~com\-ponent in common.

Finally,
our Assumption \eqref{condpipi}
is inspired from Vign\'eras \cite{Vigl} 3.5.
{It is tempting to conjectu\-rate~that the conclusion of
Theorem \ref{MAINTHEOINTRO} still holds when \eqref{condpipi}
is replaced by the weaker condition
``$\r_\ell(\pi_1)$ \textit{and} $\r_\ell(\pi_2)$
\textit{have a component in common}'',
but we have no evidence that such a conjecture~should be true.}

\subsection{}
\label{introCLBC}

In Appendix \ref{appBC},
we discuss the case of the local base change 
from $\GL_n(F)$ to $\GL_n(K)$ for~a~cy\-clic extension~$K$ of $F$,
denoted $\CB_{K/F}$.

As in Paragraph \ref{sommerive},
choosing a field isomorphism $\ii:\CC\to\qlb$ gives an
$\ell$-adic~local base~chan\-ge map $\CB_{K/F,\ell}$.
By using the properties of the local Langlands correspondence for $\GL_n$
with respect to conjugacy by an automorphism of $\CC$,
we prove that $\CB_{K/F,\ell}$ 
does not depend on the~choice~of~$\ii$
(see Proposition \ref{garabedian}).
We also use certain results of Zou \cite{ZouThese} 1.10
to prove an analogue~of~Theo\-rem \ref{MAINTHEOINTRO} 
for $\CB_{K/F,\ell}$~(see Paragraph \ref{BClinearcyclic}),
and give
an example of integral cuspidal $\qlb$-representa\-tions~$\pi_1$,
$\pi_2$ of $\GL_2(F)$ such that
$\r_\ell(\pi_1) = \r_\ell(\pi_2)$
but $\r_\ell(\CB_{\ell}(\pi_1))\neq\r_\ell(\CB_{\ell}(\pi_2))$
(Para\-graph \ref{examCBcyc}).

\subsection{}

We thank 
R.~Abdellatif,
H.~Atobe,
R.~Beuzart-Plessis,
P.-H.~Chaudouard,
W.~T.~Gan,~H.~Grob\-ner,
G.~Henniart,
T.~Lanard,
E.~Lapid,
J.~Mahnkopf, 
N.~Matringe,
A.~Moussaoui,
D.~Prasad,~S. W.~Shin,~S. Stevens,
O.~Taïbi and
H.~Yu
for stimulating discussions about this work.

We thank Guy Henniart for having accepted to write Appendix \ref{app2}. 

This work was partially supported by 
the Erwin Schrödinger Institute in Vienna,
when we~be\-nefited from a 2020 Research in Teams grant.
We thank the institute for hospitality, 
and for~ex\-cel\-lent working conditions.

A.~M\'\i nguez thanks the JSPS and the university of Kyoto for hospitality
and ex\-cel\-lent working conditions
when the last part of this work was done. His research was partially funded by the Principal Investigator project PAT4832423 of the Austrian Science Fund (FWF).

V.~S\'echerre also thanks the 
Institut Universitaire de France
for support and ex\-cel\-lent working conditions
when this work was done. 

\section*{Notation}

Throughout the paper,
let $p$ be a prime number, 
let $\QQ_p$ be the field of $p$-adic numbers and
let $\qpb$ be an algebraic~clo\-sure of $\QQ_p$.
By a $p$-adic field,
we mean a finite extension of $\QQ_p$ in $\qpb$.

\section{Globalizing quadratic and Hermitian forms}
\label{global1}

The purpose of this section is 
to prove the~fol\-lowing result.

\begin{theo}
\label{GLOBALG}
Let $F$ be a $p$-adic field and let $G$~be~a quasi-split 
special orthogonal, unitary or symplectic group over $F$.
There exist a totally real~num\-ber field $k$
and a connected reductive group $\GG$ over $k$ 
such that
\begin{enumerate}
\item
$\GG$ is an inner form of a quasi-split special orthogonal,
unitary or symplectic $k$-group,
\item 
there is a finite place $w$ of $k$ above $p$ such that
$k_w=F$ and $\GG(F)$ is isomorphic to $G$, 
\item 
the group $\GG(k_v)$ is compact for any real place $v$,
and quasi-split for any finite place $v$. 
\end{enumerate}
\end{theo}

This theorem will be used in Section \ref{secglobalisation}
where we prove the existence of automorphic~represen\-ta\-tions
of $\GG(\AA)$ with prescribed conditions on their local components,
where $\AA$ denotes the ring of ad\`eles of $k$.

In Section \ref{argumentfinal}, 
we will need a stronger version of Theorem \ref{GLOBALG}:
in order to transfer automorphic representations of $\GG(\AA)$ to a
general linear group, 
we will need to realize $\GG$ as a pure inner form in
the orthogonal and unitary cases,
and a rigid inner form in the symplectic case.
This~is~why,
rather than Theorem \ref{GLOBALG}, 
we will prove the stronger Theorems \ref{THEOFQ} and \ref{THEOFH} below.
For the~sym\-plec\-tic case,
see Paragraph \ref{THEOFS}.

We emphasize that $p$ may be equal to $2$ in this section. 

\subsection{Quadratic forms}
\label{S31}

In this paragraph,
$\k$ denotes either a $p$-adic field for some prime number $p$,
or a real Archime\-dean local field, or a totally real number field,
and $q$ denotes a (non-degenerate) quadratic form on~a 
$k$-vector space of~di\-mension $d\>2$.
There exists a choice of non-zero scalars $\l_1,\dots,\l_d\in\k^\times$
such that~$q$~is equi\-valent to the quadratic form
$\l_1^{}x_1^2 + \dots + \l_d^{}x_d^2$
on $k^d$.
The quantity 
\begin{equation*}
\d = \d(q) = 
\l_1\dots\l_d \text{ mod } \k^{\times2} \in \k^{\times}/\k^{\times2}
\end{equation*} 
does not depend on this choice.
It is called the discriminant of $q$.
In the sequel,
we assume that the discriminant $\d$ is fixed.
All quadratic forms are assumed to be non-degenerate.

If $\k$ is a $p$-adic field,
then $q$ is, up to equivalence,
uniquely deter\-mined by its Hasse invariant 
\begin{equation*}
\e(q) = \prod\limits_{i<j} (\l_i,\l_j) \in\{-1,1\}
\end{equation*}
where $(\cdot,\cdot)$ is the Hilbert symbol over $k$
(see \cite{SerreCoursArithm} IV.2.3 Theorem 7
or \cite{JacobsonBasicAlgebra2} Theorem 9.24).

If $\k$ is isomorphic to the field of real numbers, 
$q$ is, up to equivalence,
entirely determined~by its signature $(a,b)$ with 
$a+b=d$ and $(-1)^b=\d$. 
Its Hasse invariant is equal to $(-1)^{b(b-1)/2}$.
If $\d>0$,
then $b=2c$ for some
$c\in\{0,\dots,\lfloor d/2\rfloor\}$
and the Hasse invariant is $(-1)^{c}$.

Now suppose that $\k$ is a totally real number field, and 
$\d_v>0$ for all real places $v$.
The Hasse principle (see \cite{Scharlau}
Theorem 6.6.6)~en\-su\-res that $q$~is uniquely determined,
up to equivalence,~by all 
its localizations $q_v=q\otimes_k k_v$,~where $v$ ranges
over all places of $k$.
In other words,~it is~deter\-mined by
the Hasse invariants $\e(q_v)$~for all finite $v$ and
the signatures $(d-2c(q_v),2c(q_v))$ for all real $v$.
Conversely, a family
\begin{equation*}
((\e_v)_{\text{$v$ finite}},(c_v)_{\text{$v$ real}}),
\quad \e_v\in\{-1,1\}, 
\quad c_v\in\{0,\dots,\lfloor d/2\rfloor\}, 
\end{equation*}
corresponds to a (unique) quadratic form of dimension $d$ over $k$
and discriminant $\d$ if and only~if one has
$\e_v=1$ for almost all finite places $v$ and
\begin{equation}
\label{globqf}
\prod\limits_{\text{$v$ finite}} \e_v \cdot
\prod\limits_{\text{$v$ real}} (-1)^{c_v} = 1
\end{equation}
(see \cite{Scharlau}~Theo\-rem 6.6.10).
We give more details in \S\ref{S32} and \S\ref{S33},
depending on the parity of $d$. 

\subsection{The odd orthogonal case}
\label{S32}

If $\k$ is a $p$-adic field,
there are two equivalence classes of quadra\-tic~forms of dimension $2n+1$ 
and discriminant $\d$,
in bijection with $\{-1,1\}$ through~the~Hasse invariant.
The special orthogonal groups
associated with these quadratic forms are non-iso\-mor\-phic.
The one with Hasse invariant
\begin{equation}
\label{HasseqsSOodd}
(-1,-1)^{n(n+1)/2}\cdot(-1,\d)^n
\end{equation}
(that is $x_1x_2+\dots+x_{2n-1}x_{2n} + (-1)^n\d x_{2n+1}^2$)
is split.
The other one is non-quasi-split. 

If $\k$ is isomorphic to the field of real numbers, 
there are $n+1$ equivalence classes of quadra\-tic~forms of dimension $2n+1$ 
and discriminant $\d$.
The special~or\-thogonal groups
associated with these quadratic forms are non-iso\-mor\-phic.
Exactly~one of them is compact:
this is the one with signature $(2n+1,0)$ if $\d>0$,
and $(0,2n+1)$ if $\d<0$.

\begin{prop}
\label{oddq}
Let $\k$ be a totally real number field of degree $r$,
and $\d\in\k^{\times}/\k^{\times2}$.
Suppose that $\d_v>0$ for all real places $v$.
There is a quadratic form $q$ of dimension $2n+1$ and
discrimi\-nant $\d$ such that $\SO(q)$ is
compact at all real places and quasi-split at all finite places 
if and only if $rn(n+1)/2$ is even.
When this is the case, $q$ is unique up to equivalence. 
\end{prop}

\begin{proof}
A quadratic form $q$ over $k$ of dimension $2n+1$ and discriminant $\d$ 
is entirely determi\-ned,~up to equivalence,
by the Hasse invariants $\e(q_v)\in\{-1,1\}$ for all
finite places $v$ and the~si\-gna\-tu\-res
$(2n+1-2c(q_v),2c(q_v))$ for all real places $v$ of $k$. 
Non-equivalent quadratic forms~de\-fine~non-isomorphic
special orthogonal groups.

For $\SO(q)$ to be compact at all real places and quasi-split at all 
finite places,
$q$ must have~in\-variants $c_v=0$ for all real $v$ and
$\e_v=(-1,-1)_v^{n(n+1)/2}\cdot(-1,\d_v)_v^{n}$
for all finite $v$,
where $(\cdot,\cdot)_v$ is the Hilbert symbol with respect to $k_v$.
By \eqref{globqf}, such a $q$ exists if and only if
\begin{equation*}
\prod\limits_{\text{$v$ finite}} (-1,-1)_v^{n(n+1)/2}
\times \prod\limits_{\text{$v$ finite}} (-1,\d_v)_v^{n}
= 1.
\end{equation*}
Thanks to the Hilbert reciprocity law (\cite{OMeara} VII),
the left hand side is equal to 
\begin{equation*}
\prod\limits_{\text{$v$ real}} (-1,-1)_v^{n(n+1)/2}
\times \prod\limits_{\text{$v$ real}} (-1,\d_v)_v^{n}
= (-1)^{rn(n+1)/2}
\end{equation*}
(since $\d_v>0$ for all real $v$), which gives the expected result. 
\end{proof}

\begin{rema}
Given any $k$,
let $q$ be a quadratic form of dimension $2n+1$ and discriminant~$1$ over $k$.
Then,
for any $\d\in\k^\times$,
the quadratic form $\d q$ has discriminant $\d$ and $\SO(\d q)=\SO(q)$.
\end{rema}

\subsection{The even orthogonal case}
\label{S33}

In this paragraph,
we assume that the dimension of $q$ is $2n$.
It will be convenient to use the normalized discriminant
$\a=(-1)^n\d$.

Suppose first that $\k$ is a $p$-adic field.
\begin{itemize}
\item[$\bullet$] 
If $n=1$, 
there is only one equivalence class of quadratic forms of dimension $2$ 
and~norma\-lized dis\-cri\-mi\-nant $\a=1$.
Its Hasse invariant is $1$.
The special orthogonal group associated with it is~iso\-mor\-phic~to~$\GL_1(k)$.
\item[$\bullet$]
Suppose that $n\>2$ or $\a\neq1$.
There are two equivalence classes of quadratic forms~of~di\-mension
$2n$ and discrimi\-nant $\d$,
characterized by their Hasse~in\-variant.
The special orthogonal groups
associated with them are non-iso\-mor\-phic
if and only if $\a=1$.  
When~this is the case,~the one with Hasse invariant $(-1,-1)^{n(n-1)/2}$
(that is the quadratic form $x_1x_2+\dots+x_{2n-1}x_{2n}$)~is
split, and the~other one is non-quasi-split.
Other\-wi\-se,
let $l$ be the quadratic extension of $k$~gene\-ra\-ted by a square root of $\a$:
if~$q$
is a~qua\-dratic~form of dimension $2n$ and discriminant~$\d$ over~$k$, 
then $\l q$ has same discriminant and~op\-posite Hasse invariant
for any scalar $\l\in\k^\times$ which is not an $l/k$-norm,
and $\SO(\l q)=\SO(q)$.
\end{itemize}

If $\k$ is isomorphic to the field of real numbers, 
there are $n+1$ equivalence classes of quadra\-tic~forms of dimension $2n$ 
and discriminant $\d$. 
Quadratic forms with signatures $(a,b)$ and $(a',b')$ 
define isomorphic special orthogonal groups if and only if one has 
$b'\in\{a,b\}$.
If $\d<0$, there is no compact special orthogonal group.
If $\d>0$,
there~is~exact\-ly~one compact special orthogonal group:
this is the one with $b\in\{0,2n\}$.

\begin{prop}
\label{evenq}
Let $\k$ be a totally real number field of degree $r$,
and $\d\in\k^{\times}/\k^{\times2}$.
Suppose that $\d_v>0$ for all real places $v$.
\begin{enumerate}
\item 
There is a quadratic form $q$ of dimension $2n$ and
discrimi\-nant $\d$ such that $\SO(q)$ is~compact
at all real places and quasi-split at all finite places 
if and only if either $n$ is odd, or $\d\neq(-1)^n$, 
or $rn(n-1)/2$ is even. 
\item
Assume that $\d\neq(-1)^n$.
For any finite place $w$ such that $\d_w\neq(-1)^n$
and any $\e\in\{-1,1\}$,
there is a quadratic form $q$ as in (1) satisfying 
the extra condition $\e(q\otimes k_w)=\e$.
\end{enumerate}
\end{prop}

\begin{proof}
A quadratic form $q$ over $k$ of dimension $2n$ and discriminant $\d$ 
is entirely determined, up to equivalence, 
by the Hasse in\-variants $\e(q_v)\in\{-1,1\}$ for all finite 
places $v$ and the signatures
$(2n-2c(q_v),2c(q_v))$ for all real places $v$.
A quadratic form $f$ with same dimension and discri\-mi\-nant as $q$
defines a special orthogonal group 
isomorphic to $\SO(q)$ if and only if
they have the same Hasse~in\-va\-riants
for all finite $v$ such that $\a_v=1$,
and $c(f_v)\in\{n-c(q_v),c(q_v)\}$ for all real places $v$. 

For $\SO(q)$ to be compact at all real places and quasi-split at all 
finite places,
$q$ must have~in\-variants $c_v\in\{0,n\}$ for all real places $v$ and 
$\e_v=(-1,-1)_v^{n(n-1)/2}$
for all finite places $v$ such that $\a_\v=1$.
(Recall that $\a=(-1)^n\d$.)
By \eqref{globqf}, such a $q$ exists if and only if
\begin{equation*}
\label{condscharlau}
\prod\limits_{\text{$v$ finite} \atop{\a_\v\neq1}} \e_v \times
\prod\limits_{\text{$v$ finite} \atop{\a_\v=1}} (-1,-1)_v^{n(n-1)/2} \times
(-1)^{ns} = 1
\end{equation*}
where $s$ is the number of real places such that $c_v=n$.
If $n$ is odd,
we may adjust $s\in\{0,\dots,r\}$ so that
this product is $1$.
If $\a\neq1$,
we may adjust the signs $\e_v$ for the finite $v$ such that 
$\a_\v\neq1$ so that this product is $1$.
(Since the number of such $v$ is at least $2$,
we may even assume that~$\e_w$ is equal to a given sign $\e$
for a given $w$ as in (2).)
If $n$ is even and $\a=1$, the condition is 
\begin{equation*}
\prod\limits_{\text{$v$ finite}} (-1,-1)_v^{n(n-1)/2} = 1
\end{equation*}
and Hilbert's reciprocity law says that the left hand side is equal to
\begin{equation*}
\prod\limits_{\text{$v$ real}} (-1,-1)_v^{n(n-1)/2} = (-1)^{rn(n-1)/2},
\end{equation*}
which gives the expected result. 
\end{proof}

\subsection{Globalizing the base field}

The following lemma will be useful in the remainder of this section.

\begin{lemm}
\label{arthurtotallyreal}
Let $F$ be a $p$-adic field. 
\begin{enumerate}
\item 
There exists~a totally real number field $\k$ of even degree such that $k_w=F$ 
for some finite place $w$ of $k$ dividing $p$.
\item 
{If $p\neq2$,}
we may further assume that there exists a finite place $u$ of $k$ 
such that $k_u\simeq\QQ_2$.
\end{enumerate} 
\end{lemm}

\begin{proof} 
We follow the proof of \cite{Arthur} Lemma 6.2.1.
Let us write $F=\QQ_p(\b)$ for some root $\b\in\F$ of a monic irreducible
polynomial $f$ of degree $r=[F:\QQ_p]$ with coefficients in $\QQ_p$.
Given a field $E$, we
identify the space of monic polynomials of degree $r$ with coefficients in $E$ 
with $E^r$.

By Krasner's lemma (see \cite{Robert} 3.1.5), 
there is an open neighborhood $U_p$ of $f$ in $\QQ_p^r$ such that
any~$g\in U_p$ has a root $\b'\in\overline{\QQ}_p$ such that
$\QQ_p(\b')=F$.
Let $U_\infty$ be the open subset of $\RR^r$ made of all monic polynomials with
$r$ distinct real roots.
Since the diagonal image of $\QQ^r$ in $\RR^r\times\QQ_p^r$ is dense,
the intersection
\begin{equation*}
\QQ^r \cap (U_\infty\times U_p)
\end{equation*}
is non-empty.
We may replace $f$ by a polynomial in this intersection,
which we still denote~by~$f$. 
The number field $k=\QQ(\b)$ is totally real,
and $k_w=F$ for some~fi\-ni\-te~place $w$ of $k$ dividing $p$.
If the degree of $k$ is even, we are done.
Otherwise, we
choose a monic irreducible polynomial~$g$~of degree $2$ over $\QQ$
which splits over $\RR$ and $\QQ_p$, 
whose existence can be proven in the same way as above.
Then replace $k$ by $k(\g)$ where $\g$ is a root~of $g$ in $\QQ_p$.

Suppose now that $p\neq2$,
and let $U_2$ be the open subset of $\QQ_2^r$ made of all monic polynomials with
$r$ distinct roots in $\QQ_2$.
We may replace $f$ by a polynomial in
$\QQ^r \cap (U_\infty\times U_p \times U_2)$, 
which we still denote~by~$f$. 
The number field $k=\QQ(\b)$ is totally real,
$k_w=F$ for some~fi\-ni\-te~place $w$ of $k$
dividing $p$,
and $2$ is totally split in $k$. 
If the degree of $k$ is even, we are done.
Otherwise,~we
choose a monic irreducible polynomial~$g$~of degree~$2$~over~$\QQ$
which splits over $\RR$, $\QQ_p$ and $\QQ_2$,
then replace $k$ by $k(\g)$ where $\g$ is a root~of $g$ in $\QQ_p$. 
\end{proof}

\begin{rema}
\label{arthurtotallyimag}
With a similar argument, 
one can prove in addition to part (1) of Lemma \ref{arthurtotallyreal} that, 
if $E$ is a quadratic extension of $F$ in $\qpb$,
there is a totally imaginary quadratic extension $l$ of $k$ such that 
$l_w=E$.
\end{rema}

\begin{rema}
Part (2) of Lemma \ref{arthurtotallyreal} 
will be needed in Section \ref{argumentfinal},
\textit{in the symplectic case},
in order to apply the results 
of Appendix \ref{app2}.
\end{rema}

\subsection{Proof of Theorem \ref{GLOBALG} in the
special orthogonal case}

We prove Theorem \ref{GLOBALG} in the case where $G$ is special orthogonal, 
that is, 
there is a quadratic form $Q$ over $F$ such that 
$G$ is isomorphic to $\SO(Q)$.
We will prove the following stronger~re\-sult. 

\begin{theo}
\label{THEOFQ}
Let $Q$ be a quadratic form over $F$
such that $\SO(Q)$ is quasi-split.
There~exist a totally real number field $\k$
and a quadratic form $q$ over $k$ such that
\begin{enumerate}
\item 
there is a finite place $w$ of $k$ dividing $p$ such that 
\begin{enumerate}
\item 
the field $k_w$ is equal to $F$,
\item
the~qua\-dra\-tic forms $q\otimes F$ and $Q$ are equivalent,
\end{enumerate}
\item
the group $\SO(q\otimes\k_v)$ is compact for all real $v$, 
and quasi-split for all finite $v$.
\end{enumerate}
\end{theo}

\begin{proof}
By Lemma \ref{arthurtotallyreal},
there exists a totally real number field $\k$ of even degree such that~$k_w$ and 
$F$ are equal for some finite place $w$ of $k$ dividing $p$.
Fix a $\g\in\F^\times$ such that the~discri\-minant of $Q$
is $\g \F^{\times2}$,
and fix a $\d\in k^\times$ such that $\g^{-1}\d_w\in\F^{\times2}$
and $\d_v>0$ for all real~$v$.

By Proposition \ref{oddq} when $Q$ has odd dimension 
and Proposi\-tion~\ref{evenq} when $Q$ has even dimension,
there is a quadratic form $q$ of discriminant $\d$ 
satisfying (2).
Moreover,
the quadratic forms
$q\otimes F$ and $Q$ have the same~dis\-cri\-minant and
define quasi-split special orthogonal groups. 

If $Q$ has odd dimension,
or if $Q$ has dimension $2n$ and $\g=(-1)^n$,
they are thus equivalent.

Otherwise, 
use Proposition~\ref{evenq}(2) with $\e=\e(Q)$
to ensure that 
$q\otimes F$ and $Q$~have the same Hasse invariant:
they are thus equivalent.
\end{proof}

\begin{rema}
\label{sosplit}
In addition to Theorem \ref{THEOFQ},
there is always a finite place $u\neq w$ of $k$ such that 
the group $\SO(q\otimes{k_u})$ is split: 
one can choose
\begin{enumerate}
\item 
any finite place different from $w$ in the odd orthogonal case, 
\item 
any finite place $u\neq w$ such that
$(-1)^n\d_u\in\k_u^{\times2}$
in the even orthogonal case. 
\end{enumerate}
\end{rema}

\subsection{Hermitian forms}
\label{SHERM}

In this paragraph, 
$l$ is a separable quadratic $k$-algebra
(where $k$ is as in Paragraph \ref{S31})
and $h$ is a (non-degenerate) $l/k$-Hermi\-tian form on~an 
$l$-vector space of~di\-mension $n\>1$.
There exists a choice of non-zero scalars $\l_1,\dots,\l_n\in\k^\times$
such that~$h$~is equi\-valent to the $l/k$-Hermitian form
$\l_1^{}\N_{l/k}(x_1) + \dots + \l_n^{}\N_{l/k}(x_n)$
on $l^n$. 
The quantity 
\begin{equation*}
\d = \d(h) = \l_1\dots\l_n \text{ mod }\N_{l/k}(l^\times) \in \k^{\times}/\N_{l/k}(l^\times)
\end{equation*}
does not depend on this choice.
It is called the discriminant of $h$. 
Fix an $\a\in\k^{\times}$ such that~$l$~is iso\-morphic to the $k$-algebra
$k[X]/(X^2-\a)$.
The image of $\a$ in $\k^\times/\k^{\times2}$ will still be denoted~$\a$.

Up to equivalence,
$h$ is uniquely determined by its trace form $t$,
that~is, the quadratic form~of dimension $2n$ over $k$
obtained by seeing $l^n$ as~a $k$-vector space
(\cite{Scharlau} Theo\-rem 10.1.1).

If $l$ is split,
that is, if $l\simeq k\times k$,
then $\N_{l/k}(l^{\times})=k^\times$ and we may choose $\a=1$.
There~is,~up~to equivalence, 
a unique $l/k$-Hermi\-tian form of dimension $n$.
Its discriminant is trivial, and 
the~uni\-ta\-ry group associa\-ted~with it is
(non-canonically) isomorphic to $\GL_n(k)$. 
More precisely, if one fixes an isomorphism $l\simeq k\times k$
of $k$-algebras,
$h$ identifies with a non-degenerate
bilinear form~on $k^n\times k^n$,
the group $\GL_n(l)$ identifies with
$\GL_n(k)\times\GL_n(k)$
and there is an isomorphism
\begin{eqnarray}
\label{isoUGL}
\GL_n(k) &\simeq& \U(h) \\
\notag
g &\mapsto& (g,g^*)
\end{eqnarray}
where $g^*$ is the contragredient of $g\in\GL_n(k)$ with respect to $h$.
(Note that changing the~isomor\-phism $l\simeq k\times k$ has the
effect of exchanging $g$ and $g^*$ in \eqref{isoUGL}.)
Also, the trace form $t$ of $h$ is maximally ~iso\-tropic,
that is, it is the sum of $n$ hyperbolic planes.

If $l$ is a quadratic extension of $k$,
a quadratic form of dimension $2n$ over $k$ is the trace form~of an 
$l/k$-Hermitian form if and only if $q\otimes_kl$ is maximally isotropic 
(\cite{Scharlau} Theo\-rem 10.1.2).

If $l/\k$ is a quadratic extension of $p$-adic fields, 
there are two equi\-va\-lence classes of $l/\k$-Hermi\-tian forms of dimension $n$,
in bijection with $\k^{\times}/\N_{l/k}(l^{\times})$ through~the 
dis\-criminant.
\begin{itemize}
\item[$\bullet$] 
If $n$~is odd, 
the uni\-tary groups associated with these Hermitian forms are~iso\-mor\-phic.
{More precisely,
if $\a\neq1$ and $h$ is a Hermitian form of odd dimension over $k$, 
then $\d h$
is unequivalent~to $h$ for any 
$\d\in\k^\times$ such that $\d\notin\N_{l/k}(l^\times)$,
and the group $\U(\d h)=\U(h)$ is quasi-split.}
\item[$\bullet$] 
If $n$ is~even, the~uni\-tary group corresponding to the discriminant 
$(-1)^{n/2}$ is~quasi-split,~and the other one is non-quasi-split.
The trace form $t$ of $h$ has discriminant $(-\a)^n$ and Hasse~in\-va\-riant 
\begin{equation}
\e(t) = (\a,\d) \cdot (-\a,-1)^{n(n-1)/2}.
\end{equation}
\end{itemize}

If $l/\k$ is isomorphic to $\CC/\RR$,
the Hermitian form $h$
is uniquely determined, up to~equi\-va\-lence, by
its~si\-gnature~$(a,b)$ with $a+b=n$.
Its discriminant is $(-1)^b$.
Its trace form $t$ has discrimi\-nant~$1$~and signature $(2a,2b)$.
The unitary group $\U(h)$ is compact if and only if $b\in\{0,n\}$.

If $l$ is a totally imaginary quadratic extension of a totally real number
field $k$ 
(thus $\a_v<0$~for all real places $v$ of $k$), 
then $h$~is~uni\-quely determi\-ned, up to equivalence, by any one of 
the~fol\-lo\-wing data:
\begin{enumerate}
\item 
the equivalence class of its trace form $t$,
\item
the Hasse invariants $\e(t_v)$ for all finite $v$
and the integers $b(t_v)$ for all real $v$,
\item
the equivalence classes of its localizations $h_v=h\otimes_k k_v$ for all $v$,
\item
the discriminants $\d(h_v)$ for all finite $v$
and the integers $b(h_v)$ for all real $v$.
\end{enumerate}
We have just seen that (3) and (4) are equivalent, 
and we have seen that (1) and (2) are equiva\-lent in Paragraph \ref{S31}.
Now the fact that (2) and (3) are equivalent follows form the formulas 
\begin{equation*}
\e(t_v) = (\a_v,\d(h_v))_v \cdot (-\a_v,-1)_v^{n(n-1)/2} 
\text{ for finite $v$,}
\quad
b(t_v)=2b(h_v)
\text{ for real $v$,} 
\end{equation*}
the first one including the case where $l_v=l\otimes_k k_v$ splits over $k_v$
(as $\a_v=1$ in this case)
and the fact that, 
when $\a_v\neq1$,
the map $x\mapsto(\a_v,x)_v$
is a bijection from
$\k_v^\times/\N_{l_v/k_v}(l_v^\times)$~to~$\{-1,1\}$.~Con\-ver\-sely, a family
\begin{equation}
\label{locinvh}
((\d_v)_{\text{$v$ finite}},(b_v)_{\text{$v$ real}}),
\quad \d_v\in\k_v^\times/\N_{l_v/k_v}(l_v^\times), 
\quad b_v\in\{0,\dots,n\}, 
\end{equation}
corresponds to an $l/k$-Hermitian form of dimension $n$ 
if and only if there exists a $\d\in\k^\times/\N_{l/k}(l^\times)$ such that
$\d\equiv\d_v$ mod $\N_{l_v/k_v}(l_v^\times)$ for all $v$
(where we have put $\d_v=(-1)^{b_v}$ at all real places $v$),~and
when it is the case such a Hermitian form is unique. 
Indeed,
this is certainly a necessary~con\-di\-tion and,
when it is satisfied, 
the family 
\begin{equation}
\label{locinvt}
((\e_v)_{\text{$v$ finite}},(2b_v)_{\text{$v$ real}}),
\quad 
\e_v=(\a_v,\d_v)_v \cdot (-\a_v,-1)_v^{n(n-1)/2}, 
\end{equation}
satisfies $\e_v=1$
for almost all finite places $v$ together with 
\begin{eqnarray*}
\prod\limits_{\text{$v$ finite}} \e_v
\times \prod\limits_{\text{$v$ real}} (-1)^{b_v} &=&
\prod\limits_{\text{$v$ real}} (\a_v,\d_v)_v \times
\prod\limits_{\text{$v$ real}} (-\a_v,-1)_v^{n(n-1)/2} 
\times \prod\limits_{\text{$v$ real}} (-1)^{b_v} \\
&=& \prod\limits_{\text{$v$ real}} (-1,-1)_v^{b_v}  
\times \prod\limits_{\text{$v$ real}} (-1)^{b_v} 
\end{eqnarray*}
which is equal to $1$
(thanks to the fact that $\a_v<0$ 
and $\d_v=(-1)^{b_v}$ for all real~$v$).
Thus there is a unique
quadratic form of dimension $2n$ over $k$ and discriminant 
$(-\a)^n$ with lo\-cal~in\-va\-riants \eqref{locinvt}.
One can verify that it is maximally isotropic over $l$
(as it is maximally isotropic over $l_v$~for all $v$).
It is thus the trace form of an 
$l/k$-Hermitian form of dimension $n$, as expected. 

\begin{prop}
\label{herm}
Let $\k$ be a totally real number field of degree $r$
and $l$ be a totally imaginary quadratic~ex\-ten\-sion of $k$. 
\begin{enumerate}
\item 
There is a Hermitian form $h$ of dimension~$n$
such that $\U(h)$ is~com\-pact at all~real places.
\item
There is a Hermitian form $h$ of dimension~$n$
such that $\U(h)$ is~com\-pact
at all~real places and quasi-split at all finite places
if and only if either $n$ is odd,
or $n$ and $rn/2$ are both even.
\item
Assume that $n$ is odd. 
For any finite place $w$ 
and any 
$\e\in\k_w^\times/\N_{l_w/k_w}(l_w^\times)$,
there is a~Her\-mi\-tian form $h$ as in (2) satisfying
the extra condition $\d(h\otimes k_w)=\e$.
\end{enumerate}
\end{prop}

\begin{proof}
Assertion (1) is verified by any Hermitian form $h$ of dimension~$n$
over $k$ 
such that we have $b(h_v)\in\{0,n\}$ at all real places $v$.

Assume now that $n=2m$ for some $m\>1$.
A Hermitian form~$h$~of dimension $n$
and dis\-crimi\-nant $\d$ satisfies (2) if and only if 
$b(h_v)\in\{0,n\}$ and $\d_v>0$ for all real~$v$, 
and $\d_v=(-1)^{m}$ for all finite $v$.
Such a $\d\in\k^\times/\N_{l/k}(l^\times)$ exists if and only if 
$\d_v\in\N_{l_v/k_v}(l_v^\times)$ for almost all finite $v$,
and 
\begin{equation*}
\prod\limits_{v} (\a_v,\d_v)_v=1.
\end{equation*}
The first condition is satisfied since $l_v$ is either split over $k_v$ or an 
unramified extension of $k_v$ for almost all finite $v$.
The second condition follows from 
\begin{equation*}
\prod\limits_{v} (\a_v,\d_v)_v = \prod\limits_{\text{$v$ finite}} (\a_v,-1)_v^m
= \prod\limits_{\text{$v$ real}} (\a_v,-1)_v^m
=(-1)^{rm}
\end{equation*}
thanks to the Hilbert reciprocity law and the fact that $\a_v<0$ for all
real $v$.

Assume now that $n$ is odd.
A Hermitian form $h$ of dimension~$n$ 
and dis\-crimi\-nant $\d$ satisfies~(2) 
if and~only if $b(h_v)\in\{0,n\}$~at all real places $v$,
and satisfies (3) 
if and only if $b(h_v)\in\{0,n\}$~at~all real~places $v$ and
$\d\e^{-1} \in \N_{l_w/k_w}(l_w^\times)$.
Fix a finite place $y\neq w$ and a $\kappa\in l_y^\times$.
We claim that such an $h$ exists,
with the extra conditions
\begin{itemize}
\item 
$b(h_v)=0$ for all real places $v$,
\item
$\d\in\N_{l_v/k_v}(l_v^\times)$ for all places $v\notin\{w,y\}$ 
and $\d\kappa^{-1} \in \N_{l_y/k_y}(l_y^\times)$. 
\end{itemize}
Arguing as in the case when $n$ is even,
it suffices to choose a $\kappa\notin\N_{l_y/k_y}(l_y^\times)$,
which~is possible as soon as $y$ has been chosen such that $\a_y\neq1$.
\end{proof}

\subsection{Proof of Theorem \ref{GLOBALG} in the unitary case}

We prove Theorem \ref{GLOBALG} in the case where $G$ is unitary, 
that is, 
there are a quadratic extension $E$ of $F$ 
and an $E/F$-Hermitian form $H$ over $F$ such that 
$G$ is isomorphic to $\U(H)$. 
We will~pro\-ve the following more precise theorem. 

\begin{theo}
\label{THEOFH}
Let $H$ be an $E/F$-Hermitian form such that $\U(H)$ is quasi-split.
There exist a totally real number field $\k$, a to\-tally imaginary quadratic
extension $l$ of $k$ and an $l/k$-Hermitian form $h$ such that:
\begin{enumerate}
\item 
there is a finite place $w$ of $k$ above $p$ such that 
\begin{enumerate}
\item 
one has $k_w=F$ and $l_w=E$,
\item
the Hermitian forms $h\otimes\F$ and $H$ are equivalent,
\end{enumerate}
\item
the group $\U(h\otimes\k_v)$ is
compact for all real $v$,
and quasi-split for all finite $v$. 
\end{enumerate}
\end{theo}

\begin{proof}
By Lemma \ref{arthurtotallyreal} and Remark \ref{arthurtotallyimag},
there are a totally real number field $\k$ of even degree and 
a totally imaginary quadratic~ex\-tension $l$ of $k$ 
such that $k_w=F$ and $l_w=E$
for some finite place $w$ of $k$ dividing $p$.

By~Pro\-po\-si\-tion~\ref{herm},
there exists an $l/k$-Hermitian form $h$ satisfying (2).
Moreover, 
the~Her\-mi\-tian forms $h\otimes F$ and $H$
define quasi-split unitary groups.

If $H$ has even dimension,
they are thus equivalent.
If $H$ has odd dimension,
one uses Proposition~\ref{herm}(3) with $\e=\d(H)$
to ensure that 
$h\otimes F$ and $H$~have the same discriminant:
they are thus equivalent.
\end{proof}

\begin{rema}
\label{bavardise}
In addition to Theorem \ref{THEOFH},
there is always a finite place $u\neq w$ of $k$ such that 
$\U(h\otimes k_u)$ is split: 
it suffices to choose any finite place $u\neq w$ such that
$l_u\simeq\k_u\times\k_u$.
\end{rema}

\subsection{The symplectic case}
\label{Sp}
\label{THEOFS}

We now consider the case where $G$ is a symplectic group, 
that is, 
there exists a non-degenerate symplectic form $A$
over $F$ such that $G=\Sp(A)$.
By Lemma \ref{arthurtotallyreal},
there exists a totally real number field $\k$ of even degree such that
$k_w=F$ for some finite place $w$ of $k$~dividing~$p$.
(Moreover,
when $p$ is odd,
we may further assume that there is a finite place $u$ of $k$ 
such that~$k_u\simeq\QQ_2$.)
In this case,
Theorem \ref{GLOBALG} is given by \cite{TaibiMRL16} 2.1.1.
See also \cite{Taibi}~Pro\-po\-sition 3.1.2,
where the inner form $\GG$ is
rea\-li\-zed as a {rigid inner form} of $\Sp_{2n}$ over $k$. 

\section{Congruences of automorphic forms of definite groups}
\label{sec:GlobReps}
\label{global2}

In this section, we fix a prime number $\ell$.
Let $\qlb$ be an algebraic closure of the field of $\ell$-adic integers,
$\zlb$ be its ring of integers and $\flb$ be its residue field.
We fix a field isomorphism 
\begin{equation}
\label{isoiota}
\iota:\CC\to\qlb 
\end{equation}
and a number field $k$.
We denote by $\AA=\AA_f\times\AA_\infty$ the ring of ad\`eles of $k$. 

Given a locally compact, totally disconnected group $G$,
an open subgroup $K$ of~$G$,~a
commuta\-tive ring $R$ and a smooth $R$-representation $\rho$ of $K$, 
we denote by
\begin{equation*}
\Hh_R(G,\rho)
\end{equation*}
the endo\-mor\-phism $R$-algebra of the
compact induction of $\rho$ to $G$,
called the Hecke $R$-algebra~of $G$~rela\-ti\-ve to~$\rho$.
When $\rho$ is the trivial $R$-character of $K$,
it naturally identifies with the convolution $R$-algebra 
made of $K$-bi-inva\-riant, compactly~sup\-por\-ted
$R$-valued functions on $G$,
and we denote it by $\Hh_R(G,K)$. 

Let $F$ be a $p$-adic field for some $p\neq\ell$,
$G$ be the group of $F$-rational points of a reductive group defined over $F$
and $\pi$ be an irreducible (smooth) representation of $G$ on a
$\qlb$-vector space $V$.~It
is said to be \textit{integral} if $V$ carries a $G$-stable 
$\zlb$-lattice.
Given such a lattice $L$,
the representation of $G$ on the $\flb$-vector space $L\otimes\flb$
(where $\flb$ is the residue field of $\zlb$)
is smooth and has finite length,
and its semi-simplification does not depend on the choice of
$L$ (\cite{Vigw} Theorem 1).
This semi-simplification is de\-noted $\r_\ell(\pi)$,
and called the \textit{reduction mod $\ell$} of $\pi$.
One defines similarly the reduction mod $\ell$ of an irreducible
$\qlb$-representation of a compact, open subgroup of $G$.

\subsection{}

Let $\GG$ be a connected reductive group defined over $k$.
We assume that $\GG$ is definite, that is,
the group $\GG(\AA_\infty)$ is compact.
We embed diagonally $\GG(k)$ in $\GG(\AA_f)$ and set
\begin{equation*}
\Y=\GG(k) \backslash \GG(\AA_f).
\end{equation*}
The quotient $\Y$ is compact (\cite{PlatonovRapinchuk} \textsection 5)
and hence $\Y/\K$ is finite
for any open
subgroup~$\K$ of $\GG(\AA_f)$.

We denote by $\Aa( \GG)$ the space 
of functions $\GG(k)\backslash\GG(\AA)\to\CC$
which are square integrable with respect to a Haar measure on $\GG(\AA)$.
It is endowed with the natural unitary $\CC$-re\-presentation of $\GG(\AA)$.

\subsection{}

Given an open subgroup $\K$ of $\GG(\AA_f)$, 
let $\Alg(\GG,\K)$ be~the free $\ZZ$-module of finite rank 
made of all functions $\Y\to\ZZ$ which are invariant under
right~trans\-lations by $\K$
(\cite{Vigl} 3.3). 
We consider the $\ZZ$-module
\begin{equation*}
\Alg(\GG) = \Cc^\infty (\Y,\ZZ) = \underset{\K}{\bigcup} \ \Alg(\GG,\K)
\end{equation*}
(where $\K$ ranges over all open subgroups of $\GG(\AA_f)$)
of locally constant functions $\Y\to\ZZ$,
called \emph{trivial-at-infinity algebraic automorphic forms} for $\GG$
(see Paragraph \ref{sub:alg} below).
This module
is endowed with the natural~$\ZZ$-re\-presentation of $\GG(\AA_f)$.
Given any commutative ring $\R$, we write
\begin{equation*}
\Alg_\R(\GG)=\Alg(\GG)\otimes_\ZZ\R,
\quad
\Alg_\R(\GG,\K)=\Alg(\GG,\K)\otimes_\ZZ\R.
\end{equation*}
The natural $\R$-representation of $\GG(\AA_f)$ on $\Alg_\R(\GG)$
is~ad\-mis\-sible (\cite{Vigl} 3.3.2).

If $\R$ is the field $\qlb$,
the representation of $\GG(\AA_f)$ on $\Alg_{\qlb}(\GG)$
is semi-simple
and any of~its~ir\-re\-du\-cible components has an $\Oo_\E$-struc\-ture 
for some finite extension $\E$ of $\mathbb{Q}_\ell$ in $\qlb$
(\cite{Vigl}~3.3.2).

\subsection{}
\label{sub:alg}

Let $\K$ be an open subgroup of $\GG(\AA_f)$
and $\R$ be a commutative ring. 
The Hecke $\R$-algebra of $\GG(\AA_f)$ relative to $\K$ is the convolution
$\R$-algebra
\begin{equation*}
\Hh_\R(\GG,\K) = \Hh_\R(\GG(\AA_f),\K)
\end{equation*}
made of $\K$-bi-invariant, compactly supported
functions $\GG(\AA_f)\to\R$.
It naturally acts on the~$\R$-module $\Alg_\R(\GG,\K)$.

As $\GG$ is
definite, 
there is, by \cite{Gross} Proposition 8.5,
an explicit isomorphism 
\begin{equation*}
\Alg_\CC(\GG,\K)\simeq \Aa(\GG)^{\K\times \GG(\AA_\infty)}
\end{equation*}
of $\Hh_\CC(\GG,\K)$-modules
(see \cite{Gross} (8.4) and Proposition 8.3).
In particular, there is a bijection
\begin{equation}
\label{algFAtoFA}
\Theta \leftrightarrow \Pi
\end{equation}
between
\begin{itemize}
\item 
the irreducible subrepresentations
$\Theta$ of $\Alg_\CC(\GG)$ such that 
the space $\Theta^\K$ of $\K$-fixed vectors in $\Theta$
is non-zero,
\item
the irreducible automorphic representations
$\Pi=\Pi_f \otimes \Pi_\infty$ of $\GG(\AA)=\GG(\AA_f)\times\GG(\AA_\infty)$,
that is, the irreducible subrepresentations of $\Aa(\GG)$
such that $\Pi_\infty$ is trivial and $\Pi_f^\K$ is non-zero. 
\end{itemize}

\subsection{}
\label{sub:conditions}

Let us fix an irreducible automorphic representation $\Pi$ of $ \GG(\AA)$ 
which is trivial on $ \GG(\AA_\infty)$.~By \eqref{algFAtoFA}, 
we can see $\Pi$ as an irreducible subrepresentation
of $\Alg_{\CC}(\GG)$,
thus as an irreducible factor of $ \Alg_\qlb (\GG)$
\textit{via} the isomorphism $\iota$ fixed in \eqref{isoiota}.

We fix two finite places $w$ and $u$  of $k$ not dividing $\ell$ and a finite
set $ \SS $ of finite places of $ k $ such that: 
\begin{enumerate}
\item
the set $\SS$ contains $w,u$ and all finite places above $\ell$, 
\item
for any finite place $v \notin \SS$,
the group $\GG$ is unramified over $k_v$, 
and the local component~$ \Pi_v $ is un\-ra\-mi\-fied with respect to some
hyperspecial maximal compact subgroup $\K_v$ of $\GG(k_v)$. 
\end{enumerate}

Any finite place $ v \notin \SS $ thus defines a character
\begin{equation}
\label{defSat}
\chi_v : \Hh_\qlb(\GG(k_v), \K_v) \to \qlb
\end{equation}
which we call the \textit{Satake parameter} of $\Pi_v$.

Recall that 
$\Pi$ is admissible and has an $\Oo_\E$-structure
for some finite extension $\E$ of $\QQ_\ell$ in $\qlb$.
Let us write $\Pi=\Pi_{(\ell)} \otimes \Pi^{(\ell)}$,
where 
$\Pi_{(\ell)}$ is the~ten\-sor pro\-duct of all $\Pi_v$ such that $v$ 
divides~$\ell$~and 
$\Pi^{(\ell)}$ is the restricted ten\-sor pro\-duct of all $\Pi_v$ such that $v$ is
finite and does not divides $\ell$.
By~\cite{Vigl} A.3,
both $\Pi_{(\ell)} $ and $\Pi^{(\ell)}$
have an $\Oo_\E$-structure.~By
applying \cite{Vigl} A.4 to $\Pi^{(\ell)}$,
we get that each~$\Pi_v$, 
for $ v \notin \SS $ finite,
has an $\Oo_\E$-structure.
Fixing such an $\Oo_\E$-structure,
the~$\Oo_\E$-algebra~$\Hh_{\Oo_\E}(\GG(k_v),\K_v)$ acts on it 
through the character $\chi_v$, 
which thus has values in $\Oo_\E$. 
It follows that the 
restriction of~$\chi_v$ to $ \Hh_{\zlb}(\GG(k_v),\K_v) $ 
has values in $ \zlb$. 

\subsection{}
\label{sub:defs}

We now make the following assumptions on the representation $\Pi$ of
Paragraph \ref{sub:conditions}:
\begin{itemize}
\item 
the local component $\Pi_w$ is cuspidal,
and is compactly induced from an irreducible represen\-ta\-tion 
of some compact
open subgroup $\K_w$ of $\GG(k_w)$, 
\item 
the local component $\Pi_u$ is cuspidal,
and is compactly induced from an irreducible represen\-tation $\eta$ of
some compact
{mod centre}
open subgroup $\K_u$ of $\GG(k_u)$.
\end{itemize}
For any finite place $v\in\SS$ such that $v\notin \{u,w\}$,
we fix a compact open subgroup $\K_v$ of $\GG(k_v)$~such that
$\Pi_v$ has a non-zero $\K_v$-invariant vector.
Recall that,
for any finite place $v\notin \SS$,
we have fixed a hyperspecial
maximal compact open subgroup $\K_v$ of $\GG(k_v)$ in
Paragraph \ref{sub:conditions}.
We define 
\begin{equation*}
\K = \prod \limits_{\text{$ v $ finite}} \K_v.
\end{equation*}
This is an 
open subgroup of $\GG(\AA_f)$.

Given an irreducible representation $\kappa$ of $\K_w$, we 
define an irreducible representation $ \La = \La(\kappa)$ of $\K $ by
\begin{equation*}
\La = \bigotimes \limits_{\text{$ v $ finite}} \La_v
\end{equation*}
with
$ \La_w = \kappa $, $ \La_u = \eta $ and
$\La_v$ is the trivial character of $\K_v$ for $ v \notin\{ w, u\} $.

We denote by $ \Alg_\qlb(\GG,\La) $ the subrepresentation of
$\Alg_\qlb(\GG)$ generated
by its $ \La $-isotypic~com\-po\-nent, 
that is,
the subrepresentation formed by the irreducible factors $ \Theta $ such
that
\begin{itemize}
\item
the local component $ \Theta_w $ contains $ \kappa$,
\item
the local component $ \Theta_u $ contains $ \eta$,
\item
the local component $ \Theta_v $ has a non-zero $\K_v $-invariant vector
for all finite $ v \notin\{ w, u\} $.
\end{itemize}
This amounts to considering the space
\begin{equation*}
\Mm =\Mm(\kappa) = \Hom_\K\left(\La, \Alg_\qlb(\GG)\right)
\end{equation*}
seen as a module over the endomorphism $\qlb$-algebra 
$ \Hh_\qlb(\GG,\La) = \Hh_\qlb(\GG(\AA_f),\La) $ 
of the com\-pact induction of $ \La$
from $\K$ to $ \GG(\AA_f) $.
We have an $\Hh_\qlb(\GG,\La) $-module decomposition 
\begin{equation}
\label{decoM}
\Mm = \bigoplus \limits_{\Theta} \Hom_\K(\La, \Theta)
\end{equation}
where $ \Theta $ ranges over the irreducible factors of $ \Alg_\qlb(\GG,\La)$, 
and each $ \Hom_\K(\La, \Theta) $ is of finite~di\-men\-sion as $ \Theta $ is
admissible.
By admissibility again,
the number of $ \Theta $ contributing to the direct sum of \eqref{decoM}
is finite.

Denote by $\X$ the set of finite places of $k$.
For any non-empty subset $\T$ of $\X$
and any irreducible factor $\Theta$ contributing to the right hand side of
\eqref{decoM}, we write
\begin{equation*}
\K_\T = \prod_{v\in\T} \K_v,
\quad
\La_\T = \bigotimes_{v\in\T} \La_v,
\quad
\Theta_\T = \bigotimes_{v\in\T} \Theta_v.
\end{equation*}
We thus have
$ \K = \K_\SS \times \K_{\X \backslash \SS}$,
$\La = \La_\SS \otimes \La_{\X \backslash \SS}$ and
$\Theta$ is isomorphic to $\Theta_\SS\otimes\Theta_{\X \backslash \SS}$. 
Accordingly, we have an isomorphism of $\qlb$-algebras
\begin{equation}
\label{isoh}
\Hh_\qlb(\GG,\La) \simeq
\Hh_\qlb\left(\GG(\AA_\SS), \K_\SS\right)
\otimes \Hh_\qlb\left(\GG(\AA_{\X \backslash \SS}), \La_{\X \backslash \SS}\right)
\end{equation} 
where $\AA_\SS$ and $\AA_{\X \backslash \SS}$ have their obvious
meaning,
and an isomorphism of $ \Hh_\qlb(\La) $-modules 
\begin{equation*}
\Hom_\K(\La, \Theta) \simeq
\Hom_{\K_{\SS}}\left(\La_{\SS}, \Theta_{\SS}\right)
\otimes 
(\Theta_{\X \backslash \SS}) ^{\K_{\X \backslash \SS}} 
\end{equation*}
\textit{via} \eqref{isoh}.
The factor $(\Theta_{\X \backslash \SS})^{\K_{\X \backslash \SS}}$
has dimension $1$ over $ \qlb $ and
$\Hh_\qlb(\GG(\AA_{\X \backslash \SS}), \K_{\X \backslash \SS})$
acts on this
line \textit{via} a character denoted $\chi_\SS(\Theta)$.
Let $d_\SS(\Theta)$ be the dimension of 
$ \Hom_{\K_{\SS}}(\La_{\SS},\Theta_{\SS})$. 
Denoting~by $ \Mm_\SS $ the restriction of $ \Mm $ to
$\Hh_\qlb(\GG(\AA_{\X \backslash \SS}),\K_{\X \backslash \SS}) $,
we therefore have an isomorphism
\begin{equation}
\label{decchiMS2}
\Mm_\SS \simeq \bigoplus_\Theta d_\SS(\Theta) \cdot \chi_\SS(\Theta)
\end{equation}
of $ \Hh_\qlb(\GG(\AA_{\X \backslash \SS}),\K_{\X \backslash \SS}) $-modules.

\subsection{}
\label{Montespan}

Assume now that $ \eta$ is integral. 
Fix a $\K_u$-stable~$\zlb$-lattice $\L_\eta $ of $ \eta$ 
with semi-simple reduction
(by \cite{DatNuTempered} Lemma 6.11).
Since $\K_w$ is compact,
$\kappa$ is integral.
We also fix a $\K_w$-stable $ \zlb$-lattice~$\L_\kappa$~of $ \kappa$
with semi-simple reduction.
If $v$ is a finite place~diffe\-rent from $u,w$,
let $\L_v=\zlb$ be the natural lattice of $\La_v=\qlb$.
Tensoring these $\La_v$ altogether,~we
obtain a $\K$-stable~$\zlb$-lat\-tice $\L_\La $ of $ \La$.
Set
\begin{equation*}
\Mm ^{\circ} = \Hom_\K\left(\L_{\La}, \Alg_\zlb(\GG)\right).
\end{equation*}
It is a module over $\Hh_\zlb(\GG(\AA_f), \L_{\La}) $.
Set $ \overline{\La} = \L_{\La}\otimes_{\zlb}\flb$. 
By \cite{Vigl} Lemme 3.7.3
(which we can apply since $\K$ satisfies
\cite{Vigl} (3.5.1) by \cite{Vigl} Lemme 3.8)
the $\zlb$-module
$ \Mm ^{\circ} $ is a $\zlb$-lattice of $ \Mm$ 
and 
\begin{equation*}
\Mm ^{\circ} \otimes_{\zlb} \flb \simeq
\Hom_\K\left(\overline{\La}, \Alg_\flb(\GG)\right). 
\end{equation*}
We denote the left hand side by $\overline{\Mm}$, 
and we continue to see it as a module over
$\Hh_{\zlb}(\GG(\AA_f),\L_{\La}) $.
Note~that $\overline{\Mm}$ is semi-simple
and depends only on $\r_\ell(\kappa)$,
the reduction mod $\ell$ of $\kappa$.

\subsection{}

We now assume that $\GG(k_w)$ is isomorphic to a
special orthogonal, unitary or symplectic~group.
We do not assume that it is quasi-split,
but we assume that it is not isomorphic to
the~split~spe\-cial~orthogonal group $\SO_2(k_w)\simeq k_w^\times$.
Equivalently,
we assume that it has compact centre~(\cite{MiyauchiStevens} \S 4.2).
Consequently, any cus\-pi\-dal irreducible $\qlb$-representation
of $\GG(k_w)$ is integral.
In the rest of this section, 
we will prove the following theorem.

\begin{theo}
\label{kharevigneras}
Assume that $\GG(k_w)$ is isomorphic to a
special orthogonal, unitary or symplec\-tic group with compact centre,
and that $w$ does not divide $2$.
Let $\Pi$ be an irreducible automorphic representation of $\GG(\AA)$ such that 
\begin{itemize}
\item 
$\Pi_\infty$ is trivial,
\item 
$\Pi_w$ is cuspidal (and integral), 
\item
$\Pi_u$ is integral and compactly induced from
a compact mod centre, open subgroup of $\GG(k_u)$.
\end{itemize}
Let $\pi'$ be an (integral) irreducible cuspidal $\qlb$-representation of
$\GG(k_w)$ such that
\begin{equation}
\label{pipip}
\r_\ell(\Pi_w) \leq \r_\ell(\pi').
\end{equation}
There is an irreducible automorphic representation $\Pi'$ of $\GG(\AA)$
such that
\begin{enumerate}
\item 
the Archimedean component $\Pi'_\infty$ is trivial,
\item
the local component
$\Pi'_{w}$ is isomorphic to $\pi'$,
\item 
the local components
$\Pi'_{u}$ and $\Pi^{\phantom{'}}_u$ are isomorphic,
\item
for any finite place $v\notin\SS$,
the local component $\Pi'_{v}$ is $\K^{\phantom{'}}_v$-unramified,
with Satake parameter 
$\chi'_v : \Hh_{\zlb}(\GG(k_v),\K_v) \to \zlb$,
and 
$\chi^{\phantom{'}}_v$, $\chi'_v$
are congruent mod the~ma\-xi\-mal ideal  of $\zlb$. 
\end{enumerate}
\end{theo}


\begin{rema}
\label{SaintAignan}
Since $\Pi$ is integral, 
it automatically follows from \cite{Vigl} A.3, A.4 that the repre\-sentation
$\Pi_u$
is integral. 
However, 
for clarity, 
we added the integrality assumption in the hypo\-theses of Theorem 
\ref{kharevigneras}. 
\end{rema}

\begin{proof}
We follow the argument of Khare \cite{Khare} and 
Vign\'eras \cite{Vigl}.
We start with following~lem\-ma,
which we will prove in Paragraph \ref{ashes}.

\begin{lemm}
\label{ashesdaurade}
Let $p$ be a prime number different from $2$,
let $F$ be a $p$-adic field and $G$~be~a~spe\-cial orthogonal, 
unitary or symplectic group over $F$.
{Suppose that $G$ has compact centre,
that~is,
$G$ is not isomorphic to the split special orthogonal group
$\SO_2(F)\simeq F^\times$.}
Let $\pi$ and $\pi'$ be (integral) cuspidal $\qlb$-re\-pre\-sentations of $G$ 
such that
\begin{equation*}
\r_\ell(\pi)\<\r_\ell(\pi').
\end{equation*}
There are a compact open subgroup $J$ of $G$ and irreducible 
$\qlb$-representations $\tau$ and $\tau'$ of $J$ such that
$\pi$ is isomorphic to the compact induction of $\tau$ to $G$ and
$\pi'$ is isomorphic to the compact~in\-duction of $\tau$ to $G'$
{and $\r_\ell(\tau)\<\r_\ell(\tau')$.}
\end{lemm}


\begin{rema}
It is known (\cite{ShaunINVENT08}) that any cuspidal $\qlb$-representation 
of $G$ is isomorphic to the compact induction of an (irreducible)
representation of some compact open subgroup of $G$.
The point here is that one can choose the same compact open subgroup
for $\pi$ and $\pi'$.
\end{rema}

Applying Lemma \ref{ashesdaurade} to the group $\GG(k_w)$ and the
cuspidal $\qlb$-representations $\Pi_w$ and $\pi'$, 
we obtain a~com\-pact open subgroup $\K_w$ of $\GG(k_w)$
and irreducible representations $\tau$ and $\tau'$ of $\K_w$
such that $\Pi_w$ is isomorphic to the compact induction of $\tau$ to
$\GG(k_w)$ and
$\pi'$ is isomorphic to the compact induction of $\tau'$ to $\GG(k_w)$. 
Since $\K_w$ is compact,
$\tau$ and $\tau'$ are integral.

Let $\n$ be an irreducible representation of
a compact mod centre,
open subgroup $\K_u$ of $\GG(k_u)$ such~that the
compact induction of $\n$ to $\GG(k_u)$ is isomorphic to $\Pi_u$.
Since $\Pi_u$ is integral,~the~re\-presentation $\eta$ is integral. 
Thus $\Pi$~sa\-tisfies the conditions of Paragraphs \ref{sub:defs}
and \ref{Montespan}.

As in Paragraph \ref{sub:defs},
we define $\Lambda=\Lambda(\tau)$ and $\Mm=\Mm(\tau)$. 
Associated with a choice of $\K_w$-stable $\zlb$-lattice of $\kappa$
with semi-simple reduction, 
there are $\Mm^{\circ}$ and $\overline{\Mm}$.
Similarly,
replacing $\tau$ by~$\tau'$,~we
define $\Lambda'$, $\Mm'$, $\Mm^{\prime \circ}$ and $\overline{\Mm}{}'$.
Recall that $\overline\Mm$ and $ \overline{\Mm}{}'$ are semi-simple. 
The key point is that the space $\overline{\Mm}$ is non-zero and 
contained in $\overline\Mm{}'$,
since we have $\r_\ell(\tau)\<\r_\ell(\tau')$ by 
Lemma \ref{ashesdaurade}.

The character $\chi_\SS(\Pi)$ of
$\Hh_\zlb(\GG(\AA_{\X \backslash \SS}), \K_{\X \backslash \SS})$
defined by $\Pi$ occurs in $\Mm_\SS^\circ$.
By reduction,
we get a~cha\-rac\-ter $\overline{\chi}_\SS$ of
$\Hh_\flb(\GG(\AA_{\X \backslash \SS}), \K_{\X \backslash \SS})$
occurring in $\overline\Mm_\SS$,
and therefore in $ \overline\Mm{}'_\SS$. 

By Deligne-Serre's lemma (\cite{DS} Lemma 6.11),
there is a character $\chi'_\SS$ of
$\Hh_\zlb(\GG(\AA_{\X \backslash \SS}), \K_{\X \backslash \SS})$ occurring in
$\Mm_\SS^{\prime \circ} $ such that its reduction
equals $\overline{\chi}_\SS$. 

Therefore,
there is an irreducible factor $\Pi'$ of $\Alg_\qlb(\GG)$
contributing to $\Mm_\SS^{\prime}$ such that~$\chi^{\phantom{'}}_\SS(\Pi')$
and $\chi'_\SS$ coincide on
$\Hh_\qlb(\GG(\AA_{\X \backslash \SS}), \K_{\X \backslash \SS})$.
Such a $\Pi'$ satisfies the conditions of the theorem. 
\end{proof}

\subsection{} 
\label{ashes}

In remains to prove Lemma \ref{ashesdaurade}. 

\begin{proof}
According to \cite{ShaunINVENT08} Theorem 7.14
(and \cite{MiyauchiStevens} Appendix A),
there are a compact open~sub\-group $J$ of $G$ and an irreducible 
$\qlb$-representation $\tau$ of $J$ such that
$\pi$ is isomorphic to the~com\-pact~induction of $\tau$ to $G$.
More precisely, the pair $(J,\tau)$ can be chosen among
\textit{cuspidal types} of $G$
in the sense of \cite{MiyauchiStevens} Appendix A.
It then has the following properties:
\begin{itemize}
\item
There is a normal pro-$p$-subgroup $J^1$ of $J$ such that
$J/J^1$ is isomorphic to the group of~ra\-tional points of a reductive group
$\mathscr{G}$ (whose neutral component is denoted $\mathscr{G}^\circ$)
defined over the residue field of $F$. 
\item
The representation $\tau$ decomposes as $\kappa\otimes\xi$
where $\kappa$ is a representation of $J$ whose restriction to $J^1$ is 
irreducible 
and $\xi$ is an irreducible representation of $J$ whose
restriction to $J^1$ is trivial. 
\item
More precisely,
$\kappa$ is a \textit{beta-extension}
(\cite{ShaunINVENT08} \S4.1)
of a \textit{skew semisim\-ple~character} $\t$
(\cite{ShaunINVENT08} \S3.1)
defined with respect to a \textit{skew semisimple stratum} $[\La,\b]$
(\cite{ShaunINVENT08} \S2.1)
and $\xi$ is the inflation~of~a~re\-presentation of $J/J^1$
whose restriction to the ra\-tional points of $\mathscr{G}^\circ$ is cus\-pidal. 
\item
The centre of the centralizer $G_E$ of $E=F[\b]$ in $G$ is compact,
and the parahoric subgroup $\P^\circ(\La_E)$ of $G_E$
as\-so\-ciated with $[\La,\b]$ 
(see \cite{ShaunINVENT08} \S2.1)
is a maximal parahoric subgroup of $G_E$.
\end{itemize}

\begin{lemm}
The character $\t$ occurs in $\pi'$.
\end{lemm}

\begin{proof}
By definition,
$\t$ is a character of an open pro-$p$-subgroup 
$H^1=H^1(\La,\b)$ of $G$.
Since~$\t$ occurs in the restriction of $\pi$ to $H^1$,
its reduction mod $\ell$ occurs in $\r_\ell(\pi')|_{H^1}$.
Let $V$ be an~ir\-redu\-cible summand of $\pi'|_{H^1}$ such that $\r_\ell(V)$
contains $\r_\ell(\t)$.
Since $H^1$ is a pro-$p$-group,
$V$~is~iso\-morphic to $\t$.
\end{proof}

Now let $\Cc$ denote the set of pairs $([\La',\b],\t')$
made of a skew semisimple stratum $[\La',\b]$
and a skew semisimple character
$\t'\in\Cc(\La',\b)$ occurring in $\pi'$ such that 
\begin{equation*}
\P^\circ(\La'_E) \subseteq \P^\circ(\La^{\phantom{'}}_E)
\end{equation*}
and $\Cc_{\rm min}$ denote the subset of $\Cc$ made of all 
$([\La',\b],\t')$ such that $\P^\circ(\La'_E)$ is 
minimal among all para\-ho\-ric subgroups of $G_E$ occurring this way.

Let us prove that $\Cc_{\rm min}=\Cc$.
Let $([\La',\b],\t')\in\Cc_{\rm min}$.
Then \cite{ShaunINVENT08} \S7.2
(in particular Lemma~7.4) and
\cite{MiyauchiStevens} Appendix A imply that
$\pi'$ contains a cuspidal type $(J',\kappa'\otimes\xi')$
where
$J'=J(\La',\b)$~for some skew semisimple stratum $[\La',\b]$
and $\kappa'$ is any beta-ex\-tension of $\t'$.
By definition of a~cuspi\-dal type, 
$\P^\circ(\La'_E)$ is a~maximal pa\-rahoric subgroup of $G_E$.
It is thus equal to 
$\P^\circ(\La^{\phantom{'}}_E)$.

It follows that $([\La,\b],\t)\in\Cc_{\rm min}$.
We thus may choose $([\La',\b],\t')=([\La,\b],\t)$
(hence $J'=J$) and $\kappa'=\kappa$ in the paragraph above.
Thus $\pi'$ contains~a~cus\-pi\-dal~type $(J,\kappa\otimes\xi')$.
It follows from \cite{ShaunINVENT08} Corollary 6.19 that
the compact induction of 
$\tau'=\kappa\otimes\xi'$ from $J$ to $G$
is isomorphic to $\pi'$.

The representation $\kappa$ is integral (since the group $J$ is compact)
and its reduction mod $\ell$ is~ir\-reducible
(by \cite{KSCRELLE20} Remark 6.3). 
Applying the functor ${\rm Hom}_{J^1}(\kappa,-)$ from representations of $G$
to representations of $J$ which are trivial on $J^1$,
which is compatible to reduction mod $\ell$,
we deduce from \cite{KSCRELLE20} Corollary 8.5
that $\r_\ell(\xi)\<\r_\ell(\xi')$,
thus $\r_\ell(\tau)\<\r_\ell(\tau')$.
\end{proof}

\section{Globalizing representations} 
\label{secglobalisation}
\label{global3}


In this section, 
we fix a $p$-adic field $F$
and a quasi-split special~ortho\-go\-nal, unitary or symplec\-tic group $G$ 
over $F$.
Let $k$, $w$ and $\GG$ be as in Theorem \ref{GLOBALG},
and $\jmath:\GG(F)\simeq G$ be an~iso\-morphism of locally compact groups
which we use to identify $\GG(F)$ with $G$.

Let $\ell$~denote a prime number different from $p$, 
and fix a field isomorphism $\iota$ as in \eqref{isoiota}. 
Let $u$~be a finite place of~$k$ different from $w$,
not dividing $\ell$.

In Paragraph \ref{GLOBALPIPRIMEpara} only,
the prime number $p$ will be assumed to be odd. 

\subsection{} 
\label{premiereglobalisation}

The next proposition is the first step towards Theorem \ref{GLOBALPIPRIME}.
(See also Paragraph \ref{par13intro}.)

\begin{prop}
\label{GLOBALPI}
Let $\pi$ be a unitary cuspidal irreducible complex representation of $G$,
and let $\rho$~be~a uni\-tary
cuspidal irreducible complex representation~of $\GG(k_u)$.
There~is an irreducible au\-tomor\-phic repre\-sen\-tation $\Pi$ of $\GG(\AA)$ 
such that
\begin{enumerate}
\item 
the local component $\Pi_u$ is isomorphic to $\rho$,
\item 
the local component $\Pi_w$ is isomorphic to $\pi$,
\item 
the local component $\Pi_v$ is the trivial character of $\GG(k_v)$ for any 
real place $v$ of $k$.
\end{enumerate}
\end{prop}

\begin{rema}
When the centre of $G$ is compact, 
any cuspidal irreducible representation $\pi$~of $G$ is unitarizable.
The only case when $G$ has a non-compact 
centre is when it is isomorphic to the split special orthogonal group
$\SO_2(F)\simeq\F^\times$
(see \cite{MiyauchiStevens} 4.2).
\end{rema}

\begin{proof} 
Let $\zz$ be the centre of $\GG$.
We start the proof by the following lemma.
  
\begin{lemm}
\label{lemmeprelim}
There is a unitary automorphic character 
$\Omega:\zz(\AA)/\zz(k)\to\CC^\times$
such that
\begin{enumerate}
\item
the local component $\Omega_u$ is equal to the central character $\omega_\rho$
of $\rho$,
\item 
the local component $\Omega_w$ is equal to the central character
$\omega_{\pi}$ of $\pi$,
\item
the local component $\Omega_v$ is the trivial character of $\zz(k_v)$ for any 
real place $v$ of $k$.
\end{enumerate}
\end{lemm}

\begin{proof}
Let $\U$ denote the subgroup
$\zz(k_{u}) \times \zz(k_{w}) \times \zz(\AA_\infty)$ of $\zz(\AA)$.
The intersection $\U\cap \zz(k)$ is trivial, 
thus $\U$ identifies with a locally compact subgroup of $\zz(\AA)/\zz(k)$.
By Pontryagin duality,
any unitary character of $\U$ extends to $\zz(\AA)/\zz(k)$.
(Note that $\omega_\rho$ and $\omega_{\pi}$ are unitary.)
\end{proof}

We now follow the proof of \cite{HeGL3}.
Let $\Omega:\zz(\AA)/\zz(k)\to\CC^\times$
be a unitary automorphic character as in Lemma \ref{lemmeprelim}.
Let $y$ be a finite place different from $u$ and $w$.

Let us choose coefficients $f_u$ and $f_w$ of $\rho$ and $\pi$,
respectively, which are non-zero at $1$.

For~all real places $v$ of $k$,
let $f_v$ be the constant function equal to $1$ on $\GG(k_v)$.
As this group is~com\-pact,
$f_v$ is smooth and compactly supported.

For all finite places $v\neq y$ such that
$\GG$ is unramified over $k_v$ and
$\Omega_v$ is unramified,
let $f_v$ be the complex function on $\GG(k_v)$ supported on
$\zz(k_v)\K_v$ such that
$f_v(zg) =\Omega_v(z)$ for all $z \in \zz(k_v)$ and all $g$
in a fixed hyperspecial maximal compact subgroup $\K_v$ of $\GG(k_v)$. 

For any other place $x$,
we choose a smooth complex function $f_x$ on $\GG(k_x)$,
non-zero at $1$,~com\-pactly supported modulo $\zz(k_x)$
with restriction to this later
group equal to $\Omega_x$.

We let $f$ be the product of all these $f_v$.
It is smooth and compactly supported on $\GG(\AA)$.~We
may and will assume that
\begin{itemize}
\item
the support of $f_y$ is small enough so that
\begin{equation*}
f(g^{-1}) f(\g g) = 0
\quad
\text{for all $g\in\GG(\AA)$, $\g \in\GG(k)$ such that $\g\notin\zz(k)$,}
\end{equation*}
\item
and $f_v(g)=\overline{f_v(g^{-1})}$ 
for all places $v$ of $k$ and all $g\in\GG(k_v)$. 
\end{itemize}

We construct as in \cite{HeGL3} the Poincar\'e series
\[
P f(g)=\underset{\gamma\in \zz(k) \backslash \GG(k) }{\sum}
f(\gamma g), \qquad \text{ for } g\in \GG(\AA).
\]
We are in a particular case of the proof of
\cite{HeGL3} Appendice 1,
so in particular this is well defined, non-zero,
square-integrable and even cuspidal. 
There is thus an irreducible
automorphic~repre\-sentation $\Pi$ of $\GG(\AA)$,
with central character $\Omega$,
such that, 
for each place $v$ of $k$, one has
\[
\int_{\zz(k_v) \backslash \GG(k_v)}f_v(g^{-1})\Pi_v(g) \ {\rm d}g \neq 0
\]
where ${\rm d}g $ denotes a Haar measure on $\zz(k_v) \backslash \GG(k_v)$.
In particular,
the local components $\Pi_u$ and $\Pi_w$ are isomorphic to
$\rho$ and $\pi$, 
respectively.
At any real place $v$, the representation
$\Pi_v$ contains a vector which is $\GG(k_v)$-invariant,
so $\Pi_v$ is trivial. 
\end{proof}

\subsection{} 
\label{GLOBALPIPRIMEpara}

We now assume that $p\neq2$.

\begin{theo}
\label{GLOBALPIPRIME}
Let $\pi_1$, $\pi_2$ be integral cuspidal irreducible $\qlb$-representations of 
$G$ such that
\begin{equation*}
\r_\ell(\pi_1) \< \r_\ell(\pi_2).
\end{equation*}
Let $\rho$~be~a unitary cuspidal irreducible complex representation of $\GG(k_u)$
which is compactly~indu\-ced from~some compact mod centre,
open subgroup of $\GG(k_u)$. 
{Assume that $G$ is not the split~spe\-cial orthogonal group 
$\SO_2(F)\simeq F^\times$.}
There are irreducible~automor\-phic~re\-pre\-sentations $\Pi_1$
and $\Pi_2$ of $\GG(\AA)$ such that
\begin{enumerate}
\item 
$\Pi_{1,u}$ and $\Pi_{2,u}$ are both isomorphic to $\rho$,
\item 
$\Pi_{1,w}\otimes_{\CC}\qlb$ is isomorphic to $\pi_1$
and $\Pi_{1,w}\otimes_{\CC}\qlb$ is isomorphic to $\pi_2$,
\item 
$\Pi_{1,v}$ and $\Pi_{2,v}$ are trivial for any real place $v$,
\item
there is a finite set $\SS$ of places of $k$, containing all real places, 
such that for all $v\notin\SS$~:
\begin{enumerate}
\item
the local components $\Pi_{1,v}$ and $\Pi_{2,v}$ are unramified
with respect to some hyperspecial maximal compact subgroup $\K_v$ of 
$\GG(k_v)$, 
\item
the restrictions of the Satake parameters 
of $\Pi_{1,v}\otimes_{\CC}\qlb$ and $\Pi_{2,v}\otimes_{\CC}\qlb$
to the Hecke $\zlb$-algebra $\Hh_{\zlb}(\GG(k_v),\K_v)$
are congruent mod the maximal ideal $\m$ of $\zlb$.
\end{enumerate}
\end{enumerate}
\end{theo}

\begin{rema} 
The assumption on $G$ implies that the centre of $G$ is compact,
thus any~cus\-pi\-dal irreducible $\qlb$-representation of $G$ is integral.
\end{rema}

\begin{proof}
First, 
let us apply Proposition \ref{GLOBALPI} with 
$\pi=\pi_1\otimes_{\qlb}\CC$. 
(Since the centre of $G$ is~com\-pact,~the central~cha\-rac\-ter of 
$\pi$ has finite order, thus $\pi$ is unitarizable.) 
We 
obtain an irreducible automor\-phic representation $\Pi_1$
of~$\GG(\AA)$ such that
\begin{enumerate}
\item 
the local component $\Pi_{1,u}$ is isomorphic to $\rho$,
\item 
the local component $\Pi_{1,w} \otimes_{\CC}\qlb$ is isomorphic to 
$\pi_1$, 
\item 
the local component $\Pi_{1,v}$ is the trivial character of 
$\GG(k_v)$ for any real place $v$.  
\end{enumerate}
We then choose for $\SS$ a set of finite places of $k$ as in Paragraph 
\ref{sub:conditions},
that is, 
$\SS$ contains $u$,~$w$~and all places dividing $\ell$, 
and, for any finite place $ v \notin \SS$,
the local component $ \Pi_{1,v}$ is unramified with res\-pect to some
hyperspecial maximal compact subgroup $\K_v$ of $\GG(k_v)$. 
For such $v$,
this defines~a $\zlb$-character $\chi_{1,v}$ of $\Hh_{\zlb}(\GG(k_v),\K_v)$.

We now apply Theorem \ref{kharevigneras} with $\pi'=\pi_2$.
The conditions of Paragraph 3.5 are automatical\-ly satisfied for~$\Pi_{1,w}$
thanks to \cite{ShaunINVENT08}.
We get an irreducible automorphic representation $\Pi_2$ of the group $\GG(\AA)$,
trivial at infinity, such that
\begin{enumerate}
\item
the local component $\Pi_{2,w}\otimes_{\CC}\qlb$ is isomorphic to $\pi_2$,
\item 
the local component $\Pi_{2,u}$ is isomorphic to $\rho$,
\item
for all finite places $v\notin\SS$, 
the local component $\Pi_{2,v}\otimes_{\CC}\qlb$ is $\K_v$-unramified
with associated $\zlb$-character
$\chi_{2,v} : \Hh_{\zlb}(\GG(k_v),\K_v) \to \zlb$,
and $\chi_{1,v}$ and $\chi_{2,v}$ are congruent mod $\m$. 
\end{enumerate}
This proves Theorem \ref{GLOBALPIPRIME}.
\end{proof}

\section{Global transfer}
\label{transfer5} 

\subsection{Quasi-split classical groups}
\label{quasisplitg}

Let $\k$ be either a $p$-adic field for some prime number $p$,
or a real Archimedean local field, or~a totally real number field.
We will~con\-si\-der~the following families of quasi-split reductive groups 
over $\k$:
\begin{enumerate}
\item
For $n\>1$,
the (split) symplectic group $\Sp_{2n}$ defined as $\Sp(f)$,
where $f$ is the alternating form on $k^{2n}\times k^{2n}$ defined by
\begin{equation}
\label{deff}
f(x_1,\dots,x_{2n},y_1,\dots,y_{2n}) = x_1y_{2n}-x_{2n}y_1 +
\dots+x_{n}y_{n+1}- x_{n+1}y_{n}.
\end{equation}
\item 
For $n\>1$ and $\a\in\k^{\times}$,
the (split) special orthogonal group $\SO_{2n+1}$ defined as 
$\SO(q)$,~where $q$ is the qua\-dra\-tic form on $\k^{2n+1}$
of discriminant $(-1)^n\a$ defined by 
\begin{equation}
\label{defqodd}
q(x_1,\dots,x_{2n+1}) = x_1x_2+\dots+x_{2n-1}x_{2n} + \a x_{2n+1}^2.
\end{equation} 
\item
For $n\>1$ and $\a\in\k^{\times}$, 
the special orthogonal group $\SO_{2n}^{\a}$ defined as
$\SO(q)$, where $q$ is~the quadra\-tic form on $\k^{2n}$ of discriminant 
$(-1)^n\a$ defined by 
\begin{equation}
\label{defqeven}
q(x_1,\dots,x_{2n}) = x_1x_2+\dots+x_{2n-3}x_{2n-2} + x_{2n-1}^2-\a x_{2n}^2.
\end{equation} 
\item
For $n\>1$ and $\a\in\k^{\times}$, 
the unitary group $\U_{n}^{\a}$ defined as $\U(h)$,
where $h$ is the $l/k$-Hermitian form on $l^{n}$
of discriminant $(-1)^{n(n-1)/2}$ defined by 
\begin{equation}
\label{defh}
h(x_1,\dots,x_{n}) =
x_1^cx_n^{\phantom{c}}-x_2^cx_{n-1}^{\phantom{c}}+\dots
+(-1)^{n-1}x_{n}^cx_{1}^{\phantom{c}} 
\end{equation}
where $l$ is the $k$-algebra $\k[\X]/(\X^2-\a)$
and $c$ is the non-trivial automorphism of $l/k$.
If $\a\in\k^{\times2}$,
the $k$-group $\U_{n}^{\a}$ is thus isomorphic to $\GL_n$.

\end{enumerate}

In the even orthogonal and unitary cases,
the image of $\a$ in $k^\times/k^{\times2}$
will still be denoted $\a$.

\subsection{The dual group}
\label{par41}

In this paragraph, $\k$ is either a $p$-adic field
or~a~to\-tally real num\-ber field
and $\GG^*$ is one of the quasi-split special orthogonal,
unitary or symplectic $k$-groups of \ref{quasisplitg}.
We define its dual group 
\begin{equation*}
\GGH =
\left\{
\begin{array}{ll}
\SO_{2n+1}(\CC) & \text{if $\GG^*=\Sp_{2n}$}, \\
\Sp_{2n}(\CC) & \text{if $\GG^*=\SO_{2n+1}$}, \\
\SO_{2n}(\CC) & \text{if $\GG^*=\SO_{2n}^{\a}$}, \\
\GL_n(\CC) & \text{if $\GG^*=\U_n^\a$}.
\end{array}
\right. 
\end{equation*}
In the even orthogonal case, the groups 
$\SO_{2n}(\CC)\subseteq{\rm O}_{2n}(\CC)$
are defined with respect to the~sym\-metric bilinear form
$\langle\cdot,\cdot\rangle$ on $\CC^{2n}$ given by
\begin{equation*}
\langle\textsf{e}_i,\textsf{e}_j\rangle=
\left\{
\begin{array}{ll}
0 & \text{if $i+j\neq 2n+1$,} \\
1 & \text{otherwise,}
\end{array}
\right.
\end{equation*}
where $(\textsf{e}_1,\dots,\textsf{e}_{2n})$ is the canonical basis of $\CC^{2n}$.

\subsection{The local Langlands correspondence}
\label{defstd}
\label{defstdU}

In this paragraph, $\k$ is a $p$-adic field
and $\GG^*$ is either the general linear $k$-group $\GL_n$ for~some $n\>1$
(whose dual group is $\GL_n(\CC)$)
or one of the quasi-split classical $k$-groups 
of~\ref{quasisplitg}. 
We~denote by $\W_k$ the Weil group of $\qpb$ over $k$, and 
de\-fi\-ne the semi-direct product
$\GGL=\GGH\rtimes\W_k$, where
\begin{itemize}
\item[$\bullet$]
the action of $\W_k$ on $\GGH$ is trivial when $\GG^*$ is split
(that is, when $\GG^*$ is general linear,~sym\-plec\-tic, odd~or\-tho\-gonal, 
even orthogonal with $\a=1$ or 
unitary with $\a=1$),
\item[$\bullet$]
when $\GG^*$ is even orthogonal and $\a\neq1$, 
and if $l$ denotes the quadratic extension of $k$ in
$\qpb$ generated by a square root of $\a$,
the action~of $\W_k$ on $\GGH$ factors through $\Gal(l/k)$,
the generator $c$ of which
acts by conjugacy by the element $w\in{\rm O}_{2n}(\CC)$ fixing
$\textsf{e}_1,\dots,\textsf{e}_{n-1},\textsf{e}_{n+2},\dots,\textsf{e}_{2n}$
and~ex\-chan\-ging $\textsf{e}_n$ and $\textsf{e}_{n+1}$
(thus $\GGH\rtimes\Gal(l/k)$ identifies with ${\rm O}_{2n}(\CC)$),
\item[$\bullet$]
when $\GG^*$ is unitary and $\a\neq1$, 
and if $l$ denotes the quadratic extension of $k$ in
$\qpb$ generated by a square root of $\a$,
the action~of $\W_k$ on $\GGH$ factors through the group $\Gal(l/k)$ 
whose generator $c$ acts by 
\begin{equation*}
g \mapsto g^* = \J \cdot {}^{\rm t}g^{-1} \cdot \J^{-1}
\end{equation*}
where ${}^{\rm t}g$ denotes the transpose of $g\in\GL_n(\CC)$
and $\J$ is the antidiagonal matrix in $\GL_n(\CC)$~de\-fi\-ned~by
$\J_{i,j}=0$ if $i+j\neq n+1$ and $\J_{i,n+1-i}=(-1)^{i-1}$.
\end{itemize}

Let $\WD_k=\W_k \times \SL_2(\CC)$ denote the Weil-Deligne group of $k$.
A \textit{(local) Langlands parameter} for $\GG(k)$ is a group homomorphism 
\begin{equation*}
\h : \WD_{k} \to \GGH\rtimes\W_k
\end{equation*}
such that
\begin{itemize}
\item 
its restriction to $\W_{k}$ is smooth,
\item 
its restriction to $\SL_2(\CC)$ is algebraic,
\item
the projection of $\h(\W_{k})$ onto $\GGH$ is~made of semi-simple elements, and 
\item
the projection of $\h(w,x)$ onto $\W_k$
is equal to $w$ for all $(w,x)\in\WD_{k}$.
\end{itemize}
When $\GG^*$ is split,
this is the same as a morphism $\WD_{k} \to \GGH$
satisfying the first three points.
In the even orthogonal case with $\a\neq1$, 
this is the same as a morphism $\WD_{k} \to {\rm O}_{2n}(\CC)$
satisfying the first three points and whose determinant is the quadratic character
\begin{equation*}
x\mapsto(\a,x)
\end{equation*}
of $\k^\times$,
which can be seen as a character of $\W_k$ \textit{via} the Artin reciprocity map
of local class field theory. 
We say a local Langlands parameter $\h$ is \textit{bounded}
if $\h(\W_k)$ is relatively compact in $\GGH$.

Let
\begin{itemize}
\item 
$\Phi(\GG^*,k)$ be the set of $\GGH$-conjugacy classes
of local Langlands parameters for $\GG^*$ over $k$,
\item
$\Pi(\GG^*(k))$ be the set of isomorphism classes of irreducible
representations of $\GG^*(k)$.
\end{itemize}

When $\GG^*$ is the general linear group $\GL_n$,
the local Langlands~corres\-pon\-dence 
(\cite{HT,Henniart})
is~a bijection from $\Pi(\GL_n(k))$ to $\Phi(\GL_n,k)$.

When $\GG^*$ is classical, 
the local Langlands correspondence~(\cite{Arthur} Theorem 2.2.1,
\cite{Mok} Theorems 2.5.1, 3.2.1, 
see also \cite{AGRT17} Theorems 3.2, 3.6 and~Re\-marks 3.3, 3.7) 
defines
\begin{enumerate}
\item 
(symplectic, odd orthogonal and unitary cases)
a partition
\begin{equation}
\label{LLCv}
\Pi(\GG^*(k)) = \coprod \limits_{\h\in\Phi(\GG^*,k)} \Pi_\h(\GG^*(k))
\end{equation}
into non-empty finite sets $\Pi_\h(\GG^*(k))$
if $\GG^*$ is symplectic, odd special ortho\-go\-nal or unitary, 
\item
(even orthogonal case)
a partition
\begin{equation}
\label{LLCveven}
\Pi(\SO_{2n}^\a(k))
= \coprod \limits_{\h\in\Phi(\SO_{2n}^\a,k)/{\rm O}_{2n}(\CC)}
\Pi_\h(\SO_{2n}^\a(k))
\end{equation}
where each $\Pi_\h(\SO_{2n}^\a(k))$
is non-empty, finite and stable under ${\rm O}^\a_{2n}(k)$-conjugacy.
\end{enumerate}

In each case, we have the following properties:
\begin{itemize}
\item[$\bullet$]
$\Pi_\h(\GG^*(k))$ contains a tempered repre\-sen\-tation 
if and only if $\h$ is bounded.~When
this~is~the case,
all representations in $\Pi_\h(\GG^*(k))$ are tempered.
(See for instance \cite{Arthur} Theorem 1.5.1 for~sym\-plec\-tic and 
special orthogonal groups,
and \cite{KMSW} Theorem 1.6.1 for unitary groups.)
\item[$\bullet$]
$\Pi_\h(\GG^*(k))$ contains a discrete series representation
if and only if $\h$ is bounded and
the~quo\-tient of the centralizer 
of the ima\-ge of $\h$ in $\GGH$
by ${\bf Z}(\GGH)^{\W_k}$ is finite.
When
this is~the case,~all
representations in $\Pi_\h(\GG^*(k))$ are discrete series representations.
(See for instance \cite{BinXuMAMA17} Theorem 2.2 for~sym\-plec\-tic and 
special orthogonal groups,
and \cite{KMSW} Theorem 1.6.1 for unitary groups.)
\end{itemize}

\subsection{The local transfer}
\label{defloctrans}

In this paragraph, $\k$ is a $p$-adic field
and $\GG^*$ is one of the quasi-split classical $k$-groups~of~\ref{quasisplitg}.
If $\GG^*$ is symplectic or special orthogonal, 
there is a morphism ${\rm Std}:\GGH\to\GL_N(\CC)$ with
\begin{equation}
\label{defN}
N = N(\GG^*) = \left\{
\begin{array}{ll}
2n& \text{if $\GG^*=\SO_{2n+1}$ or $\GG^*=\SO_{2n}^\a$}, \\
2n+1 & \text{if $\GG^*=\Sp_{2n}$},
\end{array}
\right.
\end{equation} 
given by the natural inclusion.
We extend it to a morphism $\Std:\GGH\rtimes\W_k\to\GL_N(\CC)$ 
as follows:
\begin{itemize}
\item[$\bullet$] 
${\rm Std}$ is trivial on $\W_k$ when $\GG^*$ is split, 
\item[$\bullet$]
when $\GG^*$ is even orthogonal and $\a\neq1$, 
${\rm Std}$ is trivial on $\W_l$ and
${\rm Std}(c)=w\in{\rm O}_{2n}(\CC)$,
thus ${\rm Std}$ factors through
\begin{equation*}
\SO_{2n}(\CC) \rtimes\W_k
\twoheadrightarrow \SO_{2n}(\CC) \rtimes\Gal(l/k) 
\simeq {\rm O}_{2n}(\CC) \subseteq \GL_{2n}(\CC)
\end{equation*}
(see also \cite{AGRT17} 3.2). 
\end{itemize}

In the unitary case ($\GG^*=\U_n^\a$),
we need to introduce the $k$-group $\GL_n^{\a}$,
the restriction of $\GL^{\phantom{\a}}_n$ with respect to $l/k$. 
Its dual group is~$\GL_n(\CC)\times\GL_n(\CC)$, 
and we define the semi-direct product
\begin{equation*}
{}^L\GL_n^{\a} =
(\GL^{\phantom{\a}}_n(\CC)\times\GL^{\phantom{\a}}_n(\CC)) \rtimes \W_k
\end{equation*} 
where 
\begin{itemize}
\item[$\bullet$] 
the action of $\W_k$ is trivial when $l/k$ is split,
\item[$\bullet$] 
otherwise, 
the action of $\W_k$ factors through $\Gal(l/k)$ and 
$c$ acts by $(g,h)\mapsto(h,g)$.
\end{itemize}
It will be convenient to set
\begin{equation}
\label{defNU}
N = N(\U_n^\a) = n. 
\end{equation}
Let $\Std$ be the morphism
$\GGH\rtimes\W_k\to
(\GL^{\phantom{\a}}_N(\CC)\times\GL^{\phantom{\a}}_N(\CC)) \rtimes \W_k$
defined by $g\rtimes w\mapsto(g,g^*)\rtimes w$.

Given an irreducible representation $\pi\in\Pi(\GG^*(k))$,
let $\h\in\Phi(\GG^*,k)$ be a Langlands parameter such that
$\pi\in\Pi_\h(\GG^*(k))$.
(In the even orthogonal case,
$\h$ is determined up to ${\rm O}_{2n}(\CC)$-conjugacy only.)

If $\GG^*$ is symplectic or special ortho\-go\-nal, then, 
composing with $\Std$,
we get a local Langlands parameter
$\phi=\Std\circ\h\in\Phi(\GL_N,k)$,
uniquely determined up to $\GL_N(\CC)$-conjugacy.

If $\GG^*$ is unitary, then, 
composing with $\Std$,
we obtain a~Lang\-lands parameter
\begin{equation*}
\Std\circ\h : \WD_k \to 
(\GL^{\phantom{\a}}_N(\CC)\times\GL^{\phantom{\a}}_N(\CC)) \rtimes \W_k.
\end{equation*}
\begin{itemize}
\item[$\bullet$] 
If $l$ is non-split,
its restriction to $\WD_l$ 
has the form $(w,x)\mapsto(\phi(w,x),\phi(w,x)^*)\rtimes w$
for~a~lo\-cal Langlands parameter
$\phi\in\Phi(\GL_N,l)$,
uniquely determined up to $\GL_N(\CC)$-conjugacy. 
\item[$\bullet$] 
If $l$ is split,
it is of the form $(w,x)\mapsto(\phi(w,x),\phi(w,x)^*)\rtimes w$
for a local Langlands parameter
$\phi\in\Phi(\GL_N,k)$,
uniquely determined up to $\GL_N(\CC)$-conjugacy. 
\end{itemize}

\begin{defi}
\label{BCq}
\label{BCu}
The \textit{local transfer} of $\pi$,
denoted $\BC(\pi)$,
is the isomorphism class of irreduci\-ble~represen\-ta\-tions 
associa\-ted with $\phi$ through the local Langlands correspondence.
It is
\begin{enumerate}
\item 
a class of representations of $\GL_N(k)$ if $\GG^*$ is symplectic or special
ortho\-go\-nal,
\item
a class of representations of $\GL_N(l)$ if $\GG^*$ is unitary,
\end{enumerate}  
which is uniquely~de\-termined by the isomorphism class of $\pi$.
\end{defi}

\begin{rema}
\label{remaBCUSPLIT}
If $\GG^*$ is unitary and $l$ is split over $k$,
and if we fix an isomorphism of $k$-algebras $l\simeq k\times k$,
which~we use to identify $\U^\a_n(k)$ with $\GL_n(k)$
and $\GL_N(l)$ with $\GL_N(k)\times\GL_N(k)$,~then
\begin{equation}
\label{BCUSPLIT}
\BC(\pi)=\pi\otimes\pi^\vee
\end{equation}
(where $\pi^\vee$ is the contragredient of $\pi$).
This does not depend on the choice of $l\simeq k\times k$.
Indeed,
making the~other choice twists~the~iso\-mor\-phism
$\U^\a_n(k)\simeq\GL_n(k)$ by $g\mapsto g^*$
(see \eqref{isoUGL} and the explanation thereafter)
and the~isomor\-phism
$\GL_N(l)\simeq\GL_N(k)\times\GL_N(k)$
by $(g,h)\mapsto(h,g)$,
which gives \eqref{BCUSPLIT}
again since $g\mapsto\pi(g^*)$~is isomorphic to $\pi^\vee$.
\end{rema}

In Section \ref{ULT},
we will describe explicitly the local transfer for 
unramified representations when $\GG^*$ is unramified over $k$,
and will describe its congruence properties.

\subsection{Arthur parameters in the symplectic and orthogonal cases}
\label{cappodanno}
\label{sarbacane}

In this paragraph, $\k$ is a~to\-tally real num\-ber field
and $\GG^*$ is symplectic or quasi-split special orthogonal.
We write $\AA$ for the ring of ad\`eles of $k$ and $N=N(\GG^*)$
(see \eqref{defN}).

\begin{defi}
\label{defAparam}
A \textit{{discrete} global Arthur parameter} (for $\GG^*$) is a formal sum
\begin{equation}
\label{defpsi}
\psi = \Pi_1[d_1] \oplus\dots\oplus \Pi_r[d_r]
\end{equation}
for some integer $r\>1$,
where,
for each $i\in\{1,\dots,r\}$,
$d_i$ is a positive integer 
and $\Pi_i$ is a self-dual cuspidal automorphic irreducible representation
of $\GL_{N_i}(\AA)$ for some $N_i\>1$, such that
\begin{enumerate}
\item 
one has $N_1d_1+\dots+N_rd_r=N$,
\item
if $r\>2$ and $\Pi_i\simeq\Pi_j$ for some $i\neq j$ in $\{1,\dots,r\}$, 
then $d_i\neq d_j$, 
\item
the self-dual representation
$\Pi_i$ has the same parity as $\GGH$ if $d_i$ is odd,
and has the opposi\-te~parity if $d_i$ is even, 
where the parity of $\Pi_i$ is defined to be orthogonal if 
$L(s,\Pi_i,{\rm Sym}^2)$ has a pole at $s=1$,
and symplectic if 
$L(s,\Pi_i,\wedge^2)$ has a pole at $s=1$,
\item
the character $\omega_{\Pi_1}^{d_1}\dots\omega_{\Pi_r}^{d_r}$ is
trivial if $\GG^*=\Sp_{2n}$ or $\GG^*=\SO_{2n+1}$,
and is equal to the~qua\-dra\-tic cha\-racter 
\begin{equation*}
\chi_{\a} : x \mapsto \prod\limits_{v} (\a_v,x_v)_v \in \{-1,1\}
\end{equation*}
of $\AA^\times/k ^\times$ 
if $\GG^*=\SO^{\a}_{2n}$,
where $\omega_{\Pi_i}$ is the central character of $\Pi_i$.
\end{enumerate}
\end{defi}

A discrete global Arthur parameter
$\Sigma_1[e_1] \oplus\dots\oplus \Sigma_{s}[e_{s}]$
is said to be \textit{equivalent to} \eqref{defpsi} if~we have
$s=r$ and there is a permutation
$\varepsilon\in\mathfrak{S}_r$ 
such that $e_{i}=d_{\varepsilon(i)}$ and
$\Sigma_i\simeq\Pi_{\varepsilon(i)}$ for each $i$.~Let
\begin{equation*}
\widetilde{\Psi}_{\rm 2}(\GG^*)
\end{equation*}
be the set of equivalence classes of
discrete global Arthur parameters for $\GG^*$.

Associated with a discrete global Arthur parameter
$\psi\in\Psi_{\rm 2}(\GG^*)$ given by \eqref{defpsi}, 
there are~a~lo\-cal Ar\-thur parameter $\psi_v$ 
and a local Arthur packet $\Pi_{\psi_v}(\GG^*(k_v))$ 
for each {finite} place $v$ of $k$:
see \eqref{psiv} and \eqref{multisetApacket} below.

Let $v$ be a finite place of $k$,
and consider the local component $\Pi_{i,v}$ for some $i$. 
It is a unitarisa\-ble~irre\-du\-cible representation of $\GL_{N_i}(k_v)$. 
Associated with it through the local Langlands~cor\-res\-pon\-dence
for $\GL_{N_i}(k_v)$,
there is a local Langlands~pa\-ra\-meter 
\begin{equation*}
\phi_{i,v}:\WD_{k_v}\to\GL_{N_i}(\CC), 
\end{equation*}
uniquely determined up to $\GL_{N_i}(\CC)$-conjugacy.
Since one does not know whe\-ther $\Pi_{i,v}$ is tempe\-red,
the parameter $\phi_{i,v}$ might not be~boun\-ded.

We define a morphism 
\begin{equation}
\label{psiv}
\psi_v = (\phi_{1,v} \boxtimes \SS_{d_1})
\oplus\dots\oplus
(\phi_{r,v} \boxtimes \SS_{d_r}) : 
\WD_{k_v}\times\SL_2(\CC) \to \GL_{N}(\CC)
\end{equation}
where $\SS_d={\rm Sym}^{d-1}$ denotes the unique irreducible algebraic 
representation of $\SL_2(\CC)$ of dimen\-sion $d\>1$.
Recall that we have defined a morphism $\Std$ in \ref{defstd}.
By \cite{Arthur} Theorem 1.4.2,
there~is a local Arthur parameter
$\xi:\WD_{k_v}\times\SL_2(\CC)\to\GGH\rtimes\W_{k_v}$
such that $\psi_v$ is $\GL_N(\CC)$-conjugate to
$\Std\circ\xi$.
The parameter $\xi$ is uniquely~determined~up to $\GGH$-conjugacy,
except if $\GG^*=\SO_{2n}^\a$ and
all $N_1d_1,\dots,N_rd_r$ are even,
in which case there are two $\GGH$-conjugacy classes of such $\xi$.

Associated with $\psi_v$,
there is a multiset $\Pi_{\psi_v}(\GG^*(k_v))$
of irreducible smooth re\-pre\-sen\-ta\-tions of 
$\GG^*(k_v)$, that is, a map
\begin{equation}
\label{multisetApacket}
\Pi(\GG^*(k_v)) \to \ZZ_{\>0}
\end{equation}
with finite support,
where $\Pi(\GG^*(k_v))$ is the set of isomorphism classes of 
irreducible smooth~re\-presentations of $\GG^*(k_v)$.
If $\psi_v(\W_{k_v})$ is relatively compact in $\GL_{N}(\CC)$,
this comes from \cite{Arthur}~Theo\-rems 1.5.1, 2.2.1, 2.2.4
and \eqref{multisetApacket} is supported in the subset 
$\Pi_{\rm unit}(\GG^*(k_v))$ of unitarisable repre\-sen\-tations.
Thanks to Moeglin~(\cite{MoeglinMultiplicite1},
see also \cite{BinXuCJM17} Theorem 8.12),
it does not take any value~$>1$,
that is, $\Pi_{\psi_v}(\GG^*(k_v))$ can be~re\-gar\-ded as 
a finite subset~of $\Pi_{\rm unit}(\GG^*(k_v))$.

When $\psi_v(\W_{k_v})$ is not relatively compact, 
$\Pi_{\psi_v}(\GG^*(k_v))$ 
is obtained from~the relatively compact
case~by a parabo\-lic induction process:
see \cite{Arthur} 1.5 in the symplectic and orthogonal cases
and~\cite{AGRT17} 6.5 in the even orthogonal case.
For our purpose, it will be enough to make the following remark. 

\begin{rema}
\label{psibdd}
Let $v$ be a finite place of $k$,
and assume that $\Pi_{\psi_v}(\GG^*(k_v))$
contains a cus\-pi\-dal representation.
Then $\psi_v(\W_{k_v})$ is relatively compact in $\GL_N(\CC)$.
\end{rema}

When $\psi_v$ is trivial on the $\SL_2(\CC)$-factor,
that is,
$\psi_v$ is a local Langlands parameter for $\GG^*(k_v)$,
the local Arthur packet $\Pi_{\psi_v}(\GG^*(k_v))$
coincides with the $L$-packet associated with $\psi_v$
by the local Langlands correspondence in \eqref{LLCv} and \eqref{LLCveven}.
{(See \cite{AGRT17} top of Paragraph 6.3.)}

\subsection{Transfer}
\label{ST}

In this paragraph,
$k$ is a totally real number field,
$\GG^*$ is one of the quasi-split special orthogo\-nal,
unitary or symplectic $k$-groups
{of \ref{quasisplitg}}
and $\GG$ is an inner form of $\GG^*$ over $k$ 
such that $\GG(k_v)$ is~com\-pact for any real place $v$
and quasi-split for any finite place $v$.

In order to state the following theorem,
we need more than the group $\GG$. 
Following \cite{Taibi}~and \cite{KMSW},
we realize $\GG$ as 
\begin{itemize}
\item[$\bullet$]
a rigid inner twist of $\GG^*$ in the symplectic case (see Paragraph \ref{Sp}),
\item[$\bullet$] 
a pure inner twist of $\GG^*$ in the special orthogonal and unitary cases,
that is,
we fix
a~qua\-dra\-tic form $q$ such that $\GG=\SO(q)$
or a Hermitian form $h$ such that $\GG=\U(h)$.
{(See for instance \cite{BookInvolutions} Sections 29.D, 29.E.)}
\end{itemize}

If $\GG^*$ is special orthogonal,
let $q^*$ be the quadratic form \eqref{defqodd} or \eqref{defqeven} 
such that $\GG^*=\SO(q^*)$,
and let $\a=(-1)^{n(n-1)/2}\d(q^*)$ be its normalized discriminant.
Let $v$ be a finite place of $k$:
\begin{itemize}
\item[$\bullet$] 
if $q\otimes k_v$ is equivalent to $q^*\otimes k_v$,
any choice of $k_v$-isomorphism $f$ such that 
$q=q^*\circ f$ defines a group isomorphism
$\jmath:\GG(k_v)\simeq\GG^*(k_v)$,
and changing $f$ changes $\jmath$ by an inner automorphism, 
which does not affect isomorphism classes of representations of these
groups;
\item[$\bullet$] 
if $q\otimes k_v$ is not equivalent to $q^*\otimes k_v$,
which can only happen when $\GG^*=\SO^\a_{2n}$ with $\a\neq1$,
there is a $\l\in k_v^\times$ 
such that $q\otimes k_v$ is equivalent to $\l\cdot(q^*\otimes k_v)$.
We thus have (canonically upto an inner automorphism)
$\GG(k_v)\simeq\SO(\l\cdot(q^*\otimes k_v))=\GG^*(k_v)$.
\end{itemize}

If $\GG^*$ is unitary, let $h^*$ be the $l/k$-Hermitian form \eqref{defh} 
such that $\GG^*=\U(h^*)$.
Let $v$ be~a~fi\-nite~place of $k$:
\begin{itemize}
\item[$\bullet$] 
if $h\otimes k_v$ is equivalent to $h^*\otimes k_v$,
any choice of isomorphism $f$ such that 
$h=h^*\circ f$ defines a group isomorphism
$\jmath:\GG(k_v)\simeq\GG^*(k_v)$,
and changing $f$ changes $\jmath$ by an inner automorphism,
which does not affect isomorphism classes of representations of these
groups;
\item[$\bullet$] 
if $h\otimes k_v$ is not equivalent to $h^*\otimes k_v$,
which can only happen when $\GG^*=\U^\a_{2n+1}$ with $\a\neq1$,
there is a $\d\in k_v^\times$ 
such that $h\otimes k_v$ is equivalent to $\d\cdot(h^*\otimes k_v)$.
We thus have (canonically up to an inner automorphism)
$\GG(k_v)\simeq\U(\d\cdot(h^*\otimes k_v))=\GG^*(k_v)$.
\end{itemize}

If $\GG^*$ is the symplectic group $\Sp_{2n}$,
then $\GG(k)$ is the group made of all $g\in\Mat_n(D)$ such~that $g^*g=1$, 
where~$D$ is a quaternion $k$-algebra which is split~at~each finite
place and definite ~at~each real place,
and $g^*$ in the matrix whose $(i,j)$-entry is the conjugate of $g_{ji}$.
(See \ref{THEOFS} and \cite{TaibiMRL16} 2.1.1.)
Let $v$ be a finite~place of $k$,
and fix an isomorphism of $k_v$-algebras
$u:D\otimes_k k_v\simeq \Mat_2(k_v)$. 
Through $u$, the group $\GG(k_v)$ identifies with 
$\Sp(f_v)$ for some alternating form $f_v$ on $k_v^{2n}\times k_v^{2n}$.
Changing $u$ changes this identification by an inner automorphism. 
We thus have (canonically up to an~in\-ner automorphism)
a group isomorphism
$\GG(k_v)\simeq\GG^*(k_v)$.

In all cases, 
we have explained how to~ca\-no\-ni\-cally identify representations of
$\GG(k_v)$ with those of $\GG^*(k_v)$.
This thus defines a local transfer for irreducible representations of
$\GG(k_v)$.

\begin{theo}
\label{transfertTaibi}
\label{compatibilitelocalglobal}
Assume that the group $\GG^*$ is symplectic or special orthogonal.
Let $\pi$ be~an~ir\-re\-ducible au\-tomorphic representation 
of $\GG(\AA)$
and suppose~that the\-re is a finite place $u$ of $k$ such that
both the local component $\pi_u$ and its local transfer to $\GL_N(k_u)$
are cuspidal. 
There is a uni\-que
self-dual cuspidal automorphic representation $\Pi$
of $\GL_{N}(\AA)$ such that
\begin{enumerate}
\item 
for all {finite} places $v$ of $k$,
the local transfer of $\pi_v$ to $\GL_N(k_v)$ is $\Pi_v$, 
\item 
for all real places $v$ of $k$,
the infinitesimal character of $\Pi_v$ is algebraic regular.
\end{enumerate}
\end{theo}

\begin{proof}
First note that,
associated with any discrete global Arthur parameter
$\psi\in\Psi_{\rm 2}(\GG^*)$~and
any finite place $v$~of~$k$, there is 
a local Arthur packet $\Pi_{\psi_v}(\GG^*(k_v))$.
We explained how to~ca\-no\-ni\-cally identify representations of
$\GG(k_v)$ with those of $\GG^*(k_v)$.
This thus defines a local Arthur packet $\Pi_{\psi_v}(\GG(k_v))$. 

Now, as $\GG$ is compact at all real places and quasi-split at all finite
places, \cite{Taibi} Theorem 4.0.1 and Remark 4.0.2 apply. 
We thus get a global Arthur parameter $\psi$ for $\GG^*$ such that
\begin{enumerate}
\item 
$\pi_{v} \in \Pi_{\psi_v}(\GG(k_v))$ for all {finite} places $v$ of $k$,
\item 
the infinitesimal character of $\psi_v$ is algebraic regular
for all real places $v$ of $k$.
\end{enumerate} 
In the remainder of the proof,
we follow an argument which has been suggested to us by A.~Moussaoui,
whom we thank for this. 
First, at $v=u$, we have
\begin{equation*}
\pi_u \in \Pi_{\psi_u}(\GG(k_u)) 
\end{equation*}
and it follows from Remark \ref{psibdd} 
that $\psi_u(\W_{k_u})$ is relatively compact in $\GL_N(\CC)$. 
Associa\-ted~with $\psi_u$ in \cite{MoeglinLIE09} 4.1, 
there is its~\textit{ex\-ten\-ded cuspidal support} 
(or \textit{infinitesimal character}),
denoted $\l_u$.
It is the $N$-dimensional~re\-pre\-sen\-tation of $\W_{k_u}$
defined by
\begin{equation*}
\l_u(w) = \psi_u(w,\textsf{d}_w,\textsf{d}_w),
\quad
\textsf{d}_w =
\begin{pmatrix}
|w|^{1/2} & 0 \\ 0 & |w|^{-1/2}
\end{pmatrix}
\in\SL_2(\CC),
\quad
w\in\W_{k_u},
\end{equation*}
where $w\mapsto|w|$ is the character $\W_{k_u}\to\RR^\times_{+}$
defined by $|w|=q^{-v(w)}$,
where $q$ is the cardinality of the residue field of $k_u$ and
$v(w)\in\ZZ$ is the valuation of $w$,
normalized so that any geometric Frobenius element
has valuation $1$.
If we write explicitly
\begin{equation*}
\psi_u = \bigoplus_{i=1}^{m} \s_{i} \boxtimes \SS_{a_i} \boxtimes \SS_{b_i}
\end{equation*}
for some $m\>1$,
with $a_i,b_i\>1$ and where $\s_i$ is an irreducible representation of 
$\W_{k_u}$, then 
\begin{equation}
\label{expandlu}
\l_u = \bigoplus_{i=1}^{m} \bigoplus_{j=0}^{b_i-1} \bigoplus_{k=0}^{a_i-1}
\s_{i} |\cdot|^{(b_i-1)/2+(a_i-1)/2-j-k}. 
\end{equation}
On the other hand,
by \cite{MoeglinLIE09} 4.1 again, 
the \textit{extended cuspidal support} 
(or \textit{infinitesimal character})~of $\pi_u$ is 
the repre\-sen\-tation $\l$ of $\W_{k_u}$ defined by
$\l(w) = \phi(w,\textsf{d}_w)$ for all $w\in\W_{k_u}$,
where $\phi=\Std\circ\h$ and $\h$ is the Langlands parameter
associated with $\pi_u$
(up to ${\rm O}_{2n}(\CC)$-conjugacy in the even~ortho\-gonal case).
Given the~as\-sump\-tion that we made on $\pi_u$,
the extended cuspidal support $\l$ is~irre\-du\-cible.
By \cite{MoeglinLIE09} Proposition 4.1,~the
extended~cus\-pidal supports of $\psi_u$ and $\pi_u$
coincide.
It~fol\-lows that \eqref{expandlu} is ir\-reducible,
which~im\-plies that $m=1$ and $a_1=b_1=1$.
Thus $\psi$ satisfies $r=1$ and $d_1=1$.

We thus have $\psi=\Pi[1]$ for a uniquely determined self-dual
cuspidal automorphic irreducible representation
$\Pi$ of $\GL_N(\AA)$. 
Given a finite place $v$ of $k$,
the local component $\pi_v$ is in the Arthur packet
$\Pi_{\psi_v}(\GG(k_v))$.
Since $\psi_v$ is a Langlands parameter (as $d_1=1$),
this Arthur packet is an $L$-packet,
thus $\Pi_v$ is the local transfer of $\pi_v$ to $\GL_N(k_v)$.
\end{proof}

We now consider the case of unitary groups.

\begin{theo}
\label{transfertLabesse}
Assume that the group $\GG^*$ is unitary. 
Let $\pi$ be an irreducible automorphic~re\-presentation of $\GG(\AA)$,
and suppose~that the\-re is a finite place $u$ of $k$ such that
$\GG(k_u)$ is split and $\pi_u$ is cuspidal. 
There exists~a~uni\-que
conjugate-self-dual cuspidal automorphic representation $\Pi$
of $\GL_{N}(\AA_l)$ such that
\begin{enumerate}
\item 
for all {finite} places $v$ of $k$,
the local transfer of $\pi_v$ to $\GL_N(l_v)$ is $\Pi_v$, 
\item 
for all real places $v$ of $k$,
the infinitesimal character of $\Pi_v$ is algebraic regular.
\end{enumerate}
\end{theo}

\begin{proof} 
Since $\GG$ is compact at all real places,
the assumptions of \cite{Labesse} Corollaire 5.3 are~sa\-tis\-fied~(see
the paragraph following \cite{Labesse} Remarque 5.2 regarding 
Property $(*)$).
By \cite{Labesse}~Co\-rol\-laire 5.3,
there is an integer $r\>1$ and,
for each $i\in\{1,\dots,r\}$,
there is a conjugate-self-dual discrete automorphic representation
$\Pi_i$ of $\GL_{N_i}(\AA_l)$ for some $N_i\>1$, such that
\begin{itemize}
\item 
one has $N_1+\dots+N_r=N$,
\item
if $\Pi$
is the irreducible automorphic representation of $\GL_N(\AA_l)$ obtained by
parabolic induction from $\Pi_1 \otimes \dots \otimes \Pi_r$,
then $\Pi_v$ is the local transfer of $\pi_v$ for all finite places $v$ 
which are either unramified or split.
(The local base change of \cite{Labesse} is the same as the local
transfer of Paragraph \ref{defloctrans}: see \cite{Labesse} 4.10.) 
\end{itemize}
In particular, 
for $v=u$,
the group $\GG(k_u)$ is split, 
thus $\Pi_u^{}$ is isomorphic to $\pi_u^{} \otimes\pi_u^\vee$
\textit{via} the~choice of a $k_u$-algebra isomorphism 
$l_u \simeq k_u\times k_u$
(see Remark \ref{remaBCUSPLIT}). 
Since $\pi_u$ is cuspidal, $\Pi_u$ is~cuspi\-dal~as well.
It follows that $r=1$ and $\psi$ is cuspidal.  
By \cite{Labesse} Th\'eor\`eme 5.9, 
we get that
\begin{itemize}
\item[$\bullet$]
$\Pi_v$ is the base change of the trivial character of $\GG(\k_v)$,
thus its infinitesimal character~is~al\-gebraic regular, 
for all real places $v$ of $k$,
\item[$\bullet$]
and the local transfer of $\pi_v$ to $\GL_N(l_v)$ is $\Pi_v$
for all {finite} places $v$ of $k$.
\end{itemize}
This finishes the proof of Theorem \ref{transfertLabesse}.
\end{proof}

\section{Unramified local transfer}
\label{ULT}

In this section,
we examine the congruence properties of the local transfer
(as defined in~Pa\-ra\-graph \ref{defloctrans})
for unramified representations of unramified classical groups. 

\subsection{}
\label{Clery}

Let $F$ be a non-Archimedean locally compact field of residue
characteristic $p$, 
and $G$ be~the group of rational points of 
an unramified reductive group $\GG$ defined over $F$.
Let
${\bf S}$ be a maximal $F$-split torus in $\GG$,
$\TT$ be the centralizer of ${\bf S}$ in $\GG$ and
$K$ be a~hy\-per\-spe\-cial maximal compact subgroup of $G$
corresponding to a hy\-per\-spe\-cial point in the apartment associated with
${\bf S}$ in the reduced Bruhat-Tits building of $(\GG,F)$. 
Let $W$ be the Weyl group associated with $T=\TT(F)$~and
$\Li$ be the $\ZZ$-lattice $T/(T\cap K)$. 
We have the Satake isomorphism (\cite{Satake}) of $\CC$-algebras
\begin{eqnarray*}
\CC[K\backslash G \slash K] &\to& \CC[\Li]^W \\
f &\mapsto& \left(t \mapsto \d^{1/2}(t)\int\limits_{U} f(tu)\ {\rm d}u \right)
\end{eqnarray*}
where $U$ is the group of rational points of the unipotent radical of a Borel
subgroup ${\bf B}=\TT{\bf U}$~of $\GG$,
${\rm d}u$ is the Haar measure on $U$ giving measure $1$ to $U\cap K$
and $\d^{1/2}$ is the square root of the mo\-du\-lus character $\d$ of $B={\bf B}(F)$
defined with respect to the positive square root $\sqrt{q}\in\RR_{>0}$~of $q$,
the cardinality of the residue field of $F$.

The same formula applies when one replaces $\CC$ by $\qlb$.
We then get a Satake isomorphism of $\qlb$-algebras
$\qlb[K\backslash G \slash K] \to \qlb[\Li]^W$
depending on the choice of a square root $q^{1/2}$
of~$q$~in~$\qlb$.~By
\cite{HenniartVignerasCRELLE} \S7.10--15,
as this square root and its inverse are contained in $\zlb$, 
this isomorphism in\-du\-ces by restriction an isomorphism 
\begin{equation}
\label{bilanof}
\zlb[K\backslash G \slash K] \to \zlb[\Li]^W
\end{equation}
of $\zlb$-algebras.

\subsection{}
\label{par62}

Let $\pi$ be a $K$-unramified irreducible $\qlb$-representation~of~$G$, 
that is, $\pi$ has a non-zero $K$-fi\-xed~vector.
Recall that its \textit{Satake parameter} is
the character $\chi$ of
$\qlb[K\backslash G \slash K]$ 
through which this algebra acts on the $1$-di\-men\-sional 
space $\pi^K$ of $K$-invariant vectors of $\pi$.
Through the Satake isomorphism,
it defines~a cha\-rac\-ter of $\qlb[\Li]^W$.
Such a character is of the form 
\begin{equation}
\label{omom}
f \mapsto \int\limits_{T} f(t) \omega(t)\ {\rm d}t
\end{equation}
for some unramified $\qlb$-character $\omega$ of $T$
-- which we may consider as a cha\-racter of $\Li$ -- 
uniquely determined up to $W$-conjugacy. 
(Here ${\rm d}t$ is the Haar measure giving measure $1$ to $T\cap K$.)~{By
\cite{Satake},
the $W$-conjugacy class of $\omega$ 
is the cuspidal support of $\pi$,
that is, $\pi$ occurs as an irreducible component of the
representation obtained by parabolically inducing $\omega$ to $G$
along $B$, 
\textit{where~pa\-ra\-bolic induction is normalized by the same square root of
the $\qlb$-modulus $\d$ as the one used to define the Satake $\qlb$-isomorphism.}}

Now assume that the restriction of $\chi$ to $\zlb[K\backslash G \slash K]$
has values in $\zlb$.
Thanks to \eqref{bilanof},~it~defi\-nes a $\zlb$-cha\-rac\-ter of 
$\zlb[\Li]^W$, still denoted $\chi$.
Let us prove that 
$\omega$ has values in~$\overline{\mathbb{Z}}{}_\ell^\times$. 
For~this,~let $\mu$~be the $\qlb$-cha\-rac\-ter 
of $\qlb[\Li]$ defined by \eqref{omom}.
Its restriction to $\zlb[\Li]^W$ is equal to $\chi$.
According to \cite{BourbakiAC} Chapter 5, \S1, n°9, Proposition 22, 
the ring $\zlb[\Li]$ is integral over $\zlb[\Li]^W$.
As $\chi$~takes~va\-lues~in $\zlb$ on $\zlb[\Li]^W$,
and as $\zlb$ is integrally closed,
it follows that $\mu$ takes values in $\zlb$ on $\zlb[\Li]$.
By~eva\-lua\-ting $\mu$ at the characteristic function of any $\l\in\Li$,
we get $\omega(\l)\in\zlb$.
So far,
we proved the following result.

\begin{prop}
\label{enonceULT0}
Let $G$ be the group of rational points of 
an unramified group defined over $F$, 
let $K$ be a~hy\-per\-spe\-cial maximal compact subgroup of $G$
and $\pi$ be a $K$-unramified 
$\qlb$-re\-pre\-sen\-ta\-tion of $G$ with Satake parameter
$\chi$.
Then $\pi$ is integral if and only if $\chi$ is integral
(that is, it~ta\-kes integral values on $\zlb[K\backslash G \slash K]$).
\end{prop}

\begin{proof}
Indeed, 
using the notation above, 
the cuspidal support of $\pi$ is the $W$-conjugacy class of the
unramified character $\omega$ of $T$,
and $\pi$ is integral if and only if $\omega$ is.
(For this latter fact, see
\cite{DHKMfiniteness} Corollary 1.6.)
\end{proof}

Finally,
assume that $\chi_1$ and $\chi_2$ are congruent
$\zlb$-characters of $\zlb[K\backslash G \slash K]$.
One can~see~them
\textit{via} \eqref{bilanof}
as congruent $\zlb$-characters of $\zlb[\Li]^W$,
still denoted $\chi_1$ and $\chi_2$. 
For $i=1,2$,~let~$\mu_i$~be~a character
of $\zlb[\Li]$ extending $\chi_i$.
It takes the form \eqref{omom} for a uniquely determined
unramified~cha\-rac\-ter $\omega_i$ of $T$,
which is integral thanks to the previous paragraph. 
Reducing~mod~the~maxi\-mal ideal of $\zlb$,
the characters $\mu_1$ and $\mu_2$ define
$\flb$-characters $\overline{\mu}_1$ and $\overline{\mu}_2$
of $\flb[\Li]$ which, by as\-sump\-tion,
coincide on $\flb[\Li]^W$.
Applying the corollary of
\cite{BourbakiAC} Chapter 5, \S2, n°2, Theorem 2, 
it follows that
the~cha\-racters $\r_\ell(\omega_1)$ and $\r_\ell(\omega_2)$ 
are $W$-conjugate.
We thus proved:

\begin{prop}
\label{enonceULT1}
Let $G$ be the group of rational points of 
an unramified group defined over $F$, 
let $K$ be a~hy\-per\-spe\-cial maximal compact subgroup of $G$,
let $\pi_1$ and $\pi_2$ be $K$-unramified~irredu\-ci\-ble
$\qlb$-re\-pre\-sen\-ta\-tions of $G$ whose Satake parameters
$\chi_1$ and $\chi_2$ define congruent $\zlb$-characters of 
$\zlb[K\backslash G \slash K]$ and let $\omega_1$ and $\omega_2$
be unramified $\qlb$-characters of $T$ such that $\pi_i$
occurs in the~pa\-ra\-bolic induction of $\omega_i$ to $G$ along $B$,
for $i=1,2$.
Then 
$\r_\ell(\omega_1)$ and $\r_\ell(\omega_2)$
are $W$-conjugate.
\end{prop}

\subsection{}
\label{NRCLASS}

From now on and until the end of this section,
we assume that $\GG$ is an unramified special~or\-tho\-gonal,
unitary
or~sym\-plectic group among the groups of Paragraph \ref{quasisplitg}.
The associated dual group $\GGH$ has been defined in
Paragraph \ref{par41}.
Recall that $G=\GG(F)$.

Let $\pi$ be an integral $K$-unramified $\qlb$-representation~of~$G$.
Its cuspidal support is the $W$-or\-bit of an unramified $\zlb$-cha\-rac\-ter
$\omega$ of $T$.
Its Satake parameter is a character
$\chi:\zlb[K\backslash G \slash K]\to\zlb$.
They are 
related through the Satake isomorphism by the formula \eqref{omom}.

Restriction from $T$ to $S={\bf S}(F)$ induces an isomorphism 
$\Li \simeq S/(S\cap K)$,
thus between~un\-ra\-mified characters of $T$ and 
unramified characters of $S$.
The later is the dual group $\widehat{\MM}(\qlb)$.

Let $\Phi$ be a Frobenius element in the Weil group $\W_F$.
By \cite{BorelCorvallis} 6.4, 6.5,
the surjection of $\widehat{\TT}(\qlb)$ onto
$\widehat{\MM}(\qlb)$~in\-du\-ces a bijection between
\begin{itemize}
\item 
$N$-conjugacy classes in $\widehat{\TT}(\qlb)\rtimes\Phi$, and 
\item
$W$-conjugacy classes in $\widehat{\MM}(\qlb)$,
\end{itemize}
where $N$ is the inverse image of $W$ in the normalizer
of $\widehat{\TT}(\qlb)$ in $\GGH(\qlb)$,
and the embedding of of $\widehat{\TT}(\qlb)$
in $\GGH(\qlb)$ induces a bijection between
\begin{itemize}
\item 
$N$-conjugacy classes in $\widehat{\TT}(\qlb)\rtimes\Phi$, and 
\item
$\GGH (\qlb)$-conjugacy classes of semi-simple elements in
$\GGH (\qlb) \rtimes\Phi$.
\end{itemize}
The $W$-orbit of $\omega$ thus determines
the $W$-conjugacy class of a point 
$s\in\widehat{\MM}(\zlb)$,
then the $\GGH(\qlb)$-conju\-gacy class of 
a semi-simple element $t\rtimes\Phi\in\GGH (\qlb) \rtimes\Phi$.
We are going to prove that $t\rtimes\Phi$ may be chosen in
$\widehat{\TT}(\zlb)\rtimes\Phi\subseteq\GGH(\zlb)\rtimes\Phi$.
Let us fix a uniformizer $\w$ of $F$.

When~$\GG$~is split,
we have $\TT=\MM$,
thus $t=s$ is in $\widehat{\TT}(\zlb)\subseteq\GGH(\zlb)$.
Explicitly,
if we identify~$T$ with $(F^\times)^m$ for some integer $m\>1$,
then $\omega$ identifies with the tensor product of $m$
unramified characters 
$\omega_1,\dots,\omega_m$ of $F^\times$
and $t\rtimes\Phi$ is $\GGH(\qlb)$-conjugate to
\begin{itemize}
\item[$\bullet$]
${\rm diag}(\omega_1(\w),\dots,\omega_m(\w),1,
\omega_m(\w)^{-1},\dots,\omega_1(\w)^{-1})\in\GL_{2n+1}(\zlb)$
if $\GG=\Sp_{2n}$, 
\item[$\bullet$]
${\rm diag}(\omega_1(\w),\dots,\omega_m(\w),
\omega_m(\w)^{-1},\dots,\omega_1(\w)^{-1})\in\GL_{2n}(\zlb)$ 
if $\GG=\SO_{2n+1}$
or $\GG=\SO_{2n}^1$, 
\end{itemize}
with $m=n$ in all cases. 

Now assume that $\GG$ is non-split,
thus either $\GG=\SO_{2n}^\a$ or $\GG=\U_n^\a$,
with $\a\neq1$.
\begin{itemize}
\item[$\bullet$] 
In the even orthogonal case, 
we have $S\simeq (F^\times)^{m}$ and $T\simeq S\times\SO_2^\a(F)$
with $m=n-1$ (see \cite{BorelLAG} \S23.4),
thus $\widehat{\TT}(\qlb)\simeq\overline{\QQ}{}_\ell^{m+1}$~sur\-jects
onto $\widehat{\MM}(\qlb)\simeq\overline{\QQ}{}_\ell^{m}$ through 
\begin{equation*}
(t_1,t_2,\dots,t_{m+1}) \mapsto 
(t_1,t_2,\dots,t_{m})
\end{equation*}
and $\widehat{\TT}(\qlb)\rtimes\W_F$
embeds into $\GGH(\qlb)\rtimes\W_F$ through
\begin{equation*}
(t_1,t_2,\dots,t_{n}) \rtimes w \mapsto {\rm diag}(t_1^{\phantom{1}},\dots,
t_{n}^{\phantom{1}},t_{n}^{-1},\dots,t_1^{-1}) \rtimes w
\end{equation*}
thus the image of $t\rtimes\Phi$ in $\GGH(\qlb)\rtimes\Phi$ is
$\GGH(\qlb)$-conjugate to
\begin{equation*}
{{\rm diag}(\omega_1(\w),\dots,\omega_m(\w),1,1,
\omega_m(\w)^{-1},\dots,\omega_1(\w)^{-1})\rtimes\Phi
\in\GL_{2n}(\zlb)\rtimes\Phi.}
\end{equation*}
\item[$\bullet$] 
In the unitary case,
we have $S\simeq (F^\times)^m$ and $T\simeq (E^\times)^m$
where $E$ is the quadratic extension of $F$ generated by a square root of 
$\a$ (note that it is unramified since $\GG$ is assumed to be~unra\-mi\-fied), 
and $m=\lfloor n/2\rfloor$ is the Witt~in\-dex of $\GG$
(see \cite{BorelLAG} \S23.9), thus
$\widehat{\TT}(\qlb)\simeq\overline{\QQ}{}_\ell^{2m}$
surjects onto $\widehat{\MM}(\qlb)\simeq\overline{\QQ}{}_\ell^{m}$
through 
\begin{equation*}
(t_1,t_2,\dots,t_{2m}) \mapsto (t_1t_{2m},t_2t_{2m-1},\dots, t_{m}t_{m+1})
\end{equation*}
and $\widehat{\TT}(\qlb)\rtimes\W_F$
embeds into $\GGH(\qlb)\rtimes\W_F$ through
\begin{equation*}
(t_1,t_2,\dots,t_{2m}) \rtimes w \mapsto
\left\{
\begin{array}{ll}
{\rm diag}(t_1^{\phantom{1}},\dots,
t_{m}^{\phantom{1}},t_{m+1}^{-1},\dots,t_{2m}^{-1}) \rtimes w
& \text{if $n=2m$ is even}, \\ 
{\rm diag}(t_1^{\phantom{1}},\dots,
t_{m}^{\phantom{1}},1,t_{m+1}^{-1},\dots,t_{2m}^{-1}) \rtimes w
& \text{if $n=2m+1$ is odd},
\end{array}
\right.
\end{equation*}
thus the image of $t\rtimes\Phi$ in $\GGH(\qlb)\rtimes\Phi$
is $\GGH(\qlb)$-conjugate to
\begin{equation*}
\left\{
\begin{array}{ll}
{\rm diag}(\omega_1(\w)^{1/2},\dots,\omega_m(\w)^{1/2},
\omega_m(\w)^{-1/2},\dots,\omega_1(\w)^{-1/2}) \rtimes\Phi
& \text{if $n=2m$ is even}, \\ 
{\rm diag}(\omega_1(\w)^{1/2},\dots,\omega_m(\w)^{1/2},1,
\omega_m(\w)^{-1/2},\dots,\omega_1(\w)^{-1/2}) \rtimes\Phi
& \text{if $n=2m+1$ is odd},
\end{array}
\right.
\end{equation*}
which is in $\GL_{n}(\zlb) \rtimes\Phi$ in both cases. 
\end{itemize} 

We now define an unramified local Langlands parameter
$\h:\WD_F\to \GGH(\zlb)\rtimes\W_F$
by
\begin{itemize}
\item 
$\h(\Phi)=t\rtimes\Phi$, and 
\item
$\h$ is trivial on the inertia subgroup $I_F$ of $\W_F$ and on $\SL_2(\CC)$.
\end{itemize}
It is uniquely determined by the $K$-unramified representation $\pi$,
or equivalently by its Satake~pa\-ra\-me\-ter $\chi$.
Com\-posing with $\Std$
(or just restricting to $\WD_E$ in the unitary case), 
we get~an~un\-rami\-fied Langlands parameter
$\phi\in\Phi(\GL_N,E)$,
where $E=F$ in the symplectic and orthogonal ca\-ses
and $E$ is the quadratic extension of $F$ generated by a square root
of $\a$ in the unitary~case. 
This $\phi$ uniquely determines an 
unramified $\qlb$-representation of $\GL_N(E)$,
denoted $\BC_\ell(\pi)$.

In the symplectic and special orthogonal cases,
$\BC_\ell(\pi)$ is the~uni\-que unramified irreducible~com\-po\-nent of 
\begin{equation*}
\left\{
\begin{array}{ll}
\omega_1^{\phantom{1}}\times \dots\times
\omega_n^{\phantom{1}}\times\omega_n^{-1}\times
\dots\times\omega_1^{-1}
& \text{if $\GG=\SO_{2n+1}$ or $\GG=\SO_{2n}^1$}, \\
\omega_1^{\phantom{1}}\times \dots\times
\omega_n^{\phantom{1}}\times 1\times\omega_n^{-1}\times
\dots\times\omega_1^{-1}
& \text{if $\GG=\Sp_{2n}$}, \\
\omega_1^{\phantom{1}}\times \dots\times
\omega_{n-1}^{\phantom{1}}\times 1 \times 1 \times \omega_{n-1}^{-1}\times
\dots\times\omega_1^{-1}
& \text{if $\GG=\SO_{2n}^\a$ with $\a\neq1$},
\end{array}
\right.
\end{equation*}
where $\times$ denotes the parabolic induction to $\GL_N(F)$ normalized
with respect to $q^{1/2}$.

In the unitary case,
the Weil group $\W_E$ is generated by $\Phi^2$ and $I_F$
(since $E/F$ is unramified),
thus $\phi$ is uniquely determined by
$\phi(\Phi^2)=tt^*\rtimes\Phi^2$,
with
\begin{equation*}
tt^*=
\left\{
\begin{array}{ll}
{\rm diag}(\omega_1(\w),\dots,\omega_m(\w),
\omega_m(\w)^{-1},\dots,\omega_1(\w)^{-1}) 
& \text{if $n=2m$ is even}, \\ 
{\rm diag}(\omega_1(\w),\dots,\omega_m(\w),1,
\omega_m(\w)^{-1},\dots,\omega_1(\w)^{-1}) 
& \text{if $n=2m+1$ is odd},
\end{array}
\right.
\end{equation*}
which gives $\BC_\ell(\pi)$ explicitly.
Namely, $\BC_\ell(\pi)$ is the unique unramified irreducible component of
\begin{equation*}
\left\{
\begin{array}{ll}
\omega_1^{\phantom{1}}\times \dots\times
\omega_m^{\phantom{1}}\times\omega_m^{-1}\times
\dots\times\omega_1^{-1}
& \text{if $\GG=\U_{2m}^\a$}, \\ 
\omega_1^{\phantom{1}}\times \dots\times
\omega_m^{\phantom{1}}\times 1\times\omega_m^{-1}\times
\dots\times\omega_1^{-1}
& \text{if $\GG=\U_{2m+1}^\a$},
\end{array}
\right.
\end{equation*}
where $\times$ denotes the parabolic induction to $\GL_n(E)$ normalized
with respect to $(q^{1/2})^{2}=q$~(as 
$E$ is quadratic and unramified over $F$).
(See also \cite{AlbertoLivre}.)
We have:

\begin{prop}
\label{enonceULT}
Let $G$ be the group of rational points of 
an unramified special orthogonal, unitary or symplectic $F$-group 
among the groups of Paragraph \ref{quasisplitg}.
Let $K$ be a~hy\-perspecial~ma\-xi\-mal compact subgroup of $G$
and let $\pi_1$, $\pi_2$ be $K$-unramified irreducible
$\qlb$-re\-pre\-sen\-ta\-tions of~$G$ whose Satake parameters
$\chi_1$, $\chi_2$ define congruent $\zlb$-characters of 
$\zlb[K\backslash G \slash K]$.
Then
\begin{enumerate}
\item 
the representations $\BC_\ell(\pi_1)$ and $\BC_\ell(\pi_2)$ 
of $\GL_N(E)$ are integral,
\item
their Langlands parameters are integral and congruent.
\end{enumerate} 
\end{prop}

\begin{rema}
\label{VarNR}
Note that the reductions mod $\ell$ of $\BC_\ell(\pi_1)$ and $\BC_\ell(\pi_2)$ 
may not have any irredu\-ci\-ble component in common.
However,
if $\tau_i$ denotes the unique unramified irreducible component of
$\r_\ell(\BC_\ell(\pi_i))$ 
for $i=1,2$,
then $\tau_1$ and $\tau_2$ have the same cuspidal support.

For instance,
assume that $\GG=\SO_5$ and $\ell$ divides $q^2-1$. 
Let $\pi_1$ (resp. $\pi_2$)
be the unramified representation of $G=\SO_5(F)$
with respect to some hyperspecial maximal compact group,
with cuspidal support the $W$-conjugacy class of
$|\cdot|^{-1/2}\otimes|\cdot|^{1/2}$
(resp. $|\cdot|^{3/2}\otimes|\cdot|^{1/2}$),
where~$|\cdot|$~is the~abso\-lute~value of $F^\times$. 
By assumption,
these cuspidal~sup\-ports are congruent 
(for $|\cdot|^{3/2}$~and $|\cdot|^{-1/2}$~have the same reduction mod 
$\ell$).  
Then
$\BC_\ell(\pi_1)$ is the unique~un\-ra\-mi\-fied~irreducible~compo\-nent of
\begin{equation*}
|\cdot|^{-1/2}\times|\cdot|^{1/2}\times|\cdot|^{-1/2}\times|\cdot|^{1/2}
\end{equation*}
that is, $\BC_\ell(\pi_1)=1_2\times 1_2$
(where $1_2$ is the trivial character of $\GL_2(F)$).
Similarly, $\BC_\ell(\pi_2)$ is equal to $|\det|\times|\det|$
(where $\det$ is the determinant of $\GL_2(F)$).
Now assume further that $\ell\neq2$~and $\ell$ divides $q+1$,
thus $\ell$ does not divide $q-1$.
Then $\r_\ell(\BC_\ell(\pi_1))$ and $\r_\ell(\BC_\ell(\pi_2))$ are both
irreducible and twists of each other by the non-trivial character $|\cdot|$, 
thus non-isomorphic. 
\end{rema}

\begin{rema}
\label{rossmiller}
Let $\ii$ be an isomorphism of fields $\CC\to\qlb$ taking the positive 
square root of~$q$ in $\RR$ to the square root $q^{1/2}\in\qlb$ 
of Paragraph \ref{Clery}.
According to 
Arthur \cite{Arthur} 6.1 (see~p.~304)~and
Mok \cite{Mok} 7.1
(see also Labesse \cite{Labesse} p.~38-39),
which describe the local transfer map~$\BC$ of Defini\-tion~\ref{BCq}
for unramified representations of unramified groups,
we have 
\begin{equation*}
\BC(\pi)\otimes_\CC\qlb = \BC_\ell(\pi\otimes_\CC\qlb)
\end{equation*}
for any $K$-unramified complex representation $\pi$ of $G$,
where tensor products are taken with~res\-pect to $\ii$. 
Proposition \ref{NaturaBC} and Remark \ref{rossmiller2}
will describe the dependency of $\BC_\ell$ 
in $q^{1/2}$.  
\end{rema}

\subsection{}
\label{appaAforNRrep}

In this paragraph, 
$G$ is an unramified classical group 
and $K$ is a hyperspecial maximal compact subgroup of $G$
as in Paragraph \ref{NRCLASS}.

Unlike Paragraph \ref{NRCLASS} however,
we will consider \textit{complex} representations
rather than $\qlb$-repre\-sentations.
We~exa\-mi\-ne the dependency of the unramified transfer 
from $G$ to $\GL_N(E)$~with~res\-pect to the choice~of a square root of $q$,
and to the action of $\Aut(\CC)$.
This will be useful~in~Para\-graph \ref{appA}.

Let $\pi$ be a $K$-unramified irreducible complex representation of $G$. 
Associated with $\pi$,
there is the $W$-conjugacy class of an unramified character $\omega$
of $T$, such that $\pi$ is the~uni\-que $K$-unramified irreducible component
of the normalized parabolic induction of $\omega$ to $G$ along~$B$.
Writing $\Ind^G_B$ for unnormalized parabolic induction from $T$ to $G$ 
along~$B$ and $\ip^G_B$ for normalized parabolic~in\-duction,
we have
\begin{equation*}
\ip_B^G(\omega) = \Ind^G_B(\d^{1/2}\omega).
\end{equation*}
Let us write $T \simeq E^{\times m}\times T_0$, where 
\begin{itemize}
\item[$\bullet$] 
$T_0$ is trivial if $G$ is split or $G\simeq\U_{2m}^\a(F)$ with $\a\neq1$, 
\item[$\bullet$]
$T_0=\SO_2^\a(F)$ if $G\simeq\SO_{2n}^\a(F)$ with $\a\neq1$,
\item[$\bullet$]
$T_0=\U_1^\a(F)$ 
if $G\simeq\U_{2m+1}^\a(F)$ with $\a\neq1$.
\end{itemize}
The character $\omega$ can thus be written
$\omega_1\otimes\dots\otimes\omega_m\otimes1$,
where each $\omega_i$ is an unramified character of $E^\times$
and $1$ is the trivial character of $T_0$.
By \cite{MVW} IV.4 p.~69,
the modulus character 
of the~pa\-rabo\-lic~subgroup $P\supseteq B$ with Levi component 
$\GL_m(E)\times T_0$ is equal to 
\begin{equation*}
\nu_m^{d-m-e} \otimes 1 
\end{equation*}
where $\nu_m$ is the unramified character ``absolute value of the 
determinant'' of $\GL_m(E)$ and 
\begin{itemize}
\item[$\bullet$] 
$d=2n$ and $e=-1$ if $G=\Sp_{2n}(F)$, 
\item[$\bullet$] 
$d=2n+1$ and $e=1$ if $G=\SO_{2n+1}(F)$, 
\item[$\bullet$] 
$d=2n$ and $e=1$ if $G=\SO_{2n}^\a(F)$, 
\item[$\bullet$] 
$d=n$ and $e=0$ if $G=\U ^\a_{n}(F)$ with $\a\neq1$.  
\end{itemize} 
By using the transivity property of parabolic induction, 
we deduce that 
\begin{eqnarray*}
\d^{1/2} &=&
|\cdot|_E^{(d-m-e)/2+(m-1)/2}\otimes\dots \otimes
|\cdot|_E^{(d-m-e)/2-(m-1)/2}\otimes 1 \\
&=& 
|\cdot|_E^{(d-e-1)/2}\otimes\dots \otimes
|\cdot|_E^{(d-e-1)/2-m+1}\otimes 1
\end{eqnarray*} 
(where $|\cdot|_E$ is the absolute value of $E$).
Replacing $\sqrt{q}$ by the opposite square root changes 
$|\cdot|_E$ to $\n|\cdot|_E$,
where $\n$ is the unramified character of $E^\times$ of order $2$.
It thus changes $\omega_i$ to $\omega_i\n^{(d-e-1)f}$,
where $f$ is the residual degree of $E$ over $F$
(which is $1$ or $2$ depending whether $E=F$ or not).

Similarly, 
replacing $\sqrt{q}$ by its opposite square root 
has the effect of twisting 
normalized~para\-bo\-lic induction~from $E^{\times N}$ to $\GL_N(E)$ 
(along the Borel subgroup made of upper triangular~ma\-trices)
by $\n^{(1-N)f}$.

Consequently, 
considering the explicit formulas of Paragraph \ref{NRCLASS},
replacing $\sqrt{q}$ 
by the opposite square root 
has the effect of~twis\-ting $\BC(\pi)$ by the character
$\n^{(d-e-N)f}$. 
We have
\begin{itemize}
\item[$\bullet$] 
$d-e-N=2n+1-(2n+1)=0$ if $G=\Sp_{2n}(F)$, 
\item[$\bullet$] 
$d-e-N=(2n+1)-1-2n=0$ if $G=\SO_{2n+1}(F)$, 
\item[$\bullet$] 
$d-e-N=2n-1-2n=-1$ if $G=\SO^\a_{2n}(F)$, 
\end{itemize}
and $f=2$ if $G$ is unitary. 
The integer $(d-e-N)f$ is thus even, 
except if $G$ is even orthogo\-nal.

\begin{exem}
If $G=\SO_2^1(F)\simeq F^\times$,
the transfer of any unramified character $\omega$ of $F^\times$ is
the unique unramified irreducible component of $\omega\times\omega^{-1}$.
If $G$ is the compact group $\SO_2^\a(F)$ with $\a\neq1$, 
the transfer of the trivial character of $G$ is $1\times 1$.
In both cases, 
the transfer depends on the choice of a square root of $q$.
\end{exem}

Now consider an automorphism $\g\in\Aut(\CC)$.
Given a representation $\pi$ of a group $H$ on~a~com\-plex vector space $V$, 
we write $\pi^\g$ for the representation of $H$ on $V\otimes_{\CC}\CC_{\g}$, 
where $\CC_\g$ is the field~$\CC$ considered as a $\CC$-algebra 
\textit{via} $\g$.
Consider the map 
$\pi\mapsto\BC(\pi^{\g^{-1}})^\g$
from $K$-unramified ir\-re\-ducible representations of $G$ to 
irreducible representations of $\GL_N(E)$.
It is the unramified local~trans\-fer map from $G$ to $\GL_N(E)$
with respect to the square root $\g(\sqrt{q})$.
We thus have:

\begin{prop}
\label{NaturaBC}
Let $G$ and $K$ be as above,
and let $\pi$ be a $K$-unramified irreduci\-ble representation of $G$.
Let $\g\in\Aut(\CC)$.
\begin{enumerate}
\item 
If $G$ is not even orthogonal, 
then $\BC(\pi^\g)=\BC(\pi)^\g$.
\item
If $G$ is even orthogonal, then 
$\BC(\pi^\g)=\BC(\pi)^\g\cdot\varepsilon_\g$, 
where $\varepsilon_\g$ is the unramified character
\begin{equation*}
x \mapsto \left(\frac {\g(\sqrt{q})} {\sqrt{q}} \right)^{{\rm val}_F(x)}
\end{equation*}
of $F^\times$.
\end{enumerate}
\end{prop}

\begin{rema}
\label{rossmiller2}
We now go back to $\qlb$-representations. 
We deduce that the map 
\begin{equation}
\label{rossmiller3}
\pi \mapsto \BC_\ell(\pi)
\end{equation}
from (isomorphism classes of) unramified $\qlb$-representations of $G$ 
to those of $\GL_N(E)$
is insensitive to the choice of~a~squa\-re root of $q$ in $\qlb$,
except when $G$ is even orthogonal, 
in which case changing this square root to its opposite
has the effect of twisting \eqref{rossmiller3} by 
$\n$. 
\end{rema}

\section{Representations of local Galois and Weil groups}
\label{GaloisWeil}

In this section, 
$F$ is a $p$-adic field. 
We write $\Ga$ for the Galois group $\Gal(\qpb/F)$
and $\W$ for the associated Weil group,
considered as a subgroup of $\Ga$. 
It is endowed with a smooth character $w\mapsto|w|$ 
with kernel $\I$,
the inertia subgroup of $\W$, 
taking any geometric~Frobenius~ele\-ment to $q^{-1}$,
where $q$ is the cardinality of the residue field of $F$.

All representations of $\Ga$ and $\W$ considered in this section
will be finite-dimensional.

Let $\ell$ be a prime number different from $p$.

\subsection{}
\label{smoothGaW}

For this paragraph,
the reader may refer to \cite{BHGL2} Chapter 7 and \cite{TateCorvallis} 4.2.

If $\s$ is a smooth representation of $\Ga$,
then its restriction $\s|_{\W}$ to $\W$ is smooth.

Restriction from $\Ga$ to $\W$ induces an injection from isomorphism classes 
of irreducible smooth representations of $\Ga$ to isomorphism classes 
of irreducible smooth representations of $\W$.
The~ima\-ge is made of those representations of $\W$ whose determinant has 
finite order (see \cite{BHGL2} 28.6~Pro\-position).

If $\rho$ is a smooth $\ell$-adic representation
(that is, $\qlb$-representation)
of $\W$ on a vector space 
$\V$, and if $\Phi\in\W$ is a Frobenius element,
the following assertions are equivalent:
\begin{enumerate}
\item $\rho$ is semi-simple,
\item $\rho(\Phi)$ is a semi-simple element in $\GL(\V)$, 
\item $\rho(w)$ is a semi-simple element in $\GL(\V)$ for any $w\in\W$
\end{enumerate}
(see \cite{BHGL2} 28.7 Proposition).

If $\s$ is a continuous $\ell$-adic representation of $\Ga$,
its restriction to $\W$ is a continuous $\ell$-adic~re\-pre\-senta\-tion of $\W$,
which is irreducible if and only if $\s$ is irreducible. 

Fix a continuous surjective group homomorphism
\begin{equation}
\label{homt}
t:\I\to\ZZ_\ell.
\end{equation}
For the following proposition,~see 
\cite{BHGL2} 32.5 Theorem and \cite{SerreTate} Appendix.

\begin{prop}
Let $\s$ be a finite-dimensional
continuous $\ell$-adic re\-presentation of $\W$ on a 
$\qlb$-vector space $\V$. 
\begin{enumerate}
\item 
There is a unique~nil\-po\-tent endomorphism $N\in\End(\V)$ such that there
is an open subgroup $U$ of the inertia subgroup $\I$ such that
\begin{equation*}
\s(x) = e^{t(x)N},
\quad
x\in U.
\end{equation*}
\item
We have 
$\s(w)N\s(w)^{-1} = |w|\cdot N$ for all $w\in\W$.
\end{enumerate}
\end{prop}

Note that $N=0$ if and only if $\s$ is smooth.

The subspaces $\Ker(N^i)$, $i\>0$ of $\V$ are stable by $\s$.
Thus, if $\s$ is irreducible, then $N=0$~and $\s$ is smooth.
More generally,
a semi-simple
representation of $\Ga$ is smooth,
and its restriction to $\W$ is smooth semi-sim\-ple 
(see also \cite{TateCorvallis} 4.2.3).

Fix a Frobenius element $\Phi\in\W$.
Associated with $\s$,
there is a smooth $\ell$-adic representation~$\rho$ of $\W$ defined by
\begin{equation*}
\rho (\Phi^a x) = \s(\Phi^a x) e^{-t(x)N},
\quad
a\in\ZZ,
\quad
x\in\I.
\end{equation*}
The pair $(\rho,N)$
is called the \textit{Deligne representation} of $\W$ associated with $\s$.
Up to isomorphism,
it does not depend on the choices of $t$ and $\Phi$
(see \cite{BHGL2} 32.6 Theorem).

The element $\rho(\Phi)=\s(\Phi)$ decomposes uniquely in
$\GL(\V)$ as $su=us$,
with $s$ semi-simple~and~$u$ unipotent.  
Define a smooth $\ell$-adic representation $\rho^{\fss}$ of $\W$ by
\begin{equation*}
\rho^{\fss}(\Phi^a x) = s^a \rho(x), 
\quad
a\in\ZZ,
\quad
x\in\I.
\end{equation*}
This defines a Deligne representation $(\rho^{\fss},N)$,
called the \textit{Frobenius-semi-simplification} of $(\rho,N)$.
By Paragraph \ref{smoothGaW}, 
the representation $\rho^{\fss}$ is a semi-simple smooth
representation of $\W$.

\subsection{}
\label{viteunnumero}

In this paragraph,
if $\kappa$ is a continuous $\ell$-adic representation of $\W$ or $\Ga$,
we will write $\kappa^{\rm ss}$ for its semisimplification.

\begin{lemm}
\label{L1}
If $\rho$ is a (finite-dimensional) smooth
$\ell$-adic representation of $\W$ such that $\rho^{\rm ss}$~is integral,
then $\rho$ is integral.
\end{lemm}

\begin{proof}
We prove it by induction on the dimension $n$ of~$\rho$.
If $\rho$ is irreducible,
there is nothing to prove.
Otherwise,
let $\tau$ be an irreducible subrepresentation of $\rho$,
of dimension $k\>1$, 
and let $\l$ be the quotient of $\rho$ by $\tau$,
which is of dimension $l=n-k$. 
Since $\rho^{\rm ss}=\tau\oplus\l^{\rm ss}$, 
we may apply the inductive hypothesis to $\l$,
from which we deduce that $\l$ is integral.
We therefore fix a basis of the vector space of $\rho$ such that
\begin{equation*}
\rho(w) =
\begin{pmatrix}\tau(w) & \a(w) \\ 0 & \l(w)\end{pmatrix}
\in\GL_{n}(\qlb),
\quad
\tau(w)\in\GL_{k}(\zlb),
\quad
\l(w)\in\GL_{l}(\zlb),
\quad
w\in\W,
\end{equation*}
and $\a:\W\to\Mat_{k,l}(\qlb)$ satisfies 
$\a(xy) = \tau(x)\a(y) + \a(x)\l(y)$ for all $x,y\in\W$.

Since $\rho$ is smooth, 
we may consider it as a representation of the discrete group $\W/U$
for some open subgroup $U$ of $\W$. 
Since this quotient is a finitely generated group,
we may consider $\rho$ as~a representation of the free group $\textsf{F}$
with $r$ generators $\textsf{f}_1,\dots,\textsf{f}_r$ for some $r\>1$.
Assume $\a$ is not identically zero,
and let $-v$ denote the minimum of the $\ell$-adic valuations of all the entries
of all the $\a(\textsf{f}_i)$.
Conjugating $\rho$ by ${\rm diag}(\ell^v\cdot{\rm id}_t,{\rm id}_{l})$,
we may and will assume that $v=0$.

We are going to prove that $\a$ takes values in $\Mat_{t,l}(\zlb)$.
We prove it by induction on the length of the words in $\textsf{F}$.
Given $x\in\textsf{F}$, write it $yf$ with $f=\textsf{f}_i$ for some 
$i\in\{1,\dots,r\}$ 
and the length of $y$ is smaller than that of $x$.
Then
\begin{equation*}
\a(x) = \tau(y)\a(f) + \a(y)\l(f) \in \Mat_{k,l}(\zlb)
\end{equation*}
thanks to the inductive hypothesis. 
\end{proof}

\begin{lemm}
\label{L2}
Let $\s$ be a (finite-dimensional) continuous $\ell$-adic representation of $\Ga$,
with~as\-so\-ciated Deligne representation $(\rho,N)$.
The restriction of $\s^{\rm ss}$ to $\W$ is equal to $\rho^{\rm ss}$.
\end{lemm}

\begin{proof}
Note that semi-simplification and restriction from $\Ga$ to $\W$
commute,
that is
\begin{equation*}
\s^{\rm ss}|_{\W} = (\s|_{\W})^{\rm ss}.
\end{equation*}
If $N$ is zero,
then $\s$ is smooth and $\rho$ is the restriction of $\s$ to $\W$,
thus $\s^{\rm ss}$ is smooth semi-simple,
and its restriction to $\W$ is smooth semi-simple as well.
Otherwise,
if $n=\dim(\s)$,
there is a basis of $\overline{\QQ}{}_\ell^n$ such that 
\begin{equation*}
\s(g) = \begin{pmatrix} \a(g) & \g(g) \\ 0 & \b(g) \end{pmatrix}
\in \GL_n(\qlb), 
\quad
g\in\Ga,
\quad\text{and}\quad
N = \begin{pmatrix} 0 & C \\ 0 & M \end{pmatrix},
\end{equation*}
where 
\begin{itemize}
\item 
$\a$ is a smooth $\ell$-adic representation of $\Ga$ 
of dimension $k=\dim(\Ker\ N)$, 
\item
$\b$ is a continuous $\ell$-adic representation of $\Ga$
of dimension $l=n-k$, 
\item
$\g$ is a continuous map from $\Ga$ to $\Mat_{k,l}(\qlb)$,
\item
$M$ is~nil\-po\-tent in $\Mat_l(\qlb)$ and $C$ is a matrix
in $\Mat_{k,l}(\qlb)$.
\end{itemize} 
We have
\begin{equation*}
e^{t(x)N} = \begin{pmatrix} {\rm id} & \e(x) \\ 0 & e^{t(x)M} \end{pmatrix}
\quad\text{with}\quad 
\e(x) = \sum\limits_{i\>1} \frac {t(x)^i} {i!} CM^{i-1},
\quad
x\in\I.
\end{equation*}
Writing $(\rho_1,N_1)$ for the Deligne representation 
associated with $\b$, we get $N_1=M$ and
\begin{equation*}
\rho(w) = \begin{pmatrix} \a(w) & \d(w) \\ 
0 & \rho_1(w) \end{pmatrix}\in\GL_n(\qlb), 
\quad
w\in\W,
\end{equation*}
for some smooth map $\d$ from $\W$ to $\Mat_{k,l}(\qlb)$
which can be explicitly described by
\begin{equation*}
\d(w) = (\g(w)-\a(w) \e(w))e^{-t(x)M}, 
\quad
w=\Phi^ax, 
\quad
a\in\ZZ, 
\quad 
x\in\I.
\end{equation*} 
By the inductive hypothesis,
we get 
\begin{equation*}
\s^{\rm ss}|_{\W} = (\a|_{\W}) \oplus (\b^{\rm ss}|_{\W})
= (\a|_{\W}) \oplus \rho_1^{\rm ss} = \rho^{\rm ss}.
\end{equation*}
This proves the lemma. 
\end{proof}

\begin{coro}
\label{C3}
Let $\s$ be a (finite-dimensional) continuous $\ell$-adic representation of $\Ga$,
with associated Deligne~re\-pre\-sentation $(\rho,N)$.
Then $\rho$ is (smooth) integral and $\r_\ell(\rho)=\r_\ell(\s)|_{\W}$. 
\end{coro}

\begin{proof}
Since $\s$ is integral (for $\Ga$ is compact),
$\s^{\rm ss}|_{\W}$ is integral. 
We deduce from Lemma \ref{L2} that $\rho^{\rm ss}$ is integral, 
then from Lemma \ref{L1} that $\rho$ is integral. 
Now write
\begin{equation*}
\r_\ell(\rho) = \r_\ell(\rho^{\rm ss}) = \r_\ell(\s^{\rm ss}|_{\W}) 
= \r_\ell(\s^{\rm ss})|_{\W} = \r_\ell(\s)|_{\W}.
\end{equation*}
This proves the corollary. 
\end{proof}

\section{Galois representations associated with automorphic
representations}
\label{sec8}

Recall that we have fixed an isomorphism of fields $\ii:\CC\to\qlb$. 
Fix a positive integer $N$.

Let $k$ be a totally real number field,
and $l$ be either $k$ or a quadratic totally imaginary~exten\-sion of $k$
in an algebraic closure $\overline{\QQ}$ of $\QQ$. 
For any place $v$ of $l$, let $l_v$ denote the completion of $l$ at $v$.

For any finite place $v$,
fix a decomposition subgroup $\Ga_v$ of $\Gal(\overline{\QQ}/l)$ at $v$
and write $\W_v$ for the associated Weil group.
For any finite place $v$ not dividing $\ell$, write
\begin{itemize}
\item 
$\WD(\s)$ for the Deligne representation of $\W_v$ associated with
a continuous $\ell$-adic representation $\s$ of $\Ga_v$ and
$\WD^*(\s)$ for its Frobenius-semi-simplification,
\item
$\LLC_v$ for the local Langlands correspondence 
(\cite {HT} Theorem A) 
between ir\-reducible smooth complex representations of $\GL_N(l_v)$
and $N$-dimensional Frobenius-semi-simple complex Deligne representations
of $\W_v$.
\end{itemize}

\subsection{}

A cuspidal irreducible  automorphic representation $\Pi$ of $\GL_N(\AA_l)$ 
is said to be
\begin{itemize}
\item 
\textit{polarized} if its contragredient $\Pi^\vee$
is isomorphic to $\Pi^c$,
where $c$ is the generator of $\Gal(l/k)$
(thus $\Pi^c=\Pi$ when $l=k$),
\item
\textit{algebraic regular} if 
{the Harish-Chandra module $\Pi_\infty$ associated with $\Pi$}
has the same infini\-te\-simal character as some irreducible algebraic
representation of the restriction of scalars from~$l$ to $\QQ$ of $\GL_N$.
\end{itemize}

Recall the following result of
Barnet-Lamb--Geraghty--Harris--Taylor (\cite{BLGHT} Theorems 1.1,~1.2).

\begin{theo}
\label{theoBLGHT}
Let $\Pi$ be an algebraic regular,
polarized,
cuspidal irreducible automorphic~re\-pre\-sentation of $\GL_N(\AA_l)$. 
There is a continuous semi-simple $\ell$-adic representation
\begin{equation*}
\Sigma : \Gal(\overline{\QQ}/l) \to \GL_N(\qlb)
\end{equation*}
such that,
for any finite place $v$ of $l$ not dividing $\ell$,
we have
\begin{equation*}
\WD^*(\Sigma|_{\Ga_v}) \simeq
\LLC_v(\Pi_v\otimes|\det|_v^{(1-N)/2}) \otimes_{\CC} \qlb.
\end{equation*}
\end{theo}

Note that the representation $\Sigma$ depends on the choice of $\iota$.

\subsection{}

The main result of this section is the following. 
Let $\m$ denote the maximal ideal of~$\zlb$.

\begin{theo}
\label{MAINTHM}
Let $\Pi_1$ and $\Pi_2$ be algebraic regular, polarized,
cuspidal irreducible automorphic re\-presentations of $\GL_n(\AA_l)$.
Suppose that 
there is a finite set $\SS$ of places of $l$, containing~all infinite places, 
such that for all $v\notin\SS$~:
\begin{enumerate}
\item
the local components $\Pi_{1,v}$ and $\Pi_{2,v}$ are unramified,
\item
the characteristic polynomials of the conjugacy classes of semisimple 
elements in $\GL_n(\qlb)$ associated with 
$\Pi_{1,v}\otimes_{\CC}\qlb$ and $\Pi_{2,v}\otimes_{\CC}\qlb$
have coefficients in $\zlb$ and are congruent mod $\m$.
\end{enumerate}
Then,
for any finite place $v$ of $l$ not dividing $\ell$,
the representations $\Pi_{1,v}\otimes_{\CC}\qlb$
and $\Pi_{2,v}\otimes_{\CC}\qlb$~are integral,
their reductions mod $\m$ 
share a common generic irreducible compo\-nent,
and such a~ge\-ne\-ric component is unique.
\end{theo}

\begin{proof}
Applying Theorem \ref{theoBLGHT} to $\Pi_1$ and $\Pi_2$,
we get continuous $\ell$-adic representations 
\begin{equation*}
\Sigma_i : \Gal(\overline{\QQ}/l) \to \GL_N(\qlb),
\quad
i=1,2,
\end{equation*}
such that,
for any finite place $v$ of $l$ not dividing $\ell$,
we have
\begin{equation*}
\WD^*(\Sigma_{i,v}) \simeq
\LLC_v(\Pi_{i,v}\otimes|\det|_v^{(1-N)/2})\otimes_{\CC} \qlb
\end{equation*}
where $\Sigma_{i,v}$ denotes the restriction of $\Sigma_i$ to $\Ga_v$
and the tensor product over $\CC$ 
is taken with respect to $\iota$. 
For all $v\notin\SS$,
the $\ell$-adic representation $\Pi_{i,v}\otimes_{\CC} \qlb$ 
is unramified, generic and integral,
thus 
\begin{equation*}
\LLC_v(\Pi_{i,v}\otimes|\det|_v^{(1-N)/2}) \otimes_{\CC} \qlb \simeq (\phi_{i,v},0)
\end{equation*}
where $\phi_{i,v}$ is an integral semi-simple 
$\ell$-adic representation of $\W_v$ trivial on $\I_v$.
It is thus entirely determined by the semi-simple matrix
\begin{equation*}
\phi_{i,v}(\Frob_v) \in \GL_N(\qlb)
\end{equation*}
where $\Frob_v$ is a Frobenius element in $\W_v$.
By assumption,
its characteristic~po\-ly\-nomial has~coeffi\-cients in $\zlb$,
hence,
as $\phi_{i,v}(\Frob_v)$ is semi-simple,
its eigenvalues are in $\overline{\ZZ}{}_\ell^\times$. 
That the~nil\-po\-tent~ope\-rator is $0$ implies that
$\Sigma_{i,v}$ is smooth,
thus
\begin{equation*}
\phi_{i,v}=(\Sigma_{i,v}|_{\W_v})^{\rm ss}.
\end{equation*}
Thus $\Sigma_{i,v}$ is trivial on the inertia subgroup $\I_v$,
that is, $\Sigma_{i,v}$ is unramified.

Given $v\notin\SS$ and $i\in\{1,2\}$,
let $\P_{i,v}(\T)$ be the characteristic~po\-ly\-nomial of
$\Sigma_{i,v}(\Frob_v)\otimes\flb$, 
that is,
the characteristic polynomial of $\phi_{i,v}(\Frob_v)\otimes\flb$.
By assumption, 
$\P_{1,v}(\T)=\P_{2,v}(\T)$~at all~$v\notin\SS$. 
Applying Deligne-Serre \cite{DS} Lemma 3.2
to the semi-simple $\flb$-representa\-tions 
$(\Sigma_1\otimes\flb)^{\rm ss}$ and $(\Sigma_2\otimes\flb)^{\rm ss}$,
which at $v\notin\SS$ give $\phi_{1,v}\otimes\flb$ and
$\phi_{2,v}\otimes\flb$ respectively, 
we deduce that $(\Sigma_1\otimes\flb)^{\rm ss}$ and 
$(\Sigma_2\otimes\flb)^{\rm ss}$ are isomorphic. 
In particular,
we deduce that
\begin{equation*}
(\Sigma_1\otimes\flb)^{\rm ss} |_{\Ga_w} \simeq 
(\Sigma_2\otimes\flb)^{\rm ss} |_{\Ga_w}
\end{equation*}
thus the continuous 
$\ell$-adic representations $\Sigma_{1,w}$ and $\Sigma_{2,w}$ of $\Ga_w$
are congruent mod $\ell$.

Now write $\WD^*(\Sigma_{i,w})=(\rho_i,N_i)$ for $i=1,2$.
Thanks to Corollary \ref{C3},
we know that $\rho_1$~and $\rho_2$ are integral and have same
reduction mod $\ell$.
By \cite{Vigl} Theorem 1.6,
we deduce that
\begin{equation*}
\mu_1 = (\Pi_{1,w}\otimes|\det|_w^{(1-N)/2})\otimes_{\CC} \qlb,
\quad
\mu_2 = (\Pi_{2,w}\otimes|\det|_w^{(1-N)/2})\otimes_{\CC} \qlb,
\end{equation*}
are integral and
\textit{have the same mod $\ell$ supercuspidal support},
that is,
the supercuspidal support
of any irreducible component $\upsilon$
of $\r_\ell(\mu_i)$
is indepen\-dent~of $i$ (and of the choice of $\upsilon$).

Since $\mu_i$ is generic (as $\Pi_{i,w}$ is 
a local component of a cuspidal 
automorphic repre\-sentation of $\GL_n(\AA_l)$),
the $\flb$-representation
$\r_\ell(\mu_i)$ contains a generic irreducible~com\-ponent $\d_i$.
It occurs~in $\r_\ell(\mu_i)$ with multiplicity $1$,
and any generic irreducible representation occurring in $\r_\ell(\mu_i)$
is isomorphic to $\d_i$.
Since $\d_i$ only depends on the mod $\ell$ supercuspidal support of 
$\mu_i$ (\cite{Vigl} III.5.10),
we deduce that $\d_1$ and $\d_2$ are isomorphic.
\end{proof}

\begin{rema}
We expect Theorem \ref{MAINTHM} to hold
without assuming that $\Pi_1$, $\Pi_2$ are polarized.
\end{rema}

\section{Proof of the main theorem}
\label{argumentfinal}

\subsection{}
\label{portnoy}

We prove our main theorem \ref{MAINTHEOINTRO}.

Let $p$ be a prime number different from $2$,
let $F$ be a $p$-adic field
and $G$ be a quasi-split special orthogonal, unitary or symplectic group
over $F$.
We thus have
\begin{itemize}
\item[$\bullet$] 
either $G=\SO(Q)$ for some non-degenerate quadratic form $Q$ over $F$, 
\item[$\bullet$]
or $G=\U(H)$ for some non-degenerate $E/F$-Hermitian form $H$,
\item[$\bullet$]
or $G=\Sp(A)$ for some non-degenerate symplectic form $A$ over $F$.
\end{itemize}
As usual, 
we write $E=F$ in the symplectic and orthogonal cases.

In this paragraph and the next one, 
we assume that the group $G$ is not the split special~ortho\-go\-nal group 
$\SO_2(F)\simeq F^\times$.
The case of split $\SO_2(F)$ will be treated in Paragraph \ref{splitSO2}.

Let $\pi_1$, $\pi_2$ be integral cuspidal irreducible $\qlb$-re\-pre\-sentations
of $G$ such that
\begin{equation*}
\r_\ell(\pi_1) \< \r_\ell(\pi_2).
\end{equation*}
First,
let $k$, $w$, $\GG$ be as in Theorem \ref{GLOBALG}.
More precisely,
we have
\begin{itemize}
\item[$\bullet$] 
either $\GG=\SO(q)$ for a quadratic form 
$q$ as in Theorem \ref{THEOFQ} if $G$ is special 
orthogonal, 
\item[$\bullet$]
or $\GG=\U(h)$ for an $l/k$-Hermitian form 
$h$ as in Theorem \ref{THEOFH} if $G$ is unitary, 
\item[$\bullet$]
or $\GG$ is as in Paragraph \ref{THEOFS} if $G$ is symplectic (see also
Paragraph \ref{ST}). 
\end{itemize}
In particular, 
we have $k_w=F$ and $l_w=E$, 
and the group $\GG(F)$ naturally identifies with $G$.~As
usual, we write $l=k$ in the symplectic and orthogonal cases.

Let $\GG^*$ be the quasi-split inner form of $\GG$ over $k$,
and write $N=N(\GG^*)$.
We thus have:
\begin{itemize}
\item[$\bullet$] 
either $\GG^*=\SO(q^*)$ where $q^*$ is a quadratic form over $k$ 
as in \eqref{defqodd} or \eqref{defqeven}, 
\item[$\bullet$]
or $\GG^*=\U(h^*)$ where $h^*$ is an $l/k$-Hermitian form 
as in \eqref{defh}, 
\item[$\bullet$] 
or $\GG^*=\Sp(f^*)$ where $f^*$ is a symplectic form over $k$ 
as in \eqref{deff}. 
\end{itemize}

Let $\BC$ be the local transfer from $\GG^*(F)$ to $\GL_N(E)$ 
given by Definition \ref{BCq}.
We explained~how to 
ca\-no\-nically~identify representa\-tions of
$\GG(F)$ with those of $\GG^*(F)$ in Paragraph \ref{ST}.
(In the symplectic case, 
we identified~$\GG(F)$ with $\Sp(f_w)$ for some symplectic 
form $f_w$ over $k_w=F$.)
This gives us a local transfer from $\GG(F)=G$ to $\GL_N(E)$,
still denoted $\BC$.

\begin{lemm}
\label{placeu}
There is a finite place $u$ of $k$ different from~$w$,
not di\-vi\-ding $\ell$, 
such that there~is a unitary cuspidal~ir\-re\-ducible complex
representation $\rho$ of $\GG(k_u)$ with the following properties:
\begin{enumerate}
\item 
$\rho$ is com\-pact\-ly~in\-du\-ced from~some compact mod centre,
open subgroup of $\GG(k_u)$,
\item
the local transfer of $\rho$ to $\GL_N(l_u)$ is cuspidal.
\end{enumerate}
\end{lemm}

\begin{proof}
If $G$ is special orthogonal,
it suffices to choose $u\neq w$ such that $\GG(k_u)$ is split
(Remark \ref{sosplit}),
and then apply Proposition \ref{propLS}.

If $G$ is unitary,
it suffices to choose $u\neq w$ such that $\GG(k_u)$ is split
(Remark \ref{bavardise}).

If $G$ is symplectic,
it suffices to choose a place $u$ 
such that $k_u$ is isomorphic to $\QQ_2$
(see Lem\-ma \ref{arthurtotallyreal})
and then apply Theorem \ref{thm:guy}.
\end{proof}

Now fix an isomorphism of fields $\ii:\CC\to\qlb$.
As in Paragraph \ref{sommerive},
let $\BC_\ell$ denote the $\ell$-adic~lo\-cal~transfer from $G$ to $\GL_N(E)$
obtained from $\BC$ thanks to $\ii$, that is, 
$\BC_\ell(V\otimes_\CC\qlb)
=\BC(V)\otimes_\CC\qlb$
for any complex representation $V$ of $G$,
where tensor products are taken with~res\-pect to $\ii$. 
Let $u$ and $\rho$ be as in Lemma \ref{placeu}.
By~Theo\-rem \ref{GLOBALPIPRIME}, 
there are irreducible~auto\-morphic~re\-pre\-sentations $\Pi_1$
and $\Pi_2$ of $\GG(\AA)$ such that
\begin{enumerate}
\item 
$\Pi_{1,u}$ and $\Pi_{2,u}$ are both isomorphic to $\rho$,
\item 
$\Pi_{1,w}\otimes_{\CC}\qlb$ is isomorphic to $\pi_1$ 
and $\Pi_{1,w}\otimes_{\CC}\qlb$ is isomorphic to $\pi_2$, 
\item 
$\Pi_{1,v}$ and $\Pi_{2,v}$ are trivial for any real place $v$,
\item
there is a finite set $\SS$ of places of $k$, containing all real places, 
such that for all $v\notin\SS$~:
\begin{enumerate}
\item
the group $\GG$ is unramified over $\k_v$, and the 
local components $\Pi_{1,v}$ and $\Pi_{2,v}$ are~un\-ra\-mified
with respect to some hyperspecial maximal 
compact subgroup $\K_v$ of 
$\GG(k_v)$, 
\item
the restrictions of the Satake parameters 
of $\Pi_{1,v}\otimes_{\CC}\qlb$ and $\Pi_{2,v}\otimes_{\CC}\qlb$
to the Hecke $\zlb$-algebra $\Hh_{\zlb}(\GG(k_v),\K_v)$
are 
congruent mod the maximal ideal $\m$ of $\zlb$.
\end{enumerate}
\end{enumerate}
Applying Theorems \ref{transfertTaibi} and \ref{transfertLabesse}
to $\Pi_1$, $\Pi_2$,
we get algebraic regular, polarized, cuspidal irreduci\-ble automorphic 
representations $\widetilde{\Pi}_1,\widetilde{\Pi}_2$ of $\GL_N(l)$ such that,
for $i=1,2$ and all {finite} places~$v$ of $k$, 
the local transfer of $\Pi_{i,v}$ to $\GL_N(l_v)$ is $\widetilde{\Pi}_{i,v}$.
Writing $\BC_v$ for the local transfer over $k_v$,
we~thus have $\widetilde{\Pi}_{i,v}=\BC_v(\Pi_{i,v})$,
or~equi\-valently 
$\widetilde{\Pi}_{i,v}\otimes_{\CC}\qlb=\BC_{v,\ell}(\Pi_{i,v}\otimes_{\CC}\qlb)$
where $\BC_{v,\ell}$ is obtained from $\BC_v$ thanks to $\ii$.

In particular,
for all $v\notin\SS$,
it follows from Proposition \ref{enonceULT} that
$\widetilde{\Pi}_{1,v}$ and $\widetilde{\Pi}_{2,v}$ are unramified and that 
the characteristic polynomials of the conjugacy classes of semisimple 
elements in $\GL_n(\qlb)$ associated with 
$\widetilde{\Pi}_{1,v}\otimes_{\CC}\qlb$ and $\widetilde{\Pi}_{2,v}\otimes_{\CC}\qlb$
have coefficients in $\zlb$ and are congruent mod $\m$.

Now apply Theorem \ref{MAINTHM} at $w$: 
the representations $\widetilde{\Pi}_{1,w}\otimes_{\CC}\qlb$
and $\widetilde{\Pi}_{2,w}\otimes_{\CC}\qlb$ are integral,
their reductions mod $\ell$ 
share a common generic irreducible compo\-nent,
and such a generic~com\-po\-nent is unique.
The result now follows from the fact that 
$\Pi_{i,w}\otimes_{\CC}\qlb\simeq\pi_i$ 
for $i=1,2$.

\subsection{}
\label{appA}

We now describe how the map $\BC_\ell$ depends on the choice of $\ii$.
Equivalently,
since any two~iso\-mor\-phisms $\ii,\ii'$ between $\CC$ and $\qlb$
give rise to a field automorphism $\ii^{-1}\circ\ii'$ of $\CC$,
we will describe the beha\-vior~of $\BC$ under the action of $\Aut(\CC)$.
More precisely, 
we prove the following result. 

\begin{prop}
\label{NaturaBCcusp}
Let $\pi$ be a cuspidal complex representation of $G$.
Let $\g\in\Aut(\CC)$.
\begin{enumerate}
\item 
If $G$ is not even orthogonal, 
then $\BC(\pi^\g)=\BC(\pi)^\g$.
\item
If $G$ is even orthogonal, then 
$\BC(\pi^\g)=\BC(\pi)^\g\cdot\varepsilon_\g$, 
where $\varepsilon_\g$ is the unramified character
\begin{equation}
\label{defepsig}
x \mapsto \left(\frac {\g(\sqrt{q})} {\sqrt{q}} \right)^{{\rm val}_F(x)}
\end{equation}
of $F^\times$,
where $q$ is the cardinality of the residue field of $F$.
\end{enumerate}
\end{prop}

Let $\pi$ be a {cuspidal} complex representation of $G$.
As in Lemma \ref{placeu}, 
let $u$~be a finite~place of $k$ different from $w$, not dividing $\ell$, 
and $\rho$~be~a~uni\-tary
cuspidal irreducible complex representation~of $\GG(k_u)$ with cuspidal
transfer.
By Proposition \ref{GLOBALPI},~we 
have an irreducible~au\-tomor\-phic repre\-sen\-ta\-tion $\Pi$ of $\GG(\AA)$ 
such that
\begin{enumerate}
\item 
the local component $\Pi_u$ is isomorphic to $\rho$,
\item 
the local component $\Pi_w$ is isomorphic to $\pi$, 
\item 
the local component $\Pi_v$ is the trivial character of $\GG(k_v)$ for any 
real place $v$ of $k$.
\end{enumerate}
Associated with $\Pi$ by Theorems \ref{transfertTaibi} and
\ref{transfertLabesse}, 
there is an algebraic regular, polarized,
cuspidal~ir\-re\-du\-cible automorphic representation 
$\widetilde{\Pi}$ of $\GL_N(\AA_l)$ such that 
$\widetilde{\Pi}_v=\BC_v(\Pi_v)$ for all finite places $v$ of $k$,
where $\BC_v$ is as in Paragraph \ref{portnoy}.
Now let $\g\in\Aut(\CC)$.
Then $\Pi^\g$ satisfies 
\begin{enumerate}
\item 
the local component $\Pi^\g_u$ is isomorphic to $\rho^\g$,
\item 
the local component $\Pi^\g_w$ is isomorphic to $\pi^\g$, 
\item 
the local component $\Pi^\g_v$ is the trivial character of $\GG(k_v)$ for any 
real place $v$ of $k$.
\end{enumerate}
Associated with it by Theorems \ref{transfertTaibi} and
\ref{transfertLabesse}, 
there is an algebraic regular, polarized, 
cuspidal~ir\-re\-du\-cible automorphic representation 
$\widetilde{\Pi}'$ of $\GL_N(\AA_l)$ such that 
$\widetilde{\Pi}'_v=\BC^{\phantom{\g}}_v(\Pi^\g_v)$ for all finite $v$.

Let $\SS$ be a finite set of places of $k$,
containing all real places,
such that for all $v\notin\SS$
the~group $\GG$ is unramified over $k_v$ and 
the local component $\Pi_v$ is unramified with respect to some 
hyperspe\-cial maximal compact subgroup of $\GG(k_v)$. 

Assume first that $G$ is not an even special orthogonal group.
For $v\notin\SS$,
Proposition \ref{NaturaBC} gives us 
$\BC^{\phantom{\g}}_v(\Pi^\g_v)=\BC_v(\Pi_v)^\g$,
thus $\widetilde{\Pi}'$ and $\widetilde{\Pi}^\g$ coincide at almost all finite places.
By strong multiplicity $1$, we deduce that $\widetilde{\Pi}'=\widetilde{\Pi}^\g$.
It follows that $\BC^{\phantom{\g}}_w(\Pi^\g_w)=\BC_w(\Pi_w)^\g$,
that is, $\BC(\pi^\g)=\BC(\pi)^\g$.

Assume now that $G$ is even special orthogonal (thus $l=k$).
For all finite places $v$ of $k$,~let~$\BC^*_v$ be the map 
$\pi\mapsto\BC_v(\pi)|\det|_v^{1/2}$,
where $|\cdot|_v$ is the absolute value of $k_v^\times$ and
$|\cdot|_v^{1/2}$ is its square root with~respect to $q_v^{1/2}$,
where $q_v$ is the cardinality of the residue field of $k_v$.
An argument similar to that of the non even orthogonal case gives us
$\widetilde{\Pi}'|\det|^{1/2}=(\widetilde{\Pi}|\det|^{1/2})^\g$
where $|\cdot|$ is the~ab\-so\-lu\-te~value of $\AA^\times$.
Looking at the local component at $w$, we deduce that
$\BC(\pi^\g)=\BC(\pi)^\g\cdot\e_\g$, whe\-re~$\e_\g$~is~defined
as in \eqref{defepsig}.
We have proved Proposition \ref{NaturaBCcusp}.

\begin{rema}
\label{contrevie}
The same argument shows that Proposition \ref{NaturaBCcusp} holds
for all discrete series~re\-presentations $\pi$ of $G$
(it suffices to replace Proposition \ref{GLOBALPI} by
\cite{Shin} Theorem 5.13).
Let us explain how this implies 
that the set of isomorphism classes of discrete series representations of $G$ 
is~sta\-ble under $\Aut(\CC)$. 
Let $\h$ be the local Langlands parameter of a
discrete series representation $\pi$
(up to ${\rm O}_{2n}(\CC)$-conjugacy in the even orthogonal case)
and let $\phi=\Std\circ\h$ be the Langlands~pa\-ra\-meter of $\BC(\pi)$.
On the one hand, 
the fact that $\pi$~is~a~dis\-crete~series representation implies~that 
the~quo\-tient of the centralizer 
of the ima\-ge of $\phi$ in $\GGH$
by ${\bf Z}(\GGH)^{\W_k}$ is finite
(see the end of~Pa\-ra\-graph \ref{defstd}).
On the other hand, 
the Langlands parameter of $\BC(\pi^\g)$ is $\phi'=\phi^\g\cdot\n\chi$
(where~$\n$ is the unramified character of $F^\times$ of order $2$
and
$\chi$ is either the character $\e_\g$ defined by
\eqref{defepsig} if $G$ is even orthogonal, 
or the trivial character otherwi\-se),
which has the same finiteness property.
Thus the $L$-packet of $\pi^\g$ is discrete. 
Thus $\pi^\g$ is a discrete series representation.
\end{rema}

\begin{rema}
\label{AutF}
Let us examine how the local transfer map 
behaves under automorphisms~of the base field $F$,
for discrete series representations.
Let $\pi$ be a discrete series representation of $G$,
and let $\phi$ be the Langlands parameter of its transfer $\BC(\pi)$.
By Moeglin \cite{Moeglin41},
an irreducible Langlands parameter $\s\boxtimes\SS_a$,
where $\s$ is an irreducible representation of dimension $k\>1$ of $\W_F$ 
and $a$ is a positive integer, 
occurs in $\phi$ if and only if:
\begin{enumerate}
\item
the cuspidal representation $\rho$ of $\GL_{k}(E)$
associated with $\s$ by the 
Langlands corres\-pon\-dence is $c$-selfdual, 
\item
if $s$ is the unique non-negative real number such that
the normalized paraboli\-cally indu\-ced~representation
$\rho\nu^s\rtimes\pi$ 
is reducible,
then $2s-1$ is a positive integer
and $2s-1-a$ is~a~non-negative~even integer.
\end{enumerate} 
Now let $\varkappa\in\Aut(F)$,
which extends to an automorphism of $E$ still denoted $\varkappa$.
Then
\begin{itemize}
\item
the cuspidal representation 
$\rho^\varkappa$ is $\varkappa^{-1} c \varkappa$-selfdual, 
\item
the irreducible representation of $\W_E$ 
associated with it by the 
Langlands corres\-pon\-dence is $\s^\varkappa$
(see \cite{HenniartJTNB01} Propri\'et\'e 1), 
\item 
the normalized paraboli\-cally induced representation
$\rho^\varkappa\nu^s\rtimes\pi^\varkappa$ is reducible, 
\item 
the representation $\s\boxtimes\SS_a$ occurs in $\phi$ if and only if
$\s^\varkappa\boxtimes\SS_a$ occurs in $\phi^\varkappa$.
\end{itemize}
It follows that the Langlands parameter of 
$\BC(\pi^\varkappa)$ is $\phi^\varkappa$.
Applying \cite{HenniartJTNB01} Propri\'et\'e 1 again,
$\phi^\varkappa$ is the Langlands parameter of $\BC(\pi)^\varkappa$. 
Thus $\BC(\pi^\varkappa)$ is equal to $\BC(\pi)^\varkappa$. 
\end{rema}

\subsection{}
\label{splitSO2}

{In this paragraph, 
we discuss the case of the split special orthogonal group 
$\SO_2(F)\simeq F^\times$.}

Let $\chi$ be a $\qlb$-cha\-racter of this group. 
Its transfer~to $\GL_2(F)$ is 
\begin{itemize}
\item 
either the normalized parabolically induced 
representation $\chi\times\chi^{-1}$ when the character 
$\chi^2$ is different from the absolute value $|\cdot|$ 
and its inverse $|\cdot|^{-1}$, 
\item
or the unique character occurring as a component of $\chi\times\chi^{-1}$ 
when $\chi^2\in\{|\cdot|,|\cdot|^{-1}\}$.
\end{itemize}
Properties (1) and (2) of Theorem \ref{MAINTHEOINTRO} thus hold, since 
\begin{itemize}
\item 
an irreducible $\qlb$-representation of $\GL_2(F)$
is integral if and only if its cuspidal support 
is integral (see \cite{Vigb} II.4.14 and \cite{DatNuTempered} Proposition 
6.7), 
\item
if $\chi$ is integral,
the supercuspidal support of any irreducible component of 
$\r_\ell(\chi\times\chi^{-1})$ is the $\GL_2(F)$-conjugacy class of the cuspidal 
pair $(F^\times\times F^\times,\chi\otimes\chi^{-1})$.
\end{itemize}

However, 
if $\xi$ is any non-trivial character of $F^\times$ with values in $1+\m$ 
(where $\m$ is the maximal ideal of $\zlb$)
such that $\xi^2\notin\{1,|\cdot|^{-2}\}$,
the characters 
$|\cdot|^{1/2}$ and $\xi|\cdot|^{1/2}$ are congruent, 
but~the transfer~of the first one is the trivial character of 
$\GL_2(F)$, 
which is not generic.
Property~(3)~thus does not hold. 
Also, 
the transfer of the second one is 
$\xi|\cdot|^{1/2}\times\xi^{-1}|\cdot|^{-1/2}$,
whose reduction mod $\ell$ contains the trivial character with multiplicity 
$1$ (if $\ell\neq2$)
or $2$ (if $\ell=2$)
by \cite{VigGL2} Th\'eor\`eme 3.

Assume further that $q$ has order $2$ mod $\ell$,
that is, $\ell$ divides $q^2-1$ but not $q-1$,
and let~$\n$~be the~unique unramified $\qlb$-character of order $2$ of 
$F^\times$. 
Then the transfer of $\n|\cdot|^{-1/2}$
(which is con\-gruent to $|\cdot|^{1/2}$)
is $\n\circ\det$,
whose reduction mod $\ell$ is a character of order $2$.
We thus have two~congruent characters of $F^\times$ whose
transfers to $\GL_2(F)$ have reductions mod $\ell$ with~no~com\-ponent
in common. 

\appendix

\section{Cyclic base change}
\label{appBC}

Let $F$ be a $p$-adic field, 
and let $K$ be a cyclic finite extension of $F$ of degree $d$.
Fix an integer $n\>1$
and write $G=\GL_n(F)$ and $H=\GL_n(K)$. 
By \cite{AC},
there exists a map~from~isomor\-phism classes of irreducible
(smooth)~com\-plex re\-pre\-sentations of $G$
to those of $H$
called the local \textit{base change},
denoted $\CB=\CB_{K/F}$.


Now let us fix a prime number $\ell$ different from $p$
and an isomorphism of fields $\ii$ between~$\CC$~and $\qlb$.
Replacing $\CC$ by $\qlb$ thanks to~$\ii$,~one 
obtains a local base change $\CB_{K/F,\ell}$
for irreducible smooth $\qlb$-re\-pre\-sen\-ta\-tions.

In this appendix,
we investigate the dependency of $\CB_{K/F,\ell}$ in the choice of $\ii$,
or equivalently the behavior of $\CB_{K/F}$ with respect to automorphisms of $\CC$.

\subsection{}

Let $\rec_F$ denote the local Langlands correspondence
from the set of isomor\-phism classes of~ir\-re\-ducible 
complex representations of $G$
to the set $\Phi(G)$ of $\GL_n(\CC)$-conjugacy classes~of local~Lang\-lands
para\-meters for $G$ (\cite{HT,Henniart}).

Replacing $\CC$ by $\qlb$ thanks to~$\ii$,~one 
obtains a local Langlands correspondence $\rec_{F,\ell}$
for irredu\-ci\-ble $\qlb$-re\-pre\-sen\-ta\-tions.
The dependency of $\rec_{F,\ell}$ in $\ii$,
or equivalently the behavior of~$\rec_{F}$ with~res\-pect to automorphisms of 
$\CC$, 
has been studied in
\cite{HenniartJTNB01,ClozelAnnArbor90}:
the map $\pi\mapsto\rec_F(\pi|\det|^{(1-n)/2})$~is
in\-sensitive to automorphisms of $\CC$.
It follows that
\begin{equation}
\label{naturec}
\rec_F(\pi^\g) = \rec_F(\pi)^\g \cdot\n_{F,\g}^{1-n}
\end{equation}
for all $\g\in\Aut(\CC)$ and all ir\-re\-ducible 
complex representations $\pi$ of $G$,
where 
\begin{equation}
\label{explimu}
\n_{F,\g}(w)=\left(\frac {\g(\sqrt{q})} {\sqrt{q}} \right)^{\v_F(w)}
\end{equation}
for all $w\in \W_F$,
where $\v_F$ is the valuation map taking any Frobenius element to $1$.

\subsection{}

Let $\res_{K/F}$ be the map from $\Phi(G)$ to $\Phi(H)$
defined by restricting local
Langlands parameters from $\WD_F$ to $\WD_K$.
The local base change $\CB_{K/F}$ is characterized by the identity
\begin{equation*}
\rec_K \circ \CB_{K/F} = \res_{K/F} \circ \rec_F.
\end{equation*}
Now let us prove that 
$\CB=\CB_{K/F}$ is insensitive to the action of $\Aut(\CC)$. 

\begin{prop}
\label{garabedian}
For all $\g\in\Aut(\CC)$ and all ir\-re\-ducible 
complex representations $\pi$ of $G$,~we have
$\CB_{K/F}(\pi^\g)=\CB_{K/F}(\pi)^\g$.
\end{prop}

\begin{proof}
Let $\pi$ be an irreducible complex representation of $G$.
We have
\begin{eqnarray*}
\rec_K(\CB_{K/F}(\pi^\g)) &=& \res_{K/F}(\rec_F(\pi^\g)) \\
&=& \res_{K/F}(\rec_F(\pi)^\g \cdot \n_{F,\g}^{1-n}) \\
&=& \rec_K(\CB_{K/F}(\pi))^\g \cdot \left(\n_{F,\g}|_{\W_K}\right)^{1-n} \\
&=& \rec_K(\CB_{K/F}(\pi)^\g ) \cdot \left(\n_{K,\g} \cdot \n_{F,\g}|_{\W_K}\right)^{1-n}.
\end{eqnarray*}
We are thus reduced to compare $\n_{F,\g}|_{\W_K}$ with $\n_{K,\g}$.
Using the explicit formula \eqref{explimu}, we get
\begin{equation*}
\n_{F,\g}|_{\W_K} = \left(\frac {\g(\sqrt{q})} {\sqrt{q}} \right)^{\v_F|_{\W_K}},
\quad
\n_{K,\g} = \left(\frac {\g(\sqrt{q'})} {\sqrt{q'}} \right)^{\v_K},
\end{equation*}
where $q'$ is the cardinality of the residue field of $K$.
Since $q'=q^{f_{K/F}}$ and $\v_F|_{\W_K}=f_{K/F}\v_K$,
we deduce that $\n_{F,\g}|_{\W_K}=\n_{K,\g}$,
thus $\CB_{K/F}(\pi^\g)=\CB_{K/F}(\pi)^\g$.
\end{proof}

\subsection{}

The map $\CB_{K/F,\ell}$ preserves the fact of being integral: 
this follows from the fact that
an~irredu\-cible $\qlb$-representation $\pi$ of $G$ is integral if and only
if the restriction of $\rec_{F,\ell}(\pi)$ to $\W_F$ is~in\-te\-gral
(\cite{Vigl} 1.4) and that the restriction to $\W_K$ of an integral 
$\qlb$-representation of $\W_F$ is integral.


\subsection{}
\label{BClinearcyclic}

We now review the congruence properties of $\CB_{K/F,\ell}$,
after J.~Zou's PhD thesis \cite{ZouThese} 1.10.

Associated with an irreducible representation $\tau$ of $\GL_n(K)$,
with coefficients in $\qlb$ or $\flb$, there is a partition
\begin{equation*}
\l(\tau)=(k_1\>k_2\>\dots)
\end{equation*}
of $n$ defined inductively as follows.
Let $k_1$ denote the largest integer $k\in\{1,\dots,n\}$ such that~the $k$th 
derivative $\tau^{(k)}$~is non-zero. 
If $k_1=n$, then $\l(\tau)=(n)$.
Otherwise, 
$(k_2\>\dots)$ is the partition of $n-k_1$ associated with the 
representation $\tau^{(k_1)}$ of $\GL_{n-k_1}(K)$.

By \cite{Vigl} V.9.2,
if $\tau$ is an~in\-te\-gral irreducible $\qlb$-representation of 
$\GL_n(K)$,
its reduction mod $\ell$ has a~uni\-que ir\-re\-du\-cible component $\pi$
such that $\l(\pi)=\l(\tau)$.
This component is denoted $\j_\ell(\tau)$. 

\begin{theo}[\cite{ZouThese} Theorem 1.10.17]
Let $\pi_1$ and $\pi_2$ be integral irreducible $\qlb$-repre\-sen\-ta\-tions 
of $\GL_n(F)$.
If $\j_\ell(\pi_1)=\j_\ell(\pi_2)$,
then $\j_\ell(\CB_{K/F,\ell}(\pi_1))=\j_\ell(\CB_{K/F,\ell}(\pi_2))$.
\end{theo}

In particular,
if $\pi_1$, $\pi_2$ are cuspidal,
which implies that $\l(\pi_1)=\l(\pi_2)=(n)$,
their base~chan\-ges $\CB_{K/F,\ell}(\pi_1)$ and
$\CB_{K/F,\ell}(\pi_2)$
are generic.
This theorem thus says that, 
if $\r_\ell(\pi_1)=\r_\ell(\pi_2)$,
then $\r_\ell(\CB_{K/F,\ell}(\pi_1))$ and $\r_\ell(\CB_{K/F,\ell}(\pi_2))$
have a unique generic irreducible component in common.
This can be seen as an analogue of Theorem \ref{MAINTHEOINTRO}
for the cyclic base change from $G$ to $H$. 

\subsection{}
\label{examCBcyc}

In this paragraph,
we give an example of congruent integral cuspidal 
$\qlb$-repre\-sen\-ta\-tions $\pi_1$,~$\pi_2$ of $G$
such that $\CB_{K/F,\ell}(\pi_1)$ and $\CB_{K/F,\ell}(\pi_2)$ are not
congruent. 

First, assume that $\pi$ is an integral cuspidal
irreducible $\qlb$-representation of $G$.
Let $m$ denote the cardinality of the set of isomorphism classes of
$\pi\chi$,
where $\chi$ runs over
the characters of $F^\times$ trivial on $\N_{K/F}(K^\times)$,
and set $e=d/m$. 
Then there exists a cuspidal irreducible
representation~$\rho$ of $\GL_{n/e}(K)$ such that 
\begin{equation*}
\CB_{K/F}(\pi)=\rho\times\rho^\a\times\dots\times\rho^{\a^{e-1}}
\end{equation*}
where $\a$ is a generator of $\Gal(K/F)$
and $\times$ denotes normalized parabolic induction with respect to a choice
of square root of $q$, the cardinality of the residue field of $F$
(see \cite{AC} Chapter 1, \S6.4).

Now assume that $n=2$ and that $K$ is a ramified quadratic extension of $F$,
and let $\omega_{K/F}$ be the character of $F^\times$ with kernel 
$\N_{K/F}(K^\times)$.
Let $\pi_1$ be an integral cuspidal $\qlb$-representation~of $G=\GL_2(F)$ of
level $0$.
By \cite{BK},
it is compactly induced~from~a representation $\boldsymbol{\l}_1$ of
$F^\times\GL_2(\Oo_F)$
whose restriction to $\GL_2(\Oo_F)$ is the inflation of a
cuspidal~ir\-re\-du\-ci\-ble representation $\s_1$ of the group
$\GL_2(\kk)$,
where $\kk$ is the residue field of $F$.
Associated with $\s_1$,~the\-re is (\cite{Green}) a character
\begin{equation}
\xi_1:\boldsymbol{l}^\times\to\overline{\ZZ}{}_\ell^\times
\end{equation}
such that $\xi_1^{q}\neq\xi_1$,
where $\boldsymbol{l}$ is a quadratic extension
of $\kk$ and $q$ is the cardinality of $\kk$. 

The~re\-pre\-sentation $\pi_1\omega_{K/F}$ is isomorphic to $\pi_1$ if and only
if $\boldsymbol{\l}_1\omega_{K/F}$ is isomorphic to
$\boldsymbol{\l}_1$.~As these representations all have the same central character,
this is equivalent to $\s_1\eta\simeq\s_1$,
where~$\n$ is the unique character of order $2$ of $\kk^\times$
(note that the restriction of $\omega_{K/F}$ to $\Oo_F^\times$ is the
inflation~of $\eta$),
which is equivalent to $\xi_1(\eta\circ\N_{l/\kk})=\xi_1^q$,
that is, $\xi_1^{q-1}$ has order $2$.
Assume that this is the ca\-se.~Thus~$e_1=2$ and we may write
$\CB_{K/F}(\pi_1)=\rho_1^{}\times\rho_1^\a$ for some
(tamely ramified, integral) charac\-ter $\rho_1$ of $K^\times$.

Assume further that $\ell$ is a prime divisor of $q^2-1$ not dividing $q-1$,
that is,
$\ell$ is an odd~pri\-me~divisor of $q+1$.
Let $\mu$ be a character of $\boldsymbol{l}^\times$ of order $\ell$
and~set $\xi_2=\xi_1\mu$.
Since $\xi_2^{q}\neq\xi_2$, 
there is~a~cus\-pidal $\qlb$-representation $\s_2$ of $\GL_2(\kk)$ associated with 
$\xi_2$.
Since $\xi_2$ and $\xi_1$ are con\-gruent,
$\s_2$~and $\s_1$ are congruent (see for instance \cite{MSf} 2.6).
Let us inflate and extend $\s_2$ to a repre\-sen\-ta\-tion $\boldsymbol{\l}_2$ of
$F^\times\GL_2(\Oo_F)$ which is~con\-gruent to $\boldsymbol{\l}_1$,
then compactly induce $\boldsymbol{\l}_2$ to a representation $\pi_2$ of
$\GL_2(F)$.
This is an integral cuspidal representation of level $0$
{which is congruent to $\pi_1$.}

Since $\mu^{q}\neq\mu$,
we have $e_2=1$,
thus $\CB_{K/F}(\pi_1)$ is a cuspidal representation $\rho_2$
of $\GL_2(K)$.
Its reduction mod $\ell$ is an irreducible cuspidal $\flb$-representation 
of $\GL_2(K)$.
It is the unique generic component of $\r_\ell(\rho_1^{}\times\rho_1^\a)$.

\section{Cuspidal representations of split $p$-adic 
orthogonal groups with irreducible Galois parameter}
\label{app1}

\subsection{}

Let $F$ be a $p$-adic field with $p\neq2$,
and let $G$ be a split special orthogonal group over $F$,
that is, $G=\SO(Q)$ where $Q$ is a maximally isotropic quadratic form
over $F$.
Let $n$ be the dimension of $Q$.
In this section, we assume that $n\neq2$.
Let $m=\lfloor n/2\rfloor$ be the Witt index of $Q$.
With the notation of Paragraph \ref{quasisplitg},
we have $G=\SO_{2m+1}(F)$ if $n$ is odd,
$G=\SO_{2m}^1(F)$ if $n$ is even.
We will prove the following result. 

\begin{prop}
\label{propLS}
There exists a cuspidal representation of level $0$ of $G$
whose transfer to $\GL_N(F)$ is cuspidal.
\end{prop}

\subsection{}

In this paragraph, we refer to \cite{LS} \S2
(see p.~1090 in particular).
Let $V$ be the $n$-dimensional $F$-vector space on which $Q$ is defined.
Write
\begin{equation*}
V=V^{\rm an}\oplus V^{\rm iso}
\end{equation*}
where $V^{\rm an}$ is anisotropic (thus $\dim(V^{\rm an})\<1$)
and $V^{\rm iso}$ is a sum of $m$ hyperbolic planes. 

Let $q$ denote the cardinality of the residue field of $F$.
The anisotropic group $G^{\rm an}=\SO(V^{\rm an})$ has a unique
(up to conjugacy) maximal parahoric~sub\-group.
Its finite reductive quotient $\EuScript{G}^{\rm an}$
has neutral component the finite special orthogonal group 
\begin{equation*}
\SO_{a}(q)
\end{equation*}
with $a=\dim(V^{\rm an})$.

For any choice of integers $m_1,m_2\>0$ such that $m_1+m_2=m$,
there is a maximal paraho\-ric~sub\-group $J=J_{m_1,m_2}$
whose finite reductive quotient $\EuScript{G}=\EuScript{G}_{m_1,m_2}$
has neutral component
\begin{equation*}
\SO_{a+m_1, m_1}(q)\times\SO_{m_2, m_2}(q)
\end{equation*}
where $\SO_{u,v}(q)$ is the special orthogonal group over $\FF_{q}$
associated with a quadratic space of~di\-men\-sion $u+v$ and Witt index $v$. 
Choose $m_2=0$, so that $\EuScript{G}$ has neutral component 
$\SO_{m, m}(q)$ if $n=2m$,
and $\SO_{m+1, m}(q)$ if $n=2m+1$. 
In other words, $\EuScript{G}^\circ$ is split. 

\subsection{}

Let $\s$ be a self-dual cuspidal irreducible representation of $\GL_{2r}(q)$
and $s\in\FF_{q^{2r}}^\times$ be a parameter corresponding to $\s$.
In particular,
$s$ has degree $2r$ over $\FF_q$ and $s^{-1}=s^{q^r}$.
Its characteristic~po\-ly\-no\-mial $P(X)$ is thus irreducible,
of degree $2r$,
and self-dual (that is, reciprocal).

The parameter $s$ can be seen in the dual group
$\EuScript{G}^{\circ,*}\subseteq\GL_{2r}(q)$.
It then defines a Lusztig~se\-ries $\EuScript{E}(\EuScript{G}^{\circ},s)$.

\begin{lemm}
The Lusztig series $\EuScript{E}(\EuScript{G}^{\circ},s)$
contains a cuspidal representation.
\end{lemm}

\begin{proof}
If $m$ is odd, see \cite{LS} \S7.2 (p.~1098). 
Assume now that $m$ is even.
We follow \cite{LS}~\S7.3. Con\-sider the group with connected centre 
$\widetilde{\EuScript{G}}={\rm GSO}_{m}^\pm$
of which $\EuScript{G}^\circ$ is a subgroup.
The scalars~$1$ and $-1$ are not eigenvalues of $s$.
The centralizer of $s$ is thus connected
and the two Lusztig~se\-ries associated with $s$ are the same.
A cuspidal representation of $\EuScript{G}^\circ$ associated with $s$
is an~irredu\-ci\-ble component of the restriction to $\EuScript{G}^\circ$ of a
cuspidal representation of
$\widetilde{\EuScript{G}}$ associated with a~semi-simple~element
$\tilde{s}\in\widetilde{\EuScript{G}}^*$ lifting $s$.
To prove the lemma,
it thus suffices to prove that the Lusztig~se\-ries 
$\Ee(\widetilde{\EuScript{G}},\tilde{s})$~con\-tains a cuspidal representation.

The two groups $\widetilde{\EuScript{G}}$ and $\EuScript{G}^\circ$ act
naturally on the same space, 
thus $\tilde{s}$ and $s$ have the same~cha\-rac\-teristic polynomial $P(X)$.
It follows from \cite{LS} \S7.2 (p.~1098) that
$\Ee(\widetilde{\EuScript{G}},\tilde{s})$ contains a
cuspidal representation. 
\end{proof}

\subsection{}

Let $\tau$ be a cuspidal representation in the
Lusztig series $\EuScript{E}(\EuScript{G}^{\circ},s)$.
Let $\l$ be an irreducible~repre\-sen\-tation of $J$ whose restriction to
$J^\circ$ (the preimage of $\EuScript{G}^{\circ}$ in $J$)
is a direct sum of conjugates (under $J$) of the inflation~of 
$\tau$.~Let $\pi$ be the representation obtained by compactly inducing
$\l$ to $\G$. 
It is a cuspidal~ir\-re\-du\-ci\-ble~representation of level $0$ of $G$.

As $G$ is split,
it follows from Moeglin \cite{Moeglin41} that
the Langlands parameter $\h$ associated with~$\pi$ is described 
by the reducibility set ${\rm Red}(\pi)$ and the Jordan set ${\rm Jord}(\pi)$
(see for instance the~in\-tro\-duc\-tion~of \cite{LS}  for a definition).

In our situation,
it follows from \cite{LS} \S8 that the sets
${\rm Red}(\pi)$ and ${\rm Jord}(\pi)$
are equal and both reduced to a single element $(\rho,1)$,
where $\rho$ is a selfdual cuspidal representation of $\GL_{N}(F)$
(with $N=n-1$ if $n$ is odd and $N=n$ if $n$ is even),
which proves Proposition \ref{propLS}.

\begin{rema}
More precisely, $\rho$ has level $0$,
and is obtained by compactly inducing a representation of
$F^\times\GL_{N}(\Oo_F)$ which is trivial on $1+\Mat_N(\p_F)$
and whose restriction to $\GL_{N}(\Oo_F)$ is the inflation of $\s$.
\end{rema}

\section{Cuspidal representations of $\Sp_{2n}(\QQ_2)$
with irreducible Galois parameter} 
\label{app2}

\begin{center}
{\bf (by Guy HENNIART at ORSAY)}
\end{center}

\def\pP{\overline{\P}}
\def\uU{\overline{\U}}
\def\bB{\overline{\B}}
\def\vol{{\rm vol}}
\def\card{{\rm card}}

\subsection{}
\label{A.1}

Let $p$ be a prime number and $\F$ a finite extension of $\QQ_p$.
Let $\overline{\F}$ be an algebraic closure of~$\F$ and $\W_\F$ the Weil group of 
$\overline{\F}/\F$. 
Let $n$ be a positive integer, and $\pi$ a cuspidal (complex)~repre\-sentation 
of $\Sp_{2n}(\F)$.
Let $\sigma$ be the Galois parameter attached to $\pi$ by Arthur \cite{Arthur},
which one sees as an orthogonal representation of 
$\W_\F\times \SL_2(\CC)$, of dimension $2n+1$.
The following result is used in the main text,
in Section \ref{argumentfinal}. 

\begin{theo}
\label{thm:guy}
Assume that $\F=\QQ_2$, 
and take for $\pi$ the (unique) simple supercuspidal~re\-pre\-sentation of 
$\Sp_{2n}(\F)$.
Then $\sigma$ is an irreducible representation of $\W_\F$. 
\end{theo}

Here \emph{simple} is in the sense of Gross and Reeder \cite{GR}.
The point of
the result is that $\pi$ is~com\-pactly induced from a compact open subgroup of
$\Sp_{2n}(\F)$, as we describe below.
Indeed when $p=2$ there is at least one
irreducible orthogonal representation $\sigma$ of $\W_\F$ of dimension $2n+1$ 
\cite{BHdyadic},
only one if $\F=\QQ_2$, and by \cite{Arthur} it is the parameter of a
cuspidal representation $\pi$ of $\Sp_{2n}(\F)$,
but it is not clear \textit{a priori} that $\pi$ is compactly induced. 

Our method is inspired by work of Oi \cite{Oi2018}.
When $p$ is odd, Oi
determines the parameter $\sigma$ of a simple cuspidal representation $\pi$ of
$\Sp_{2n}(  \F)$. In  his case  $\sigma$ is  always reducible,  but a  number of
techniques and  results remain valid  when $p=2$, and, with  extra information
given by Adrian and Kaplan \cite{AK} when $\F=\QQ_2$, that is enough for us.
It is quite likely that one can describe $\sigma$ explicitly whenever $\pi$ is
simple cuspidal, not only when $p$ is odd or $\F=\QQ_2$.
Indeed many of~our
arguments work more generally, and until \ref{A.6} we make no special
assumption on $\F$, except that in \ref{A.3} we start assuming
that\footnote{Oi and the author (\cite{HenniartOi})
can now extend Theorem \ref{thm:guy} to any $2$-adic field $\F$.} 
$p=2$. 

\subsection{}
\label{A.2}

We now proceed.
We use customary notation, $\Oo_\F$ for  the ring of integers of $\F$, $\p_\F$
for the maximal ideal of $\Oo_\F$.
We fix a uniformizer $\varpi$ of $\F$, and
write $k$ for the residue field $\Oo_\F/\p_\F$ and $q$ for its cardinality.
We also fix
a non-trivial character $\psi$ of $k$. If ${\bf H}$ is an algebraic group over
$\F$, we usually put $\H={\bf H}(\F)$. 

We use the usual explicit model of ${\bf G}=\Sp_{2n}$,
see \cite{Oi2018} \S2.4,
so elements of $\G=\Sp_{2n}(\F)$ are sym\-plectic $2n\times 2n$
matrices.
By cuspidal representation of $\G$ we mean an irreducible
smooth
complex cuspidal  representation. We are interested  in \emph{simple} cuspidal
representations of  $\G$, in the sense  of Gross and Reeder  \cite{GR}. Let us
describe them. 

The choice in \cite{GR} of a root basis and an affine root basis determines an
Iwahori subgroup~$\I$~of $\G$,~with its first two congruence subgroups
$\I^{+}$ and $\I^{++}$.
The Iwahori subgroup $\I$ is the subgroup~of $\Sp_{2n}(\Oo_\F)$
made out of the matrices which are upper triangular modulo $\p_\F$, $\I^{+}$
is made out of the matrices  which are further upper unipotent modulo $\p_\F$,
and $\I^{++}$ is made~out of the matrices $(x_{i,j})$ in $\I^{+}$ with
$x_{i,i+1}\in\p_\F$ for $i=1,\dots, 2n-1$, and $x_{2n,1}\in\p_{\F}^2$.
The quotient $\I^{+}/\I^{++}$ is isomorphic to a product of $n+1$ copies of 
$k$, \textit{via} the surjective homomorphism 
\begin{equation*}
(x_{i,j}) \mapsto (x_{1, 2} \text{ mod } \p_\F,
\dots, x_{n,n+1} \text{ mod } \p_\F, x_{2n,1}/\varpi \text{ mod } \p_\F)
\end{equation*}
from $I^{+}$ to $k^{n+1}$. 

A character of $\I^{+}$ is \emph{simple} if it is trivial on $\I^{++}$, and is
the inflation of a character of~$k^{n+1}$ which is non-trivial on each factor
$k$. The normalizer in $\G$ of a simple character $\theta$ of $\I^{+}$ is
$\Z\I^{+}$, where $\Z$ is the centre of $\G$, and $\Z\I^{+}$ is also the
intertwining of $\theta$ in $\G$, so that any~ex\-ten\-sion of $\theta$ to
$\Z\I^{+}$ gives  by compact  induction to $\G$  a cuspidal  representation of
$\G$: see \cite{Oi2018} \textsection 2.4,~Pro\-po\-sition 2.6.
Note that when $p$ is $2$,
the centre $\Z$ of $\G$ is actually contained in $\I^{++}$.
The cuspidal representations of $\G$ thus obtained are the
\emph{simple cuspidal} representations of \cite{GR}. 

The normalizer of $\I^{+}$ in $\G$ is $\Z\I$, and $\I$ acts on 
$\I^{+}/\I^{++}$ \textit{via} $\I/\I^{+}$;
identifying $\I/\I^{+}$ with $k^{\times n}$
\textit{via}
\begin{equation*}
(x_{i,j}) \mapsto (x_{1,1} \text{ mod } \p_\F,\dots, x_{n,n} \text{ mod } \p_\F),
\end{equation*}
the conjugation action of $(\chi _1 ,\dots,\chi _n)\in k^{\times n}$
on $\I^{+}/\I^{++}$ 
(identified with $k^{n+1}$) sends the family
$(u_1 ,\dots,u_{n+1})\in k^{n+1}$ to
\begin{equation*}
\left(u_1^{\phantom{1}}\chi_1^{\phantom{1}}\chi_2^{-1},
u_2^{\phantom{1}} \chi_2^{\phantom{1}} \chi_3^{-1} ,\dots,
u_{n-1}^{\phantom{1}}\chi_{n-1}^{\phantom{1}}\chi_n^{-1},
u_n^{\phantom{1}}\chi_n^2 ,
u_{n+1}^{\phantom{1}}\chi_1^{-2} \right).
\end{equation*} 
In particular, when $p=2$, a given
simple character $\theta$ of $\I^{+}$ can always be conjugated in $\I$ to the
character
\begin{equation*}
\t(a) : (u_1 ,\dots,u_{n+1}) \mapsto \psi(u_1  +\dots  +u_n+au_{n+1})
\end{equation*} 
for some $a$ in $k^\times$, uniquely determined by $\theta$.
More precisely if $\theta$ sends
$(u_1,\dots,u_{n+1})$  to $\psi(a_1  u_1 +\dots  +a_nu_n+a_{n+1}u_{n+1})$ for  some
$a_i$'s in $k^\times$, then $a$ is equal to $(a_1 \cdots a_{n-1})^2 \cdot a_n \cdot
a_{n+1}$. 
Thus\footnote{The referee remarks that one needs to know the ``intertwining implies conjugacy'' result that says that if two
simple characters $\theta$ and $\theta'$ of $I^+$ intertwine in $G$, then they are actually conjugate. The arguments are the same as for proving
that   the   construction   of   simple  cuspidals   does   give   irreducible
representations. The reader can~con\-sult \cite{Beth}
for the more general cases of epipelagic representations.} when $p=2$ there are only $q-1$ isomorphism classes of simple cuspidal
representations of $\G$, whereas, by a similar analysis
(\cite{Oi2018} \S2.4),
there are $4(q-1)$ such classes when $p$ is odd. Note that when
$q=2$ all that is obvious since $k$ has only one non-trivial character. 

\subsection{}
\label{A.3}

The group $\Sp_{2n}$ is split, and its dual group is $\SO_{2n+1}(\CC)$.
To a cuspidal representation~$\pi$~of~$\G$,
Arthur attaches the conjugacy class of
a discrete parameter, that is (the conjugacy class of)~a con\-tinuous 
homomorphism from $\W_\F\times\SL_{2}(\CC)$ into $\SO_{2n+1}(\CC)$ which,
as a represen\-tation of~di\-men\-sion $2n+1$, is a direct sum of inequivalent
irreducible orthogonal representa\-tions~$\sigma  _1 ,\dots,\sigma _r$ with the
product $\det\sigma _1 \cdots \det\sigma _r$ trivial. 
What our theorem says is that when $\F=\QQ_2$ and $\pi$ is simple cuspidal,
then $r=1$ and $\sigma _1$ is trivial on $\SL_{2}(\CC)$, \textit{i.e.} is in
fact~a~re\-presentation of $\W_\F$.
Note that \cite{BHdyadic} shows that when $p$ is
odd,  there is  no irreducible  orthogonal  representation of  $\W_\F$ of  odd
dimension $>1$, contrary to the case $p=2$,
where \cite{BHdyadic} gives a~com\-plete classification. 

\textit{From now on we assume $p=2$.}
For $a$ in $k^\times$ let us denote
by  $\pi(a)$ the  isomorphism class  of the  representation of  $\G$ compactly
induced from the character $\theta (a)$ of  $\I^{+}$. We let $\phi (a)$ be the
parameter of  $\pi(a)$, $r(a)$ the  number of irreducible components  of $\phi
(a)$, and $\Pi(a)$ the $L$-packet of  $\pi(a)$, that is the set of isomorphism
classes of  tempered (in fact,  discrete series) representations of  $\G$ with
parameter $\phi (a)$; it is known  that $\Pi(a)$ has $2^{r(a)-1}$ elements, so
one of our goals is to show that $r(a)=1$. 
Let ${\bf  G}_{\rm ad}$  be the adjoint  group of ${\bf  G}$, and  $\iota$ the
quotient map from ${\bf G}$ to ${\bf G}_{\rm ad}$. 

\begin{lemm}
\label{lem:A1}
$\pi(a)$ is stable under the action of $\G_{\rm ad}$.
\end{lemm}

\begin{proof}
We follow the proof of \cite{Oi2018} Proposition 5.2.
As there, one gets a description of the quo\-tient $\G_{\rm ad}/\iota (\G)$.
It is isomorphic to
$\Hom(\F^\times,\mu_2)$, itself isomorphic, by Kummer theory, to
$\F^\times/\F^{\times 2}$.
More concretely if $T$ is the diagonal torus of
$\G$ made out of elements
\begin{equation*}
t(b)=(b, b,\dots, b,  b^{-1},\dots, b^{-1})
\end{equation*}
(with $n$ times $b$ and $n$ times $b^{-1}$),
then for any $b$ in $\overline{\F}$
with   $b^2$  in   $\F^\times$  the   image   of  $t(b)$   in  ${\bf   G}_{\rm
  ad}(\overline{\F})$ is actually in $\G_{\rm ad}$, and the set of such
$t(b)$'s covers $\G_{\rm ad}/\iota (\G)$. 
 
If  $b^2$ is  a unit  in $\F$  then $t(b)$  actually normalizes  $\I$ and  its
congruence subgroups, and sends $\theta (a)$ to the character given by
\begin{equation*}
(u_1,\dots,u_{n+1})\mapsto \psi(u_1 +\dots +u_{n-1}+b^2 u_n+(a/b^2 )u_{n+1}),
\end{equation*}
conjugate in $\I$ to $\theta (a)$.
If $b^2$ is the uniformizer $\varpi$,
$t(b)$ conjugates $\I$ to another Iwahori subgroup, but~if $s$ is the matrix
in $\G$ with four blocks of size $n$, first line $(0, I_n)$ and second line
$(-I_n, 0)$,  then $st(b)$ normalizes  $\I$ and its congruence  subgroups, and
sends $\theta (a)$ to the character given by
\begin{equation*}
(u_1 ,\dots,u_{n+1})\mapsto \psi(u_1 +\dots +u_{n-1}+au_n+u_{n+1})
\end{equation*}
(recall that $p=2$, so $-1=1$ in $k$),
which is conjugate  to $\theta (a)$. Since the stabilizer  in $\G_{\rm ad}$ of
$\pi(a)$ is a subgroup  containing all of $\iota (\G)$, it  follows that it is
all of $\G_{\rm ad}$. 
\end{proof}

An important point is the genericity of simple cuspidal representations.
We fix the same~Whit\-ta\-ker datum as Oi \cite{Oi2018}
\S 6.3(2) to
define genericity.
By \cite{Kaletha2013} Proposition 5.1, the
$\G_{\rm ad}$-orbit~of~$\pi(a)$ contains a single generic representation,
so by the previous lemma the representation $\pi(a)$ is~ge\-neric.
Reasoning as in \cite{Oi2018} Corollary 4.9 and Corollary 5.7,
we get: 

\begin{prop}
The parameter $\phi (a)$ is trivial on $\SL_{2}(\CC)$, every element of $\Pi(a)$
is cuspi\-dal, and among them only $\pi(a)$ is a simple cuspidal representation. 
\end{prop}

It only remains to prove that $r(a)=1$.

\subsection{}
\label{A.4}

Still following \cite{Oi2018} we prove:

\begin{prop}
$\Pi(a)$ does not contain any level $0$ cuspidal representation.
\end{prop}

\begin{proof}
By \cite{Shahidi90} Corollary 9.10,
all elements of $\Pi(a)$ have the same 
formal degree.
If ${\rm d}g$ is~a Haar measure on $\G/\Z$, then the formal degree of 
$\pi(a)$ is ${\rm d}g/\vol(\I^{+}/\Z,{\rm d}g)$
(by \cite{Oi2018} Lemma~5.10),
where\-as the
formal degree of a level $0$ cuspidal representation of $\G$ is strictly
smaller, by the~fol\-lowing reasoning inspired by \textit{loc.\! cit.}, Proposition
5.11.
A level $0$ cuspidal representation $\pi'$ of~$\G$ is compactly induced
from an irreducible representation $\rho$ of a maximal parahoric subgroup $\P$
of  $\G$, trivial  on  the pro-$p$  radical  $\P^+$ of  $\P$,  and coming  via
inflation from a cuspidal representa\-tion of the finite (connected here) 
reductive group $\pP=\P/\P^+$.
The formal degree of $\pi'$ is
\begin{equation*}
\frac{\dim(\rho)} {\vol(\P/\Z, {\rm d}g)} {\rm d}g.
\end{equation*}
One can assume that $\P$ contains $\I$ and $\I^{+}$
contains $\P^+$.
Since $p=2$,
the group $\P^+$ contains $\Z$, so what we have to prove
is that $\dim(\rho )<\card(\P/\I^{+})$.
But $\I^{+}/\P^+$ is the unipotent
radical $\overline{U}$ of the Borel subgroup $\overline{B}=\I/\P^+$ of $\pP$, and obviously
$\dim(\rho )^2$ is at most $\card(\pP)$, so it is enough to check
$\card(\pP)<\card(\pP/\overline{U})^2$ or $\card(\overline{U})^2<\card(\pP)$,
which is a
consequence of the existence of the big cell $\overline{B} w \overline{U}$
in the Bruhat decomposition for $\pP$. 
\end{proof}

\begin{rema}
It is highly plausible that for a cuspidal representation $\pi'$ of $\G$ which
is not of level $0$ and is not  a simple cuspidal either, the formal degree of
$\pi'$ is bigger  than the formal degree of $\pi(a)$.  But nothing explicit is
known about such $\pi'$. 
\end{rema}

\subsection{}\label{A.5}

Now we compute the character $\xi (a)$ of $\pi(a)$ at an affine generic 
element $g$ of $\I^{+}$, where $g$ gene\-ric means that modulo $\I^{++}$, $g$
gives an $(n+1)$-tuple $(u_1 ,\dots, u_{n+1})$ in $k^{n+1}$ with all 
coordina\-tes non-zero.
As in \cite{Oi2018} Lemma 2.5, we see that an element
$y$  conjugating $g$  into $\I^{+}$  belongs  to $\I$,  so that  by the  usual
formula for the character of compactly induced representations (see 
e.\!~g.
\textit{loc.\!  cit}. Theorem 3.2), the character $\xi (a)$ of $\pi(a)$ at $g$ is
the sum
\begin{equation*}
\sum\limits_{(\chi _1 ,\dots,\chi _n)\in k^{\times n}}
\psi\left(u_1^{\phantom{1}} \chi_1^{\phantom{1}}\chi_2^{-1} +u_2^{\phantom{1}}
\chi_2^{\phantom{1}} \chi _3^{-1} +\dots+u_{n-1}^{\phantom{1}}
\chi_{n-1}^{\phantom{1}}\chi _n^{-1} +u_n^{\phantom{1}}\chi _n^2 
+au_{n+1}^{\phantom{1}}\chi _1^{-2} \right), 
\end{equation*} 
which is a kind of Kloosterman sum,
the sum
\begin{equation*}
\sum\limits_{(\eta _1 ,\dots,\eta _n)\in k^{\times n}}
\psi\left(u_1  \eta  _1  +u_2   \eta  _2  +\dots+u_{n-1}\eta  _{n-1}+u_n\eta  _n^2
+au_n\eta _{n+1}\right)
\end{equation*} 
with $\eta _{n+1}$ given by $(\eta _1 \dots \eta _{n-1})^2
\eta _n\eta _{n+1}=1$. Noting that $\psi$  takes only the values $1$ and $-1$,
we conclude: 

\begin{prop}
The value of $\xi(a)$ at a generic  element $g$ of $\I^{+}$ is an odd integer
depending only on $g$ modulo $\I^{++}$. 
\end{prop}

\subsection{}\label{A.6}

Still following \cite{Oi2018} 5.3, we now show:

\begin{prop}\label{prop:A6}
$r(a)=1$ or $2$, and,
seen as a representation of $\W_\F$ of dimension $2n+1$,
$\phi (a)$ is either irreducible or the direct sum of a character $\omega$
with $\omega ^2 =1$ and an irreducible~(or\-tho\-gonal) representation with
determinant $\omega$. 
\end{prop}

\begin{proof}
Put $s=2^{r(a)-1}$ and enumerate the elements of $\pi(a)$ as $\pi_1
=\pi(a),\dots, \pi_s$, and  let $\xi _i$ be the character  of $\pi_i$. Let $g$
be a  generic element  of $\I^{+}$.  Choose $ \varepsilon  (i)=1$ or  $-1$ for
$i=1,\dots,s$. Exactly as in the proof of Claim in \textit{loc.\!  cit}., we
get that $\varepsilon (1)\xi _1 + \dots + \varepsilon (s)\xi _s$
does not vanish at $g$.
Using that the
characteristic polynomial of $g$ is irreducible of degree $2n$
(\textit{loc.\! cit}., Lemma 7.5, still valid when $p=2$),
the proofs of Theorem 5.1 and
Corollary 5.13 in \textit{loc.\! cit}. give the result. 
\end{proof}

\subsection{}\label{A.7}

To  get the  remaining  assertion that  $r(a)$  is  in fact  $1$,  we use  new
information given by Adrian and Kaplan \cite{AK}. Unfortunately that
information is only available presently when $\F=\QQ_2$, hence the
restriction  in  our main  result,  but  we  expect  that the  computation  in
\textit{loc.\! cit}. can be carried over to the general case. 
When $\F=\QQ_2$ there is only $a=1$, so we put $\pi=\pi(1)$.
In \cite{AK} Theorem 3.13,
the authors compute the Rankin-Selberg
$\gamma$-factor  $\gamma (\pi\times  \tau ,  \psi')$ (a  rational function  in
$2^s$ for  a complex parameter  $s$) for any  tame character $\tau$  of $\QQ_2
^\times$  with $\tau  ^2 =1$  and a  character $\psi'$  of $\QQ_2$  trivial on
$2\ZZ_2$ but not on $\ZZ_2$. They find 
\begin{equation}\label{eq:A1}
 \gamma  (\pi\times \tau  , \psi')=\tau  (2)2^{1/2-s}.
\end{equation}
On  the  other  hand  if  $\phi$  is   the  parameter  of  $\pi$,  seen  as  a
representation of $\W_\F$ of dimension  $2n+1$, and $\lambda$ the character of
$\W_\F$ corresponding to $\tau$ \textit{via} class field theory, then 
\begin{equation}\label{eq:A2}
\gamma  (\pi\times \tau  , \psi')=\gamma  (\phi  \otimes \lambda  , \psi')
\end{equation}
where the right-hand side is  the Deligne-Langlands factor\footnote{That is to
say, Arthur's correspondence is compatible with Rankin-Selberg
$\gamma$-factors. It can be proved by a lo\-cal-global argument. Detail
will appear in joint work with Oi \cite{HenniartOi}.}. 

That gives new information on $\phi$ which, we recall, is by Proposition
\ref{prop:A6} either irreducible~or the direct sum of a character $\omega$
with $\omega ^2 =1$ and an irreducible representation, say $\alpha$, with
\begin{equation*}
\omega \det\alpha=1.
\end{equation*}
But the factor $\gamma (\pi\times \tau , \psi')$ has no
zero nor pole, so is equal to the factor $\varepsilon (\pi\times \tau ,
\psi')=\varepsilon (\phi \otimes \lambda, \psi')$
which has the form
\begin{equation*}
u\cdot
2^{{\rm Art}(\phi \otimes \lambda  )-\dim(\phi \otimes \lambda ))(1/2-s)}
\end{equation*}
for
some non-zero complex  number $u$: the exact value of  the exponent comes from
the fact that $\psi'$ is trivial on  $2\ZZ_2$ but not on $\ZZ_2$. This implies
that ${\rm  Art}(\phi \otimes  \lambda )=2n+2$,  and taking  $\lambda$ trivial
yields ${\rm Art}(\phi )=2n+2$. 

Assume we are in the case where $\phi =\omega \oplus \alpha$.
Taking $\lambda=\omega$ gives a pole to 
$\gamma (\phi \otimes \omega , \psi)$
which  contradicts  \eqref{eq:A1}   if  $\omega$  is  tame   (that  is,  since
$\F=\QQ_2$,  unramified).  Thus $\omega$  is  wildly  ramified, so  its  Artin
exponent is at least $2$, and the  Artin exponent of $\alpha$ is at most $2n$.
That  implies   that  $\alpha$   is  tamely  ramified,   and  in   fact  ${\rm
  Art}(\alpha)=2n$,  ${\rm Art}(\omega  )=2$.  But then  $\det\alpha$ is  also
tamely ramified,  which contradicts $\omega \det\alpha=1$.  That contradiction
shows that $\phi$ is irreducible, as desired. 

\subsection{}\label{A.8}

One can describe $\phi$ explicitly.
By the main
result of \cite{BHdyadic} an orthogonal irreducible represen\-tation of 
$\W_{\QQ_2}$ 
is induced from an order $2$ wildly ramified character $\beta$ of $\W_K$, 
where $K$ is a~total\-ly~ramified extension of $\QQ_2$ degree $2n+1$.
Such an 
extension is unique up to isomorphism,~ge\-ne\-rated~by a uniformizer $z$ with 
$z^{2n+1}=2$.
Let $\widetilde\beta$ be the character of $K^\times$ corresponding
to $\beta$~\textit{via} class field theory. Since ${\rm Art}(\phi )=2n+2$, we have
${\rm Art}(\widetilde\beta)=2$, and moreover $\det(\phi )=1$ is~the restriction of
$\widetilde\beta$ to $\QQ_2 ^\times$ times the determinant of the representation
of $\W_{\QQ_2}$ induced from~the~tri\-vial~character of $\W_K$. That
determinant is an unramified quadratic character of $\W_K$, compu\-ted in
\cite{BushneFrohlichLNM} as the unramified character taking value at Frobenius
elements the Jacobi symbol of~$2$~mo\-dulo $2n+1$. That imposes
$\widetilde\beta(z)$, and with ${\rm Art}(\widetilde\beta)=2$ and
$\widetilde\beta(1+z)=-1$ it determines $\widetilde\beta$ hen\-ce $\beta$. 

\providecommand{\bysame}{\leavevmode ---\ }


\end{document}